\documentclass[leqno]{amsart}
\usepackage[margin=1in]{geometry}
\usepackage{amsmath, amsthm, amssymb}
\usepackage{hyperref}
\usepackage{multicol}
\hypersetup{colorlinks=true, pdfstartview=FitV, linkcolor=blue, citecolor=blue, urlcolor=blue}
\usepackage{times}
\usepackage{comment}
\usepackage[usenames,dvipsnames]{color}
\usepackage{bbm}
\usepackage{subcaption}
\usepackage{mathtools}
\usepackage{pdfpages}
\usepackage{bm}
\usepackage{verbatim}
\usepackage{mathabx}
\usepackage{geometry}
\usepackage{enumerate} 
\newtheorem{thm}{Theorem}[section]
\newtheorem{Prop}[thm]{Proposition}
\newtheorem{lem}[thm]{Lemma}
\newtheorem{Cor}[thm]{Corollary}
 \newtheorem{Thm}{Theorem}[section]
 \newtheorem{Rmk}[thm]{Remark}
 \newtheorem{Lem}[thm]{Lemma}

 \def\N {\mathbb{N}}
\def\R {\mathbb{R}}

\def\cH {\mathcal{H}}
\def\cG {\mathcal{G}}
\def\cL {\mathcal{L}}

\def\cS {\mathcal{S}}

\def\cT {\mathcal{T}}
\def\cW {\mathcal{W}}

\def\bx {{\textbf{x}}}
\def\aa {\mathrm{axi}}

\def\de {{\partial}}
\def\eps {{\varepsilon}}

\newcommand{\eul}{\mathrm{e}}

\newcommand{\Div}{\operatorname{div}}

\newcommand{\supp}{\operatorname{supp}}

\newcommand{\ba}{\begin{aligned}}
\newcommand{\ea}{\end{aligned}}
\newcommand{\be}{\begin{equation}}
\newcommand{\ee}{\end{equation}}
\newcommand{\bu}{\textbf{u}}

\newcommand{\wt}{\widetilde}

\newcommand{\red}{\textcolor{red}}
\newcommand{\Ps}{\Psi_{\mathrm{app}}}
\newcommand{\Psa}{\Psi_{2}}
\newcommand{\hE}{\Psi_{\mathrm{err}}}

\newcommand{\Ome}{\Omega_{\mathrm{app}}}
\newcommand{\et}{\eta_{\mathrm{app}}}
\newcommand{\fe}{f_{\mathrm{app}}}
\newcommand{\csi}{\xi_{\mathrm{app}}}

\numberwithin{equation}{section}

\begin{document}


\title[Strong ill-posedness in $L^{\infty}$ of 2D Boussinesq in vorticity form and 3D axisymmetric Euler]{Strong ill-posedness in $L^{\infty}$ of the 2D Boussinesq equations in vorticity form and application to the 3D axisymmetric Euler Equations}\author[R. Bianchini]{Roberta Bianchini}
\address{IAC, Consiglio Nazionale delle Ricerche, 00185 Rome, Italy.}
\email{roberta.bianchini@cnr.it}
\author[L.E. Hientzsch]{Lars Eric Hientzsch}
\address{Fakult\"at f\"ur Mathematik, Universit\"at Bielefeld, 33501 Bielefeld, Germany.}
\email{lhientzsch@math.uni-bielefeld.de}
\author[F. Iandoli]{Felice Iandoli}
\address{DEMACS, Università della Calabria, 87035 Rende, Italy.}
\email{felice.iandoli@unical.it}
\date\today
 
\maketitle 
\begin{abstract}
We prove the strong ill-posedness of the two-dimensional Boussinesq system in vorticity form in $L^{\infty}(\R^2)$ without boundary, building upon the method that Shikh Khalil \& Elgindi arXiv:2207.04556v1 developed for scalar equations. We provide examples of initial data with vorticity and density gradient of small $L^{\infty}(\R^2)$ size, for which the horizontal density gradient $\de_x \rho$ has a strong $L^{\infty}(\R^2)$-norm inflation in infinitesimal time, while the vorticity and the vertical density gradient remain bounded. Furthermore, exploiting the three-dimensional version of Elgindi's decomposition of the Biot-Savart law, we apply our method to the three-dimensional axisymmetric Euler equations with swirl and away from the vertical axis, showing that a large class of initial data with vorticity uniformly bounded and small in $L^{\infty}(\R^2)$ provides a solution whose gradient of the swirl has a strong $L^{\infty}(\R^2)$-norm inflation in infinitesimal time. The norm inflation is quantified from below by an explicit lower bound which depends on time, the size of the data and is valid for small times.
\end{abstract}

\section{Introduction and Result}
The two-dimensional Boussinesq equations in $(x,y) \in \R^2$, in vorticity $\omega(t,x,y): [0, \infty) \times \R^2 \to \R $ and density $\rho(t,x,y): [0, \infty) \times \R^2 \to  \R_+$ formulation, read as follows
\begin{align}\label{eq:2Dbouss}
\de_t \omega + \bu \cdot \nabla \omega & = \de_x \rho, \notag\\
\de_t \rho + \bu \cdot \nabla \rho&= 0, 
\end{align}
where $\bu (t,x,y)=(u_1 (t,x,y), u_2 (t,x,y)) \in \R^2$ is a divergence-free velocity field that can be expressed in terms of a stream-function as $\bu=\nabla^\perp \psi$, where $\nabla^\perp=(-\de_y, \de_x)^T$ and 
\be\label{eq:streamfun}
-\Delta \psi = \omega. 
\ee
This system is obtained from the two-dimensional incompressible Euler equations for non-homogeneous fluids, after applying the so-called \emph{Boussinesq approximation} (\cite{long1965}), which neglects density variations everywhere but in gravity terms, see \cite{lannes, bianchini} for further details. Besides being very commonly used for applications in oceanography and atmospheric sciences \cite{rieutord}, the 2D Boussinesq system \eqref{eq:2Dbouss} attracted the interest of the mathematical community for several decades and the question whether finite-time blow-up of solutions emanated from finite-energy, smooth initial data can occur is an outstanding open problem (see \cite{yudovich}). 
In the presence of a boundary, several impressive contributions have been very recently made in this direction, among which Chen \& Hou \cite{chen2022} showed the blow-up of self-similar solutions by means of computer-assisted proofs and physics-informed neural networks have been used to construct approximated self-similar blow-up profiles in \cite{wang2022}. The result in \cite{chen2022} also includes the finite-time blow-up for smooth data for the 3D axisymmetric Euler equations based on an analogy between such equations in potential vorticity $\frac{\omega_\theta}{r}$ formulation where $\bu^\aa=(u_r^\aa,u_z^\aa)^T$ and the 2D Boussinesq system in vorticity $\omega$ and density gradient $\nabla \rho$, which is very well described for instance in \cite[Section 1.2]{jeong}. We report here this analogy for completeness (for explanations on the 3D axisymmetric Euler equations see Section \ref{sec:euler}) :
\vspace{0.8cm}
\begin{multicols}{2}
\noindent
\be\label{eq:2Dbouss-origianal-gradient}
\ba
D_t \omega &= \de_x \rho, \\
D_t (\de_x\rho) &= - \de_x \bu \cdot \nabla_{x,y} \rho, \\
D_t (\de_y \rho)&=  - \de_y \bu \cdot \nabla_{x,y} \rho,\\
D_t:&=\de_t + \bu \cdot \nabla_{x,y},
\ea
\ee
\be\label{eq:3Daxisymm-intro}
\ba
\wt D_t \left({\omega_\theta^\aa}/{r}\right)  &= - {r^{-4}} \de_z ((r u_\theta^\aa)^2), \notag\\
\wt D_t \de_r (r u_\theta^\aa) &=- \de_r \bu^\aa \cdot \nabla_{r, z} (r u_\theta^\aa),  \notag\\
\wt D_t  \de_z(r  u_\theta^\aa) &=- \de_z \bu^\aa \cdot \nabla_{r, z} (r u_\theta^\aa), \\
\wt D_t :&= \de_t + \bu^\aa \cdot \nabla_{r, z}.
\ea\ee
\end{multicols}
Very far from being exhaustive, regarding the outstanding open question on finite-time blow-up of smooth solutions to the 3D Euler equations we only mention the first rigorous result for $C^{1, \alpha}$ solutions by Elgindi in \cite{tarek1}, the analogous result with boundary by Chen \& Hou \cite{chen1} and the aforementioned work \cite{chen2022}. 
Besides the blow-up scenario, another important question concerning the 2D Boussinesq equations is the small-scale formation and the growth of Sobolev norms in infinite time for large classes of initial data, see the work \cite{kiselev2022} and references therein for recent developments in this direction. 

The 2D Boussinesq system \eqref{eq:2Dbouss} is of significant interest in (mathematical) fluid dynamics for the aforementioned reasons. The so-called \emph{stratified} steady solutions $(\bar \omega_{\text{eq}}, \bar \rho_{\text{eq}})=(0, \bar \rho_{\text{eq}}(y))$, where $\bar \rho_{\text{eq}}(y)$ is a continuous function of the vertical coordinate $y$ only, are of particular relevance for applications in oceanography, see for instance \cite{DY1999, rieutord}. In that context, it is customary to consider a background stratification $\bar \rho_{\text{eq}}(y)$ that decreases with height $\bar \rho_{\text{eq}}'(y)<0$, so that the steady state is called \emph{stably stratified} since the linear operator associated with the linearization of Eqs. \eqref{eq:2Dbouss} around such steady state exhibits spectral stability (see for instance \cite{gallay, lannes}).
Let us consider the perturbations, modeled by system \eqref{eq:2Dbouss}, of the simplest stably stratified density profile, namely $\bar \rho_{\text{eq}}(y) = \bar \rho_c - y$, where $\bar \rho_c >0$ is a constant averaged density. Being \eqref{eq:2Dbouss} invariant under translation of $\rho$, we consider for convenience $\bar \rho_{\text{eq}}(y) = -y$. We then introduce the perturbed variables $\wt \omega$ and $\wt \rho=  y+ \rho $ and write the equations satisfied by them below:
\begin{align}\label{eq:2Dbouss-stable}
\de_t \wt\omega + \wt \bu \cdot  \nabla \wt \omega & = \de_x \wt \rho, \notag\\
\de_t \wt \rho + \wt \bu \cdot  \nabla \wt\rho&= \wt u_2,\\
\wt \bu &=(\wt u_1, \wt u_2)= (-\Delta)^{-1} \nabla^\perp \wt \omega.\notag
\end{align}
Within the framework of a spectrally stable regime (under the assumption $\bar \rho_{\text{eq}}'(y)<0$), one might expect solutions to \eqref{eq:2Dbouss-stable} to exhibit better behavior than those of the original 2D Boussinesq equations \eqref{eq:2Dbouss}. However, even under such a spectrally stable situation, the question of finite-time blow-up versus global regularity for smooth solutions remains open. To our knowledge, the best available result in this context is due to Elgindi \& Widmayer \cite{klaus1}, which provides existence of Sobolev solutions to system \eqref{eq:2Dbouss-stable} for initial data of size $\eps$ up to times of order $t \sim \eps^{-4/3}.$
If solutions to \eqref{eq:2Dbouss} around a stably stratified state with a background stratification $\bar \rho_{\mathrm{eq}}(y)=-y$ and the Couette flow $\bar \bu_{\mathrm{eq}}=(y,0)^T$, rather than the trivial flow above, are considered, then $\varepsilon$-perturbations of the velocity and density have been proved to decay up to times of order $t \sim \varepsilon^{-2}$. However, it is important to note that, despite the decay of $\varepsilon$-perturbations, the vorticity and density gradient display an algebraic growth in time, leading to interesting long-term behavior (see \cite{bianchini1, bianchini2}).
We will show that the spectral stability of the Boussinesq system around stable stratification \eqref{eq:2Dbouss-stable} is insufficient to rule out norm-inflation at $L^\infty$-regularity for the density gradient $\nabla\rho$ and the vorticity $\omega$. More precisely, we prove that the 2D Boussinesq equations around the \emph{stably} stratified steady state $\bar \rho_{\text{eq}}(y) = - y$ in vorticity form are \emph{strongly ill-posed} in $L^{\infty}(\R^2)$. Moreover, we infer the same ill-posedness result for the original 2D Boussinesq equations \eqref{eq:2Dbouss}. Before describing the results of this paper, we write system \eqref{eq:2Dbouss} in a more convenient form. 
Dropping the \emph{tilde} for lightening the notation and taking the gradient $\nabla=(\de_x, \de_y)^T$ of the equation for $\rho$ in \eqref{eq:2Dbouss-stable}, we write the system as follows:
\be\label{eq:2Dbouss-grad}
\ba\de_t \omega+\bu \cdot \nabla \omega &= \de_x \rho, \\
\de_t (\de_x\rho) + \bu \cdot \nabla (\de_x \rho) &= \de_x u_2 - \de_x u_1 \de_x \rho - \de_x u_2 \de_y \rho, \\
\de_t (\de_y \rho) + \bu \cdot \nabla (\de_y \rho) &= \de_y u_2 - \de_y u_1 \de_x \rho - \de_y u_2 \de_y \rho,\\
\bu&= (-\Delta)^{-1}\nabla^\perp \omega = (-\Delta)^{-1}(-\de_y, \de_x)^T \omega.
\ea
\ee
As in \cite{tarek3}, our proof of strong ill-posedness relies on an approximation of \eqref{eq:2Dbouss-grad}, under a certain scaling and in terms of a small parameter $0<\alpha \ll 1$ (see \eqref{eq:polar-coord}), from which one can extract a Leading Order Model (in terms of the approximation parameter $\alpha$) that drives the ill-posedness mechanism. Interestingly, our Leading Order Model (see \eqref{eq:leading1}) discards the ``stabilizing'' contribution $u_2$ in the right-hand side of the equation for $\wt \rho$ in \eqref{eq:2Dbouss-stable}, being indeed $u_2$ a lower order term in $\alpha$ (see Remark \ref{rmk:originalB}). 

The key term leading to the $L^{\infty}$ norm inflation comes from the nonlinearity and it is precisely $-\de_x u_1 \de_x \rho$ in the right-hand side of system \eqref{eq:2Dbouss-grad}. Recalling that $\de_x u_1$ can be expressed as $-\de_{xy} (-\Delta)^{-1} \omega=R_1 R_2\omega$, where $(R_1, R_2)$ is the two-dimensional Riesz Transform (see for instance \cite{grafakos}), our work confirms that the unboundedness of the Riesz Transform as an operator on $L^\infty$ is a critical factor leading to strong ill-posedness. This fact was already exploited by Elgindi \& Masmoudi \cite{tarek3} to prove the \emph{mild} ill-posedness in $L^{\infty}$ of the stably stratified Boussinesq system \eqref{eq:2Dbouss-stable}
in vorticity form, meaning that there exist initial data $\|f_0\|_{L^{\infty}} \le \delta \ll 1,$ such that $\| f(t) \|_{L^{\infty}} > c$ for some fixed constant $c>0$ and some small time $t_\delta$ (see \cite{tarek3} for a more detailed definition). 
It is crucial to emphasize that, although the unboundedness of the Riesz transform in $L^\infty$ also plays a crucial role in this article, the \textbf{nonlinear} mechanism leading to our \textbf{strong} ill-posedness significantly differs from the \textbf{linear} mechanism leading to the \textbf{mild} ill-posedness by Elgindi \& Masmoudi. The latter approach relies on a rigorous semigroup expansion $\exp({R_1 t}) \sim \text{I}+R_1 t$ for small times and does not extend beyond mild ill-posedness due to the perturbative treatment of the nonlinear term (requiring uniform boundedness).
In particular, the semigroup $\exp({R_1 t})$ is associated with the linear part of the system \eqref{eq:2Dbouss-grad}, i.e.,
\begin{align}
\de_t \omega &= \de_x \rho, \notag\\
\de_t (\de_x \rho) &= \de_x u_2 = \de_{xx} (-\Delta)^{-1} \omega = R_1^2 \omega, \label{eq:2Dbouss-grad-lin}
\end{align}
which forms a wave-like system. The oscillations are precisely given by the symbol of $\pm R_1$, namely $\pm \widehat R_1 = \pm i \frac{\xi_1}{|\xi|}$, where $\xi=(\xi_1, \xi_2)$ is the Fourier frequency. As previously mentioned, the small-time expansion of this semigroup roughly leads to the mechanism causing mild ill-posedness in \cite{tarek3}.
In contrast, our proof, building upon the method in \cite{tarek2}, is rooted in the analysis of a \textbf{\emph{nonlinear toy model}} for \eqref{eq:2Dbouss-grad}, \textbf{\emph{considering all of the leading order}} (linear and nonlinear) \textbf{\emph{effects}}. This way, it  \emph{\textbf{involves a crucial interplay between linear and nonlinear terms for norm inflation}}.
It is essential to note that not all linear terms in \eqref{eq:2Dbouss-grad} are leading order. In other words, our toy model, presented below in \eqref{eq:leading1}, does not contain all terms of the linear system \eqref{eq:2Dbouss-grad-lin}. Specifically, $\de_x u_2$ in the right-hand side of the second equation of \eqref{eq:2Dbouss-grad} is neglected as it is considered lower order. This observation demonstrates that the instability mechanism is driven by both linear and nonlinear leading order terms, making it substantially different from the norm growth provided by \cite{tarek3}.
Exploiting the structure of the nonlinearity of system \eqref{eq:2Dbouss-stable}, we prove that there are initial data data  $\|\omega_0\|_{L^{\infty}}+\|\nabla \rho_0\|_{L^{\infty}} \leq \delta \ll 1,$ with $\delta>0$ arbitrarily small such that $\sup_{t \in [0, \delta]} \|\omega (t)\|_{L^{\infty}}+\|\nabla \rho (t) \|_{L^{\infty}}  \ge \delta^{-1}$ (strong ill-posedness). 
To the best of our knowledge, our results on the strong ill-posedness of the 2D Boussinesq equations in Theorem \ref{thm:main} and Theorem \ref{thm:main2} are original. 
\medskip

Our second result concerns the 3D  axisymmetric Euler equations.
The strong ill-posedness of the 3D axisymmetric Euler equations was first proven by Bourgain \& Li \cite{bourgain2015}. However, we emphasize that our Theorem \ref{thm:main3} significantly differs from \cite{bourgain2015}, where the ill-posedness arises in the 3D axisymmetric Euler equations \emph{without} swirl. In contrast, our study focuses on the 3D axisymmetric Euler equations \emph{with swirl} and provides $L^{\infty}$ norm inflation of the \emph{gradient of the swirl} at infini\-tesimal times.

\medskip

Finally, an important point is that our results hold in the whole $\R^2$ space: this is in sharp contrast to the blow-up results by Chen and Hou \cite{chen1}, in which the boundary plays a crucial role.
\medskip

On a technical level, is important to underline that several key steps of our proof, such as the $\alpha$-scaling of the radial coordinate $r=\sqrt{x^2+y^2}=R^{1/\alpha}$ and the decomposition of the Biot-Savart law (Theorem \ref{thm:tarek}), follow Elgindi's ideas \cite{tarek1, tarek2}. A detailed comparison with the inspiring works on ill-posedness by Elgindi et al \cite{tarek2, tarek3} and finite-time blow-up of solutions to the 2D Boussinesq system with boundary by Chen \& Hou \cite{chen2022} is provided below, as well as further comments on the ill-posedness of the 3D Euler equations without swirl by Bourgain \& Li \cite{bourgain2015}.

We prove our results on the 2D Boussinesq equations, while the application to the 3D Euler equations is confined to Section \ref{sec:euler}.
Hereafter, we  work in $\R^2$ in polar coordinates, with the radial scaling as in \eqref{eq:scaling}.
We further define the new variables depending on $(R, \beta) \in [0, \infty)\times [0, 2\pi]$ as follows 
\be\label{eq:new-var}\ba
\Omega=\Omega (R, \beta):=\omega(x,y), \qquad P=P(R, \beta):= \rho(x,y). 
\ea\ee

Before going further, we introduce our notation and convention below.
\subsection*{Notation and convention}\label{sec:notation}
\begin{itemize}
\item For the 2D Boussinesq equations, namely up to Section \ref{sec:last}, we work in the polar coordinates $(R,\beta)$ where $R=r^{\alpha}$ with $r=\sqrt{x^2+y^2}$, $\alpha>0$ small and $\beta=\arctan(y/x)$. In Section \ref{sec:euler}, we introduce new coordinates tailored to the 3D setting with axial symmetry which are used in the respective section of the paper only.
\item For $1\leq p \leq \infty$, the $L^p$ space with respect to the variables $(R,\beta)$, namely the space $L^p([0,\infty)\times [0,2\pi])$, is defined with respect to the measure $\mathrm{d}R\mathrm{d}\beta$ - unlike the common definition of $L^2(\R^2)$ in polar coordinates. If not stated otherwise, $L^p$ indicates  $L^p([0,\infty)\times [0,2\pi])$.  
\item The $L^2$-based Sobolev space $H^k=H^k([0,\infty)\times [0,2\pi])$ with $k\in \N_{0}$ and norm
\begin{equation*}
    \|f\|_{H^k}=\sum_{0\leq|\gamma|\leq k}\|\partial^{\gamma}u\|_{L^2},
\end{equation*}
where $\gamma$ indicates the $2D$-multi-index. We shall use the notation $\partial^n$ to indicate $\partial^\gamma$ with $|\gamma|=n$.
\item Following \cite{tarek2}, we introduce the weighted (non)-homogeneous Sobolev spaces $\dot{\cH}^k([0,\infty)\times [0,2\pi])$ and $\cH^k([0,\infty)\times [0,2\pi])$ respectively with norms
\begin{equation}\label{def:Hk}
\|f\|_{\dot{\cH}^m}=\sum_{i=0}^{m}\left(\|\partial_R^{i}\partial_\beta^{m-i}f\|_{L^2}+\|R^i\partial_R^{i}\partial_\beta^{m-i}f\|_{L^2}\right), \qquad \|f\|_{\cH^k}=\sum_{m=0}^k\|f\|_{\dot{\cH}^m}.
\end{equation}
Similarly, we denote for $k\in \N_{0}$ the spaces $\dot{\cW}^{m,\infty}$ and $\cW^{k,\infty}$ respectively with norms
\begin{equation}\label{def:Wk}      \|f\|_{\dot{\cW}^{m, \infty}}=\sum_{i=0}^{m}\left(\|\partial_R^{i}\partial_\beta^{m-i}f\|_{L^\infty}+\|R^i\partial_R^{i}\partial_\beta^{m-i}f\|_{L^\infty}\right), \qquad \|f\|_{\cW^{k, \infty}}=\sum_{m=0}^k\|f\|_{\dot{\cW}^m}.
\end{equation}
\item We use the notation $L_R^2, H_R^1, H_R^k$ to denote $L^2([0, \infty)\times [0,2\pi]), \cH^1([0, \infty)\times [0,2\pi]), \cH^k([0, \infty)\times [0,2\pi]), L^\infty([0, \infty)\times [0,2\pi])$ norms of radial functions.
\item We adopt the notation 
\be\label{eq:size-initial}
C_k=C_k ( \|\Omega_0\|_{\cH^k}, \|P_0\|_{\cH^k})
\ee
to denote a general constant, which may change from line to line, depending on the $\cH^k([0,\infty]\times [0, 2\pi])$ (in $(R, \beta)$ coordinates) norm of the initial data $(\omega_0(x,y), \rho_0(x,y))=(\Omega_0(R, \beta), P_0(R, \beta))$ and on the space dimension. 
\item The symbol $\lesssim $ (resp. $\gtrsim $) denotes $\le M $ (resp. $ \ge M $) for some constant $M>0$ and independent of the small parameter $\alpha$. We shall write $A\sim B$ when we have $A\lesssim B$ and $B\lesssim A$.
\item We introduce the notation $\mathbb{P}_n, \, n=0,1,2$ to denote the projection onto $n$-modes: for any function $f(\beta)$ with $\beta \in [0, 2\pi]$, 
\begin{align}\label{def:projector}
    \mathbb{P}_n (f):=\frac{1}{\pi} \left(\sin(n\beta) \int_{0}^{2\pi} f(\beta) \sin(n\beta) \, d\beta + \cos(n\beta) \int_{0}^{2\pi} f(\beta) \cos(n\beta) \, d\beta\right). 
\end{align}

\end{itemize}
\subsection*{Main Results}

Our results read as follows. The first theorem quantifies the norm inflation of the horizontal density gradient $\de_x\rho$ in terms of the small parameter $\alpha$.

\begin{Thm}[$L^{\infty}$ strong ill-posedness for the stably stratified Boussinesq system \eqref{eq:2Dbouss-stable}]
\label{thm:main}
There exists $0<\alpha_0 \ll 1$
such that for any $0<\alpha<\alpha_0$ and any $\delta>0$
there exist initial data $$(\omega_0^{\alpha, \delta}(x,y), \rho_0^{\alpha, \delta}(x, y))=(\Omega_0^{\alpha, \delta}(R, \beta), P_0^{\alpha, \delta}(R, \beta))$$ of the following form 
\begin{align}
    \Omega_0^{\alpha, \delta} (R, \beta) = \bar g_0^{\alpha, \delta} (R) \sin (2\beta), \qquad P_0^{\alpha, \delta} (R, \beta)= R^{1/\alpha} \bar \eta_0^{\alpha, \delta} (R) \cos \beta,
\end{align}
where $\bar g_0^{\alpha, \delta} (R),\bar \eta_0^{\alpha, \delta} (R) \in C_c^\infty([5/8, \infty))$ 
and 
\be\ba
\|(\omega_0^{\alpha, \delta}, \de_x \rho_0^{\alpha, \delta}, \de_y \rho_0^{\alpha, \delta})\|_{L^{\infty}(\R^2)} &=  \delta,\\
\ea \ee
such that the unique solution $\omega^{\alpha, \delta} (t,x,y), \rho^{\alpha, \delta}(t, x, y)$ to the Cauchy problem associated with the stably stratified Boussinesq system \eqref{eq:2Dbouss-stable} satisfies
\be\ba
\sup\limits_{t \in [0, T^*(\alpha)]} \|\de_x \rho ^{\alpha, \delta}(t) \|_{L^\infty(\R^2)}  &\ge \frac 12  \|\de_x \rho_0^{\alpha, \delta}\|_{L^\infty (\R^2)} \left(1+{\log |\log \alpha|}\right)^{\frac{1}{c_2}},
\ea\ee
where $T^*(\alpha) = C{\alpha}\log |\log (\alpha)|$ and $C, c_2>0$ are independent of $\alpha$.  
\end{Thm}


\begin{Rmk}[The 2D Boussinesq system without background stratification]\label{rmk:originalB}
Note that the above result can be reformulated to establish the strong ill-posedness of the original Boussinesq system \eqref{eq:2Dbouss}. 
More precisely, the ill-posedness in the stably stratified setting \eqref{eq:2Dbouss-stable} can be viewed as a particular case of the ill-posedness of the original Boussinesq equations \eqref{eq:2Dbouss}. Setting the notation for the stable stratification $\bar \rho_{\text{eq}}=-y$, this is achieved by taking the initial density $\rho_0^{\alpha, \delta}$ of the original Boussinesq system in the vicinity of the stable stratification itself, namely $\| \rho_0^{\alpha, \delta}- \bar \rho_{\mathrm{eq}}\|_{L^\infty(\R^2)} \le \delta.$
It is less evident a priori that exactly the same dynamics - the same leading-order model in \eqref{eq:leading1} - can be employed to prove the ill-posedness of the original Boussinesq system \eqref{eq:2Dbouss} for small initial density $\| \rho_0^{\alpha, \delta}\|_{L^\infty(\R^2)} \le \delta$ (rather than near a stable stratification). This is the content of the next result, and its proof involves just a minor modification of the proof of Theorem \ref{thm:main}.
\end{Rmk}
\begin{Thm}[$L^{\infty}$ strong ill-posedness for the original Boussinesq system \eqref{eq:2Dbouss}]\label{thm:main2}
Theorem \ref{thm:main} holds true with \eqref{eq:2Dbouss-stable} replaced by \eqref{eq:2Dbouss}.
\end{Thm}
Since it requires introducing a different notation and it is settled in $\R^3$, our result on the $L^{\infty}$ ill-posedness of the 3D axisymmetric Euler equations is stated and proved in Section \ref{sec:euler}. Here we provide only a very rough idea of its content.
\begin{Thm}\label{thm:main3}
For any $\delta>0$, there exists $C_c^\infty (\R^3) $  initial data with size $$\|(\omega_\theta^\aa, \de_r (r u_\theta^\aa)^2, \de_z (ru_\theta^\aa)^2)_{t=0}\|_{L^\infty (\R^2)} = \delta,$$ and a time $0<t<\delta$ such that $$\left\|\frac{(u_\theta^\aa)^2}{r}+\frac12{\de_r (u_\theta^\aa)^2}\right\|_{L^\infty([0,t])} \ge C \delta^{-1},$$
where $u_\theta^\aa = u_\theta^\aa (t, r, z)$ is the swirl component of the three-dimensional velocity $\bu^\aa$ and $\omega_\theta^\aa=\omega_\theta^\aa (t, r, z)$ is the angular vorticity of the solution to the 3D axisymmetric Euler equations with swirl \eqref{eq:3Daxisymm-intro}.
\end{Thm}
\subsection*{Comparison with the recent literature on ill-posedness for the Boussinesq equations and related models}
The first result of \emph{mild} ill-posedness in $L^{\infty}$ for the 2D Boussinesq equation in vorticity form around a stably stratified steady state with $\bar \rho_{\mathrm{eq}}=-y$ as in \eqref{eq:2Dbouss-stable} is due to Elgindi \& Masmoudi \cite{tarek3}. In \cite{tarek3}, the authors prove that for smooth initial data with vorticity and density gradient of size $\alpha$, the solution to \eqref{eq:2Dbouss-stable} satisfies $\|\omega (t)\|_{L^\infty} > C$ with $C>0$ a constant value independent of $\alpha$, for some $t < \alpha\ll 1$. In this paper we prove that the $L^\infty$ ill-posedness of \eqref{eq:2Dbouss-stable} is actually \emph{strong}, namely we provide infinitesimal initial vorticity and density gradient such that $\|\de_x \rho (t)\|_{L^\infty(\R^2)} > \log |\log \alpha|^\mu$,  $\mu\geq 1$, at some infinitesimal time is arbitrarily large. Providing a \emph{strong} ill-posedness result, the method of this paper differs substantially from the one of \cite{tarek3}. Our strategy is rather inspired by the work of Elgindi \& Shikh Khalil \cite{tarek2}, where a result of strong ill-posedness for \emph{scalar} nonlinear transport equations with a Riesz-type linear operator in two space dimensions is provided. More precisely, we use the same scaling of the polar coordinates, leading to the crucial Elgindi's decomposition of the Biot-Savart law \cite{tarek1}. As in \cite{tarek2}, we derive a Leading Order Model (see \eqref{eq:leading1}) yielding the strong ill-posedness result for a suitable choice of the initial data. We point out, however, that there are several key differences between this paper and \cite{tarek2}. In particular, while \cite{tarek2} deals with scalar equations, our Leading Order Model \eqref{eq:leading1} is a system of coupled equations. 
Moreover, while \cite{tarek2} proves norm inflation of $\omega(t)$ in $L^\infty(\R^2)$, our choice of the initial data (explicit examples are given in Section \ref{sec:initialdata}) demonstrates the norm inflation of the \text{(horizontal) density gradient} $\de_x \rho(t)$ in $L^\infty (\R^2)$, while $\omega(t)$ remains bounded in $L^\infty(\R^2)$.

Having already highlighted the substantial differences, let us now establish a connection between our Leading Order Model \eqref{eq:leading1} and the one derived in \cite{tarek2}. To this end, we sketch the main ideas, while a proper derivation of \eqref{eq:leading1} is provided in Section \ref{sec:lom}. The starting point is the use of the $\alpha$-approximation introduced by Elgindi \cite{tarek1} in the context of $C^{1,\alpha}$ velocity (i.e., $C^\alpha$ vorticity and density gradient). In particular, considering the polar coordinates 
\begin{equation}\label{eq:polar-coord}
    r = \sqrt{x^2 + y^2}, \qquad \beta = \arctan\left(\frac{y}{x}\right),
\end{equation}
we introduce the radial scaling
\begin{equation}\label{eq:scaling}
    R = r^\alpha, \quad 0 < \alpha \le \alpha_0 \ll 1,
\end{equation}
where $\alpha \ll 1$ plays the role of the small parameter of the approximation and is small enough. Writing the spatial derivatives $\partial_x$ (resp. $\partial_y$) in terms of the new radial coordinates $(R, \beta)$ leads to a scale separation between the angular part $- R^{-\frac{1}{\alpha}} (\sin \beta) \partial_\beta $ and the radial part $\alpha R^{-\frac{1}{\alpha} + 1} (\cos \beta) \partial_R$, where note that the latter term is proportional to $\alpha$. This allows expanding the equations \eqref{eq:2Dbouss-grad} in terms of different orders in $\alpha$. Keeping only the dominant terms in $\alpha$ provides our Leading Order Model \eqref{eq:leading1}. In this setting, the leading-order term in $\alpha$ of the transport $\mathbf u \cdot \nabla_{x, y} = (2 \Psi) \partial_\beta $ (see system \eqref{eq:model-new-coord} for more details), where $\Psi = \Psi(R, \beta) = r^{-2} \psi(x,y)$ as in \eqref{eq:psi-polar} below, and $\psi$ is the stream function as in \eqref{eq:streamfun}. In particular, as shown in Lemma \ref{lem:key}, the approximated (or leading order in $\alpha$) equation for $\omega$ in \eqref{eq:2Dbouss-grad} can be turned into the equation
\begin{equation}
    \partial_t g + \frac{\cL(g)}{2\alpha} \sin (2\beta) \partial_\beta g = 0,
\end{equation}
where $\cL(g) = \cL(g)(R)$ is a radial linear operator (defined in \eqref{def:L}). Interestingly, this equation is obtained - up to a change of variables - from the leading-order model in \cite{tarek2} for the Riesz Transform Problem
\begin{equation}\label{eq:riesz-transform-problem}
    \partial_t \omega + \bu \cdot \nabla \omega = R (\omega),
\end{equation}
where $R (\omega)$ is any Riesz Transform. Such a model allows in \cite{tarek2} to determine the precise point-wise growth of solutions for short time. In \cite{tarek2}, the solution to \eqref{eq:riesz-transform-problem} is $\omega (t) = g (t) + \frac{1}{2\alpha } \int_0^t \cL (\omega (\tau)) \, d\tau$ and displays a time growth from which strong ill-posedness of \cite{tarek2} is deduced. One of the important points is that, with this approach, the authors in \cite{tarek2} showed that while the \emph{linear} equation 
\begin{equation}
    \partial_t \omega = R(\omega)
\end{equation}
displays linear time growth, the full nonlinear equation \eqref{eq:riesz-transform-problem} has sub-linear time growth, meaning in a way that the transport term counteracts the linear growth. In our Leading Order Model \eqref{eq:leading1}, we recover again the same equation for $g(t)$ upon writing $\omega$ in \eqref{eq:2Dbouss-grad} as
\begin{equation}
    \omega = g(t) + \int_0^t \partial_x \rho (\tau) \, d\tau.
\end{equation}
We will see that, in contrast to \cite{tarek2}, in our case $\omega$ remains uniformly bounded in $\alpha$, while the time growth leading to strong ill-posedness is due to the density gradient.


\subsection*{Comparison with the recent literature on ill-posedness for the 3D incompressible Euler equations}
Concerning the 3D incompressible Euler equations, the $L^{\infty}$-ill-posedness has been recently proved by Bourgain \& Li \cite{bourgain2015}. While Theorem \ref{thm:main3} of the present paper for the system considered in vorticity and gradient of the swirl also constitutes a result of $L^{\infty}$ ill-posedness of the 3D axisymmetric Euler equations, the are considerable differences in the approaches. More specifically, the mechanisms that generate the ill-posedness (there is no swirl in \cite{bourgain2015}, while the swirl generates norm inflation in Theorem \ref{thm:main3}), the general form of the admissible initial data and the explicit quantification of the norm inflation by means of a lower bound - point-wise in time - of the $L^\infty$-norm of the swirl (see Theorem \ref{thm:main3}). To the best of our knowledge, Theorem \ref{thm:main3} constitutes the first ill-posedness result achieved by norm inflation of the swirl component. For a more detailed comparison, we refer the reader to Remark \ref{rmk:comparison}.
Regarding recent results on the strong ill-posedness for hydrodynamic equations, we mention for completeness the works \cite{jeong2017, jeong3} and \cite{bourgain, cordoba1, cordoba2}, which are based on different strategies. See also \cite{mazzucato, cordoba3} for results about loss of regularity. 
Finally, we highlight the results in \cite{kim2} regarding the $W^{1,\infty}$ ill-posedness of the SQG equation. As suggested by one of the referees, our approach could provide a \emph{quantitative} version of it.
\subsection*{Comparison with the recent literature on finite-time blow-up}
For $C^{1,\alpha}$ initial data, the local well-posedness of the solutions to \eqref{eq:2Dbouss} is given in \cite{chae1999}, together with a Beale-Kato-Majda-type criterion for blow-up. Next, finite-time blow-up of the local smooth solutions in domains with corners and data $(\nabla \rho (0, \cdot), \omega (0, \cdot)) \in {\mathring {C}}^{0, \alpha}$ was proved by Elgindi \& Jeong in \cite{jeong} for the 2D Boussinesq equations \eqref{eq:2Dbouss-grad} and in \cite{jeong2} for the 3D axisymmetric Euler equations. Furthermore, the finite-time blow-up of finite-energy $C^{1, \alpha}$ solutions to the 2D Boussinesq and 3D axisymmetric Euler equations in a domain with boundary (but without any corner) is due to Chen \& Hou \cite{chen1}. Since \cite{chen1} has been, together with \cite{tarek2}, a source of inspiration for our result, we comment further on similarities and differences. In fact, our strategy and the one of \cite{chen1} share some similarity - even though \cite{chen1} constitutes a result of finite-time blow-up in a regularity class where the well-posedness of \eqref{eq:2Dbouss} is well established, while we prove the strong ill-posedness of \eqref{eq:2Dbouss} and \eqref{eq:2Dbouss-stable} in the low regularity setting of $L^{\infty}$ vorticity and density gradient. First, our Theorem \ref{thm:main} yields not only the strong ill-posedness of the original 2D Boussinesq equations \eqref{eq:2Dbouss-grad}, but also of the 2D Boussinesq equations \eqref{eq:2Dbouss-stable} around a \emph{stably stratified density profile}. Note that the ill-posedness of the stably stratified 2D Boussinesq system \eqref{eq:2Dbouss-stable} could be less expected, since the linearized Boussinesq equations around $\bar \rho_{\mathrm{eq}}(y)$ with $\bar \rho_{\mathrm{eq}}'(y)<0$ and $\bar {\mathbf u}_{\mathrm{eq}}=0$ are spectrally stable (see for instance \cite{gallay}). For that reason, although both our argument and the method in \cite{chen1} are based on the construction of a leading order model for \eqref{eq:2Dbouss-grad}, our Leading Order Model for \eqref{eq:2Dbouss-stable} in the stable setting is different from the one in \cite{chen1}. In particular, even though, as in \cite{chen1}, $\de_y \rho (t, \cdot)$ has a lower order than $\omega (t, \cdot), \de_x \rho (t, \cdot)$ in terms of the small scaling parameter $0<\alpha \ll 1$, we include the equation for $\de_y \rho (t, \cdot)$ in our Leading Order Model \eqref{eq:leading1} in order to be able to close the respective remainder estimates of Section \ref{sec:rem}. There are therefore substantial differences between our Leading Order Model and the one of \cite{chen1}. It is also worth pointing out that, in contrast to \cite{chen1}, our Theorem \ref{thm:main} and Theorem \ref{thm:main2} hold in the full $\R^2$ space without any boundary.
Note also that a result of finite-time blow-up up of finite-energy smooth solutions to the 2D Boussinesq and the 3D axisymmetric Euler equations with boundaries is the recent work \cite{chen2022}. 
\subsection*{Plan of the paper}
The paper is organized as follows. The derivation of the approximated system leading to the $L^{\infty}$ ill-posedness of \eqref{eq:2Dbouss-stable} is given in Section \ref{sec:derivation}. Next, Section \ref{sec:estimateLOM} provides the $L^\infty$ ill-posedness of the approximated system (Proposition \ref{prop:expl}) as well as $\cH^k$ a priori estimates for the approximated system. Section \ref{sec:rem} provides quantitative estimates of the difference in $\cH^k$ between the approximated and the true Boussinesq system \eqref{eq:2Dbouss-grad}. The proofs of Theorem \ref{thm:main} and Theorem \ref{thm:main2} are achieved in Section \ref{sec:last}. Finally, the application to the 3D axisymmetric Euler equations (proof of Theorem \ref{thm:main3}) is given in Section \ref{sec:euler}.

\section{Derivation of the Leading Order Model and a priori estimates}\label{sec:derivation}
The purpose of this section is to obtain a simpler leading order model which approximates the original system \eqref{eq:2Dbouss-grad} on suitable time scales and exhibits the desired growth. To that end, we seek a suitable expansion of the stream-function $\psi$, defined in \eqref{eq:streamfun}, in terms of the parameter $\alpha$.

Consider $\Omega (R, \beta)$ as in \eqref{eq:new-var}, and further introduce 
\be\label{eq:psi-polar}
\Psi(R, \beta)=r^{-2} \psi (x, y).\ee 
We recall that
\begin{equation}\label{eq:spatial-der} 
\begin{aligned}
\de_x &=R^{-\frac 1 \alpha} ((\cos \beta)\alpha R \de_R - (\sin \beta) \de_\beta);\\
\de_y &=R^{-\frac 1 \alpha} ((\sin \beta) \alpha R \de_R + (\cos \beta) \de_\beta).
\end{aligned}
\end{equation}
 We can then express the velocity $\bu=(u_1, u_2)$ in terms of $\Psi$ as follows
\begin{align}
u_1&=-\de_y (r^2 \Psi (R, \beta))=-R^\frac 1 \alpha (2 \sin \beta \Psi + \alpha \sin \beta R \de_R \Psi + \cos \beta \de_\beta \Psi); \label{eq:u1}\\
u_2&=\de_x (r^2 \Psi (R, \beta))=R^\frac 1 \alpha (2 \cos \beta \Psi + \alpha \cos \beta R \de_R \Psi -\sin  \beta \de_\beta \Psi). \label{eq:u2}
\end{align}
Now, we recall Elgindi's decomposition of the Biot-Savart law, which is used here to derive the leading order model. We highlight the main ideas, and for more details, see \cite{tarek1, drivas}.  We remark that this is also related to the ``Key Lemma" by Kiselev \& \v{S}ver\'{a}k \cite{kiselev2014}.
\begin{Thm}[The singular part, \cite{tarek1}]\label{thm:tarek}
Given $\Omega=\Omega (R, \beta) \in H^k$ such that 
\begin{equation}\label{eq:ellittica}
4\Psi + \alpha^2 R^2 \de_{RR} \Psi + \de_{\beta \beta} \Psi + (4\alpha + \alpha^2) R \de_R \Psi = -\Omega,\end{equation}
there is a unique solution $\Psi$ that can be decomposed as follows
\be\label{eq:psi-main}
\Psi=\Psi(\Omega)(R, \beta)= \mathbb{P}_2(\Psi) + \mathbb{P}_{\neq}(\Psi) =:\Psa + \hE, 
\ee 
where $\mathbb{P}_2$ is the projection onto 2-modes and $\mathbb{P}_{\neq}=\mathrm{Id}-\mathbb{P}_2$. More precisely
\be \label{eq:psi2}
\Psa:= {\Ps} + \mathcal{R}^\alpha(\Omega),
\ee
where
\be \label{eq:psiapp}
\Ps=\Ps(\Omega)(R, \beta)=  \frac{\cL (\Omega)(R)}{4 \alpha}\sin (2\beta) + \frac{\cL^c (\Omega)(R)}{4 \alpha}\cos (2\beta),
\ee
is a solution to
\be
4\Ps+\de_{\beta \beta }\Ps=0,
\ee
with 
\be\label{def:L}
\cL (\Omega)(R):=\frac 1 \pi\int_R^\infty \int_0^{2\pi} \frac{\Omega(s, \beta) \sin (2\beta)}{s}\, d\beta \, d s, \qquad \cL^c (\Omega)(R):=\frac 1 \pi\int_R^\infty \int_0^{2\pi} \frac{\Omega(s, \beta) \cos (2\beta)}{s}\, d\beta \, d s
\ee
and
\begin{align}
    \label{def:R}
    \mathcal{R}^\alpha (\Omega):= \frac{R^{-\frac{4}{\alpha}}}{4\alpha\pi}\left(\sin (2\beta) \int_0^R \int_0^{2\pi} s^\frac{4}{\alpha} \frac{\Omega(s, \beta) \sin (2\beta)}{s}\, d\beta \, d s + \cos (2\beta) \int_0^R \int_0^{2\pi} s^\frac{4}{\alpha} \frac{\Omega(s, \beta) \cos (2\beta)}{s}\, d\beta \, d s\right). 
\end{align}
Finally, the total error term 
$$\Psi-\Ps=\mathcal{R}^\alpha(\Omega)+\hE$$ satisfies the inequality
\begin{align}\label{eq:hardy}
\|\mathcal{R}^\alpha(\Omega)\|_{H^k}+\|\hE\|_{H^k} \le M_k  \| \Omega\|_{H^k},
\end{align}
where the constant $M_k>0$ is independent of $\alpha$.
\end{Thm}
The proof of Theorem \ref{thm:tarek} is given in \cite[Theorem 4.22 and 4.23]{drivas}. We collect here the main ideas. 
\begin{proof}[Sketch of the proof of Theorem \ref{thm:tarek}]
   Let $\Psi=\Psi(R, \beta)$ and $\Omega = \Omega (R, \beta)$ take the form 
    \begin{align}
        \Psi(R, \beta)= \exp({2i\beta})\wt \Psi_2(R), \quad \Omega (R, \beta) = \exp({2i\beta})\wt \Omega_2 (R).
    \end{align}
    Substituting them into \eqref{eq:ellittica}, observing that the term $\exp({2i\beta}) \tilde \Psi_2 $ belongs to the kernel of the operator $4+\de_{\beta \beta}$, we can just solve the ODE 
    \begin{align*}
         \alpha^2 R^2 \wt \Psi_2''(R) + (4\alpha +\alpha^2) R \wt \Psi_2'(R) = - \wt \Omega_2(R). 
    \end{align*}
    It readily follows that 
    \begin{align*}
        \Psa= \exp({2i\beta}) \wt \Psi_2=\Ps+\mathcal{R}^\alpha(\Omega),
    \end{align*}
    as given in \eqref{eq:psiapp} and \eqref{def:R}.
    Moreover, the inequality $\|\mathcal{R}^\alpha(\Omega)\|_{H^k}+\|\hE(\Omega)\|_{H^k} \le M_k \|\Omega\|_{H^k}$ is uniform in $\alpha$ and its simple but refined proof based on the proof of the Hardy inequality is given in \cite[Lemma 4.24]{drivas}.
\end{proof}
For the regular part, we recall here \cite[Theorem 4.22]{drivas} for the reader's convenience.
\begin{Thm}[The regular part, \cite{drivas}]\label{thm:tarek2}
{Given $\Omega=\Omega (R, \beta) \in H^k$ such that 
\begin{align*}
    \int_0^{2\pi} \Omega(R, \beta) \exp\left(in\beta\right) \, d\beta =0, \qquad n=0, 1, 2,
\end{align*}
there is a unique solution to \eqref{eq:ellittica} satisfying
\begin{align}
    \alpha^2\|R^2\de_{RR}\Psi\|_{H^k}+\alpha \|R\de_R \Psi\|_{H^k}+\|\de_{\beta \beta}\Psi\|_{H^k}\le M_k \| \Omega \|_{H^k},
\end{align}
where the constant $M_k>0$ is independent of $\alpha$.}
\end{Thm}
\begin{proof}[Sketch of the proof of {\cite[Theorem 4.23]{drivas}}]
    Consider the series expansion
 \begin{align}
     \Psi = \sum_{n\ge 3}^\infty \exp({in\beta})\mathbb{P}_n(\Psi),
 \end{align}
where $n\ge 3$ as $\mathbb{P}_n(\Psi)=0$ for $n=0,1,2$. 
Taking now the scalar product of \eqref{eq:ellittica} against $\de_{\beta \beta}\Psi$ in $L^2([0, 2\pi ]\times [0, \infty))$, integrating by parts and using the crucial estimate
$$\|\de_{\beta} \Psi\|_{L^2}^2 \le \frac{1}{9} \|\de_{\beta \beta} \Psi\|_{L^2}^2$$
yields the proof. 
\end{proof}
The following lemma is a direct consequence of Theorem \ref{thm:tarek2}.

\begin{Lem}\label{lem:psierr}
    Under the assumption and notation of Theorem \ref{thm:tarek}, consider
    \begin{align}\label{eq:regular-part}
        \hE-\mathbb{P}_0(\Psi)-\mathbb{P}_1(\Psi). 
    \end{align}
    It satisfies the following estimate
    \begin{align}
    &\alpha^2\|R^2\de_{RR}(\hE-\mathbb{P}_0(\Psi)-\mathbb{P}_1(\Psi))\|_{H^k}+\alpha \|R\de_R (\hE-\mathbb{P}_0(\Psi)-\mathbb{P}_1(\Psi))\|_{H^k}\notag\\& +\|\de_{\beta \beta}(\hE-\mathbb{P}_0(\Psi)-\mathbb{P}_1(\Psi))\|_{H^k}\le M_k \left\|
    \Omega\right\|_{H^k}.
\end{align}
\end{Lem}
\subsection{The Leading Order Model}\label{sec:lom}
 \noindent 
First, we rely on \eqref{eq:spatial-der} and \eqref{eq:u1}-\eqref{eq:u2} to rewrite \eqref{eq:2Dbouss-grad} in terms of $\Psi$ as defined in \eqref{eq:psi-polar} and the new variables
\be \label{eq:variables}
\eta (t, R, \beta)= \de_x \rho (t, x, y), \quad \xi (t, R, \beta)=\de_y \rho (t, x,y),
\ee
which yields 
\be\label{eq:model-new-coord}
\ba
\de_t \Omega + \left( - (\alpha R \de_\beta \Psi ) \de_R  + (2\Psi + \alpha R \de_R \Psi) \de_\beta \right) \Omega & = \eta, \\
\de_t \eta + \left( - (\alpha R \de_\beta \Psi ) \de_R + (2\Psi + \alpha R \de_R \Psi) \de_\beta \right) \eta  & = - (\de_x u_1) \eta + \de_x u_2 (1- \xi), \\
\de_t \xi + \left( - (\alpha R \de_\beta \Psi ) \de_R + (2\Psi + \alpha R \de_R \Psi) \de_\beta \right) \xi  & = \de_x u_1 (\xi -1)-\de_y u_1 \eta. 
\ea
\ee
Second, to construct our leading order model, we rely on the expansion of $\Psi$, namely we plug the formula \eqref{eq:psiapp} for $\Ps$ into the expressions of $\de_x u_1, \de_x u_2, \de_y u_1$ . More precisely, by means of \eqref{eq:spatial-der}-\eqref{eq:u2}, Theorem \ref{thm:tarek} and a crucial cancellation, we extract the following main order terms:
\be\label{eq:approx-der-vel}
\ba 
\de_x u_1& = - \de_y u_2 = - \cos (2\beta) \de_\beta \Psi + \frac{1}{2}\sin(2\beta) \de_{\beta \beta}\Psi + \mathcal{O}(1)=  - \frac{\cL(\Omega)}{2 \alpha} +\mathcal{O}(1),
\ea 
\ee
see \eqref{def:dexu1} below for details, while
\be\label{eq:approx-der-vel2}
\de_x u_2 = \frac{\cL^c(\Omega)}{2 \alpha} +\mathcal{O}(1), \quad  \de_y u_1 = - \frac{\cL^c(\Omega)}{2 \alpha} +\mathcal{O}(1),
\ee
where the notation $\mathcal{O}(1)$ stands for lower order terms that are bounded by some constant $C>0$ independent of $\alpha$. 

We further reduce \eqref{eq:model-new-coord} by substituting the expansion \eqref{eq:psi-main} for $\Psi$ (in particular, $\Ps$ as in \eqref{eq:psiapp}) into the transport terms. In view of \eqref{eq:approx-der-vel}, we only keep 
the leading order terms in $\alpha$. This yields the following approximate system:

\begin{subequations}
\begin{align}
\de_t \Ome  + (2 \Ps)  \de_\beta \Ome &= \et, \label{eq:leading11}\\
\de_t \et +   (2 \Ps)  \de_\beta \et &= \left(\frac{\cL(\Ome)}{2 \alpha} \right)\et+\left(\frac{\cL^c(\Ome)}{2 \alpha} \right) (1-\csi), \\
\de_t \csi +  (2 \Ps)   \de_\beta \csi &= \left(\frac{\cL(\Ome)}{2 \alpha} \right)(1-\csi) + \left(\frac{\cL^c(\Ome)}{2 \alpha} \right) \et. \label{eq:leading11last}
\end{align}
\end{subequations}
The equations for $\et$ and $\csi$ are \emph{linear}, and the unique solutions with radial initial data $(\eta_{0, \mathrm{app}}, \xi_{0, \mathrm{app}})$ are also radial as both $\cL(\cdot)$ and $\cL^c(\cdot)$ are radial functions.
This crucial observation motivates us to consider radial solutions $(\eta_{\mathrm{app}}, \xi_{\mathrm{app}})$, allowing for a further reduction of the leading order model.

\begin{Lem}\label{lem:key}
Let $\Ome (t, R, \beta)$ be a solution to \eqref{eq:leading11} with initial datum $\Omega_{0, \mathrm{app}}^{\alpha, \delta}(R, \beta)=\bar g_0^{\alpha, \delta}(R) \sin (2\beta)$ where $ g_0^{\alpha, \delta}\in C_c^{\infty}([1,\infty))$ and suppose that the right-hand side $\et=\et (t, R)$ is a radial function. 
Then $\Ome$ admits the representation formula 
\begin{equation}\label{def:g1}
    \Ome (t)=g(t)+ \int_0^t \et (\tau, R) \, d\tau,
\end{equation}
where $g(t)$ is the unique solution to the transport equation 
\be\label{eq:g}
\de_t g +  \frac{\cL (g)}{2\alpha} \sin (2\beta)   \de_\beta g=0, \qquad g(0)= \bar g_0^{\alpha, \delta} (R) \sin (2\beta).
\ee
In particular, $\cL^c(\Ome)=0$ and $\Ps =\Ps (\Ome)$ in \eqref{eq:psiapp} contains only the odd component, i.e.
\begin{align}\label{eq:odd}
    \Ps(\Ome)= \frac{\cL(\Ome)}{4\alpha} \sin(2\beta). 
\end{align}
\end{Lem}
\begin{proof}
From \eqref{def:g1} and the assumption that $\et$ is a radial function, we deduce that
$$\cL(\Ome)=\cL(g), \qquad \cL^c(\Ome)=\cL^c(g),$$
with $\cL, \cL^c$ as defined in \eqref{def:L}. It follows from the definition of $g(t)$ in \eqref{def:g1}, the equation for $\Ome$ in \eqref{eq:leading11} and the expression of $\Ps$ in \eqref{eq:psiapp} that 
$g(t)$ satisfies
\begin{equation}\label{eq:g-long}
    \de_t g + \frac{\cL(g)}{2\alpha}\sin(2\beta) \de_\beta g + \frac{\cL^c(g)}{2\alpha}\cos(2\beta) \de_\beta g =0, 
    \qquad 
    g(0)=\bar g_0^{\alpha, \delta}(R) \sin (2\beta).
\end{equation}
The solution to \eqref{eq:g} (with initial data $g(0)= \bar g_0^{\alpha, \delta}(R) \sin (2\beta)$) is unique and odd, as explicitly given in \eqref{eq:g-formula}. Consequently, $\mathcal{L}^c(g(t))=0$ for all times and, in particular, such $g(t)$ solves \eqref{eq:g-long}.  Thus, by uniqueness, the unique solutions to the Cauchy problems \eqref{eq:g} and \eqref{eq:g-long} coincide. Finally, \eqref{eq:odd} follows. Further details can be found in \cite[Proposition 2.1]{tarek2}.
\end{proof}

%
In the framework of Lemma \ref{lem:key}, namely for $\et$ radial and odd-symmetric $\Ome$, we reduce system \eqref{eq:leading11}-\eqref{eq:leading11last} replacing $\Ps$ with \eqref{eq:odd}. Observe that from \eqref{eq:approx-der-vel2} it follows that
\begin{align*}
    \de_x u_2 = \mathcal{O}(1), \quad \de_y u_1 = \mathcal{O}(1). 
\end{align*}
Finally, this observation leads to the following Leading Order Model: 
\be\label{eq:leading1}\tag{LOM}
\ba
\de_t \Ome  + \frac{\cL(\Ome)}{2\alpha} \sin(2\beta)  \de_\beta \Ome &= \et, \\
\de_t \et +   \frac{\cL(\Ome)}{2\alpha} \sin(2\beta)  \de_\beta \et &= \left(\frac{\cL(\Ome)}{2 \alpha} \right)\et, \\
\de_t \csi +  \frac{\cL(\Ome)}{2\alpha}  \sin(2\beta) \de_\beta \csi &= \left(\frac{\cL(\Ome)}{2 \alpha} \right)(1-\csi). 
\ea
\ee
We will construct initial data such that $(\eta_{0, \mathrm{app}},\xi_{0, \mathrm{app}})$ are radial functions and $\omega_{0, \mathrm{app}}$ has odd symmetry in $x$ and $y$, giving rise to solutions within the framework of Lemma \ref{lem:key}, which enables to reduce the evolution equations to \eqref{eq:leading1}. 

This way, let $\eta_{0, \mathrm{app}}=\et(0, R, \beta)= \bar \eta_0^{\alpha, \delta} (R)$ be a radial function. Then the transport equation
\begin{equation}\label{eq:eta}
\de_t \et +  \frac{\cL(\Ome)}{2\alpha}  \sin(2\beta) \de_\beta \et =  \left(\frac{\cL(\Ome)}{2 \alpha} \right)\et, \qquad \et(0, R, \beta)= \bar \eta_0^{\alpha, \delta} (R)
\end{equation}
admits the following solution:
\be\label{eq:eta-sol}\et(t, R, \beta)= \bar \eta_0^{\alpha, \delta} (R) \exp{\left(  \frac{1}{2\alpha} \int_0^t {\cL(\Ome (\tau))}\, d\tau\right)}.\ee
Let us now plug the above formula inside the equation for $\Omega_{\mathrm{app}}$ in \eqref{eq:leading1}, yielding
\be\label{eq:omega}
\de_t \Ome + \frac{\cL(\Ome)}{2\alpha}  \sin(2\beta)  \de_\beta \Ome=\bar \eta_0^{\alpha, \delta} (R) \exp{\left( \frac{1}{2\alpha} \int_0^t {\cL(\Ome (\tau))} \, d\tau\right)}.
\ee

Relying on the definition of $g(t)$ in \eqref{def:g1} yields a formula for $\Ome$:
\be\label{eq:def-g}
\Ome (t)= g(t) + \bar \eta_0^{\alpha, \delta}(R) \int_0^t \exp{\left( \frac{1}{2\alpha} \int_0^\tau {\cL(\Ome (\mu))}\, d\mu \right)}\, d\tau.
\ee
In particular, the behavior of $\Ome (t)$ can be deduced by studying the equation for $g(t)$ in \eqref{eq:g}. Note that this equation \eqref{eq:g}, upon applying a similar change of variables, corresponds to the leading-order model derived in \cite{tarek2}.
This observation allows us to rely on the following set of known facts. 
\begin{Lem}[List of known facts, from {\cite[Lemma 3.1, Proposition 3.3]{tarek2}}]\label{prop:LOM}
Let us consider the Cauchy problem \eqref{eq:g}. The following hold. 
\begin{enumerate}
\item The solution $g(t)$ reads as follows:
\be\label{eq:g-formula}
\ba
g(t)=g(t)(t, R, \beta)&= \bar g_0^{\alpha, \delta}(R)\frac{2(\tan \beta) \exp{\left(-\frac 1 \alpha \int_0^t \cL (g(\tau)) \, d \tau \right)}}{1+(\tan^2\beta) \exp{\left(-\frac 2 \alpha \int_0^t \cL (g(\tau)) \, d \tau \right)}}\\
&=\bar g_0^{\alpha, \delta}(R)\frac{\sin (2\beta) \exp{\left(-\frac 1 \alpha \int_0^t \cL (g(\tau)) \, d \tau \right)}}{\cos^2\beta+(\sin^2\beta) \exp{\left(-\frac 2 \alpha \int_0^t \cL (g(\tau)) \, d \tau \right)}}.
\ea
\ee
\item For some $c_1>0, c_2>0$ independent of $\alpha$, the lower and upper bound below are fulfilled
\be\label{eq:Lg-bounds}
c_1 \int_R^\infty \frac{\bar g_0^{\alpha, \delta}(s)}{s} \,  \exp{\left(-\frac 1 \alpha \int_0^t \cL(g(\tau)) \, d\tau\right)}\, ds \le \cL(g(t))(R) \le c_2 \int_R^\infty \frac{\bar g_0^{\alpha, \delta}(s)}{s} \,  \exp{\left(-\frac 1 \alpha \int_0^t \cL(g(\tau)) \, d\tau\right)}\, ds,
\ee
\item the following estimates hold
\be\label{eq:g-upperandlower}
\frac{2\alpha}{c_2} \log \left(1+\frac{c_2}{2\alpha} t \cG_0(R)\right) \le \int_0^t \cL(g(\tau))(R) \, d\tau  \le \frac{2\alpha}{c_1} \log \left(1+\frac{c_1}{2\alpha} t \cG_0(R) \right),
\ee
where
$$ \cG_0(R) = \cG_0(\bar g_0^{\alpha, \delta})(R) = \int_R^\infty \frac{\bar g_0^{\alpha, \delta}(s)}{s} \, ds.$$
\item Finally, $\Ome (t)$ can be explicitly expressed in terms of $g(t)$ as follows:
\be\label{eq:ome-explicit}
\Ome (t)= g(t) + \bar \eta_0^{\alpha, \delta}(R) \int_0^t \exp{\left( \frac{1}{2\alpha} \int_0^\tau {\cL(g(\mu))}\, d\mu \right)}\, d\tau.
\ee
\end{enumerate}
\end{Lem}
\begin{proof}
Item (1) is readily verified by plugging formula \eqref{eq:g-formula} into equation \eqref{eq:g}. Let us consider item (2). By \eqref{def:L} and formula \eqref{eq:g-formula},
\be\ba
\cL(g(t))&= \frac 4 \pi \int_R^\infty  \frac{\bar g_0^{\alpha, \delta}(s)}{s} \int_0^{2\pi} \frac{(\sin^2 \beta) \exp{\left(-\frac 1 \alpha \int_0^t \cL(g(\tau)) \, d\tau\right)}}{1+(\tan^2 \beta) \exp{\left(-\frac 2 \alpha \int_0^t \cL(g(\tau)) \, d\tau\right)}} \, d\beta\, ds \\
& =  4 \int_R^\infty  \frac{\bar g_0^{\alpha, \delta}(s)}{s} \frac{\exp{\left(-\frac 1 \alpha \int_0^t \cL(g(\tau)) \, d\tau\right)}}{\left(1+\exp{\left(-\frac 1 \alpha \int_0^t \cL(g(\tau)) \, d\tau\right)}\right)^2} \, ds.
\ea\ee
From the last line, one notices in particular that $\cL(g(t)) \ge 0$ if $\bar g_0^{\alpha, \delta}\ge 0$, from which the estimates of (2) follow. Item (3) is proved in \cite[Lemma 3.2]{tarek2}. The idea is to verify that the operator 
\be
\cG_t (g(t) ) :=  \int_R^\infty \frac{\bar g_0^{\alpha, \delta}(s)}{s} \,  \exp{\left(-\frac 1 \alpha \int_0^t \cG_\tau (g(\tau)) \, d\tau\right)}\, ds
\ee
satisfies the Riccati type equation
$$\de_t  \cG_t (g(t) )  = - \frac{(\cG_t (g(t) ))^2}{\alpha} \quad \Rightarrow \quad \cG_t (g(t) )=\frac{ \cG_0 (\bar g_0^{\alpha, \delta})(R)}{1+\frac t \alpha   \cG_0 (\bar g_0^{\alpha, \delta})(R)},
$$
from which one has the desired estimates after integrating in time and infer suitable lower and upper bounds \eqref{eq:g-upperandlower} by solving the respective differential inequalities.
Finally, \eqref{eq:ome-explicit} relies on the observation that $\cL(\Ome)=\cL(g)$ (as in Lemma \ref{lem:key}). 
\end{proof}

\subsection{A concrete example of a sequence of initial data}\label{sec:initialdata}
The initial data is constructed in terms of the scaled coordinates $(R,\beta)$, and a characterization in the original coordinates is provided. For our purposes, it is convenient to work in the weighted Sobolev spaces $\mathcal{H}^k$ defined in \eqref{def:Hk}. These spaces enjoy suitable embeddings in $L^{\infty}$, see Lemma \ref{lem:embedding}. Note that the $H_{xy}^k(\mathbb{R}^2)$ and $\cH_{R,\beta}^k([0,\infty)\times [0,2\pi])$ norms are not equivalent due to the change of measure $r\,dr$ or $dR$.\\
\textbf{Initial vorticity}.
The initial vorticity is given by $\Omega_0(R,\beta) = \overline{g}_{0}^{\alpha,\delta}(R) \sin(2\beta)$, where $\overline{g}_{0}^{\alpha,\delta} \in C_c^{\infty}([1,\infty))$ and $\|\overline{g}_{0}^{\alpha,\delta}\|_{L^{\infty}}=\delta$. This choice is motivated by the odd-symmetry in both $x$ and $y$, which is relevant for solving \eqref{eq:g} explicitly, and the action of the Riesz-Transform on such functions. The localization of the support for $R\geq 1$ is required for Lemma \ref{prop:LOM}. More precisely, in $xy$-coordinates, one has
\begin{equation*}
    \omega_0(x,y) = \overline{g}_{0}^{\alpha,\delta}(|\textbf{x}|^{\alpha}) \frac{2xy}{|\bx|^2}, \quad \bx:=(x,y).
\end{equation*}
While $\omega_0$ is of size $\delta$ in $L^{\infty}$-norm, it is not uniformly bounded (in $\alpha$) in Sobolev spaces $H^s(\mathbb{R}^2)$ in $xy$-coordinates. The radial cut-off $\overline{g}_{0}^{\alpha,\delta}(R)$ does not depend on $\alpha$, but its support in the original radial coordinate $r=|\textbf{x}|$ spreads out as $\alpha\rightarrow 0$. 
However, note that $\Omega_0$ is uniformly bounded in $\alpha$ in the Sobolev spaces $\mathcal{H}^k$ tailored to the coordinates $(R,\beta)$.
%
Indeed, for all $k\in \N$ there exists $C_k>0$ independent of $\alpha$ such that
\begin{equation*}
\left\|\Omega_0\right\|_{\cH^k}\leq C_k.
\end{equation*}
For \eqref{eq:leading1}, we set $\Omega_{\mathrm{app},0}=\Omega_0$ so that the initial error for the vorticity vanishes, i.e. $\Omega_{r,0}:=\Omega_0-\Omega_{\mathrm{app},0}=0$. \\
\textbf{Initial density}. 
We consider a localized small amplitude initial perturbation $\rho_0$ of the stratified equilibrium $\bar \rho_{\text{eq}}(y) = \bar \rho_c - y$ in the horizontal direction. To solve \eqref{eq:leading1}, we seek initial data for the density gradient $(\eta_0,\xi_0)$ that is radial at leading order, see Section \ref{sec:lom}. Here, leading order refers to an error measured in $\cH^k$-norms. For the sake of simplicity, we construct data such that $\xi_0$ vanishes at leading order. Let
\begin{equation}\label{eq:initial-rho}
\frac{\rho_0(x,y)}{x} =\eta_0^{\alpha, \delta}\left(|\bx|^\alpha-\frac{1}{2} \right)=:\bar \eta_0^{\alpha, \delta}(|\bx|^\alpha)=\bar \eta_0^{\alpha, \delta}(R),
\end{equation}
where $\eta_0^{\alpha, \delta}\in C_c^{\infty}\left((\frac{1}{8},\frac{1}{4})\right)$ and $\|\eta_0^{\alpha, \delta}\|_{L^{\infty}}=\delta$. This gives $\supp \bar \eta_0^{\alpha, \delta}(R) \subset \left(\frac 58, \frac 34\right)$.
Observe that, from
$$\de_x=R^{-\frac 1 \alpha} (\alpha (\cos \beta) R \de_R - (\sin \beta) \de_\beta ),$$
it follows
\be
\de_x  \rho_0^{\alpha, \delta} (x, y) = \bar \eta_0^{\alpha, \delta}(R) + \alpha  R \de_R\bar \eta_0^{\alpha, \delta}(R) \cos^2(\beta). 
\ee
Concerning the initial data for the Leading Order Model \eqref{eq:leading1}, the choice in \eqref{eq:eta} consists of $$\de_x \rho_{\mathrm{app}} (0, \cdot)= \de_x \rho_{0,\mathrm{app}}^{\alpha, \delta}  (\cdot) =\eta_\mathrm{app}(0, R, \beta)= \bar \eta_0^{\alpha, \delta} (R),$$
which, in other words, means that we make an error at the initial time that is
\be
\de_x \rho_r (0, \cdot):=\de_x \rho_0^{\alpha, \delta} (\cdot)- \de_x \rho_{0, \mathrm{app}}^{\alpha, \delta} (\cdot)=\alpha R (\de_R \bar  \eta_0^{\alpha, \delta})(R) \cos^2 \beta. 
\ee
We estimate the size of such error
\be\label{eq:data eta rem}
\|\de_x \rho_r (0)\|_{\cH^N} = \|\eta_r (0)\|_{\cH^N} \le \alpha \|\bar \eta_0^{\alpha, \delta}\|_{\cH^{N+1}}=\alpha C_{N+1}.
\ee
Similarly, the initial datum for $\xi=\de_y \rho$ is computed from \eqref{eq:spatial-der}: 
\be
\de_y \rho (0, \cdot)=\xi_{0}(\cdot) = \alpha R(\de_R \bar \eta_0^{\alpha, \delta}(R)) \sin \beta \cos \beta=\alpha\frac{R}{2}(\de_R \bar \eta_0^{\alpha, \delta}(R)) \sin 2\beta.
\ee
For the Leading Order Model \eqref{eq:leading1}, we then set 
\be\label{eq:xi-initial}
\csi (0, R, \beta)= 0,
\ee
so that the error at initial time is estimated as
\be\label{eq:data xi rem}
\|\de_y \rho_r (0)\|_{\cH^N} = \|\csi (0)\|_{\cH^N} \le \alpha \|\bar \xi_0^{\alpha, \delta}\|_{\cH^{N+1}}=\alpha C_{N+1}.
\ee

\section{Estimates on the Leading Order Model}\label{sec:estimateLOM}
The following result yields $L^\infty(\R^2)$ ill-posedness of our Leading Order Model \eqref{eq:leading1}. Afterwards, this section provides $\cH^k$ estimates on the Leading Order Model \eqref{eq:leading1} for initial data $(\Omega_{0, \mathrm{app}}, \eta_{0, \mathrm{app}}, \xi_{0, \mathrm{app}})$ where $\eta_{0, \mathrm{app}}=\bar \eta_{0}^{\alpha, \delta}(R)$ is a radial function, for instance as in Section \ref{sec:initialdata}. First, the ill-posedness of the \eqref{eq:leading1} is stated below.
\begin{Prop}\label{prop:expl}
There exist $0<\alpha_0 \ll 1$  such that, for any $0<\alpha \le \alpha_0 $ and any $\delta>0$, there exist initial data $(\Omega_{0, \mathrm{app}}^{\alpha, \delta}(R, \beta), \eta_{0, \mathrm{app}}^{\alpha, \delta}(R, \beta), \xi_{0, \mathrm{app}}^{\alpha, \delta}(R, \beta))$ of the following form 
\begin{align}
 \Omega_{0,\mathrm{app}}^{\alpha, \delta} (R, \beta) = \bar g_0^{\alpha, \delta} (R) \sin (2\beta), \qquad \eta_{0, \mathrm{app}}^{\alpha, \delta} (R, \beta)= \bar \eta_0^{\alpha, \delta}(R), \qquad  \xi^{\alpha, \delta}_{0, \mathrm{app}}(R, \beta)= 0,
\end{align}
where $\bar g_{0, \mathrm{app}}^{\alpha, \delta} (R), \bar \eta_{0, \mathrm{app}}^{\alpha, \delta} (R) \in C_c^\infty([5/8, \infty))$ with
$$\|(\Omega_{0, \mathrm{app}}^{\alpha, \delta}, \eta_{0, \mathrm{app}}^{\alpha, \delta}, \xi_{0, \mathrm{app}}^{\alpha, \delta})\|_{L^\infty(\R^2)} =  \delta,$$ 
such that the solution $(\Ome^{\alpha, \delta}(t), \eta_\mathrm{app}^{\alpha, \delta}(t), \xi_\mathrm{app}^{\alpha, \delta}(t))$ to the Cauchy problem associated with \eqref{eq:leading1} satisfies
\be\label{eq:eta-growth}\ba
\|\et^{\alpha, \delta}(t) \|_{L^\infty(\R^2)}  &\ge \|\eta_{0, \mathrm{app}}^{\alpha, \delta}\|_{L^\infty (\R^2)} \left(1+\frac{c_2 t}{2\alpha}C_0\right)^{\frac{1}{c_2}},
\ea\ee
where $C_0=\sup\limits_{R \in \mathrm{supp}(\bar g_0^{\alpha, \delta}(R))} \int_R^\infty \frac{\bar g_0^{\alpha, \delta}(s)}{s} \, ds$ and $c_2>0$ is independent of $\alpha$. In particular: 
\be\label{eq:eta-growth1}\ba
\sup\limits_{t \in [0, T^*(\alpha)]} \|\et^{\alpha, \delta}(t) \|_{L^\infty(\R^2)}  &\ge  \|\eta_{0, \mathrm{app}}^{\alpha, \delta}\|_{L^\infty (\R^2)} \left(1+\log |\log \alpha|\right)^{\frac{1}{c_2}},
\ea\ee
where $T^*(\alpha) =  C{\alpha}\log |\log (\alpha)|$ and $C>0$ is independent of $\alpha$. 
\end{Prop}
\begin{proof}[Proof of Proposition \ref{prop:expl}]
The lower bound for $\eta_\mathrm{app}$ is a direct consequence of the explicit formula \eqref{eq:eta-sol}, namely
\be
\eta_\mathrm{app} (t, R, \beta)= \bar \eta_0^{\alpha, \delta} (R) \exp{\left(  \frac{1}{2\alpha} \int_0^t {\cL(\Omega_\mathrm{app} (\tau))}\, d\tau\right)},
\ee
and the lower bounds \eqref{eq:g-upperandlower} for $g(t)$ satisfying \eqref{eq:g}. Then the growth in $\alpha$ simply follows from the choice of $T^*(\alpha)$. 
\end{proof}
\begin{Rmk}[On the point-wise growth of solutions for short times]
    From \eqref{eq:eta-growth}, it is evident that the point-wise growth of solutions for short times is at least linear. This observation implies that, despite their substantial differences, the linear system \eqref{eq:2Dbouss-grad-lin} (governed by the semigroup $\exp({R_1 t})$, the basis for the mild ill-posedness result in \cite{tarek3}, as discussed below \eqref{eq:2Dbouss-grad-lin}) and our leading-order model \eqref{eq:leading1} exhibit the same point-wise behavior in time, at least for suitable initial data.
    To provide further clarity, let us address three key points:
    \begin{enumerate}
        \item The choice of the (double) logarithmic timescale $T^*(\alpha)$ in Proposition \ref{prop:expl} ensures norm inflation from the inequality \eqref{eq:eta-growth}, as seen in \eqref{eq:eta-growth1}. While considering a shorter time window, such as $T^*=C{\alpha}$, is possible, it does not result in inflation but merely leads to mild ill-posedness. 
        \item A (single) logarithmic timescale might seem sufficient for our purposes initially. Strictly speaking, this is true if our goal is solely to demonstrate the norm inflation of the leading-order model. However, the requirement for a double logarithmic timescale, hence a shorter time, stems from the estimates of the remainder terms, particularly the error introduced by approximating the transport terms (see Remark \ref{rmk:largest}).
        \item The approximated model derived by Elgindi \& Khalil exhibits weaker point-wise growth compared to ours, specifically of logarithmic type (see \cite[Theorem 1]{tarek2}). However, the timescale of validity of the approximation is comparable to ours, owing to the shared underlying principles (specifically, some of the remainder terms are exactly the same).
    \end{enumerate}
\end{Rmk}
\begin{Rmk}[$\Omega$ does not blow up]
    One has from the explicit expression for $g(t)$ in \eqref{eq:g-formula} and the estimates of Proposition \ref{prop:LOM} that there exist $c, C>0$ such that
\begin{align}\label{est:g-inf}
   c\|\bar g_0^{\alpha, \delta}\|_{L^\infty} \le  \|g(t)\|_{L^\infty} \le C \|\bar g_0^{\alpha, \delta}\|_{L^\infty}.
\end{align}
Now, notice that \eqref{eq:def-g} and the estimates of (4) of Lemma \ref{prop:LOM} yield 
\begin{align*}
\Ome(t)&=g(t)+\int_0^t\et(\tau)d\tau\le g(t)+ \int_0^t\left(1+\frac{c_2 \tau}{2\alpha}C_0\right)^{\frac{1}{c_2}}d\tau\\
&= g(t)+\frac{2\alpha}{C_0(1+c_2)}\left(\left(1+\frac{t}{2\alpha}C_0c_2\right)^{\frac{c_2+1}{c_2}}-1\right).
\end{align*}
Therefore $\Ome (t)$ does not blow up in $\left[0, T^*(\alpha)=C{\alpha}\log |\log \alpha|\right]$. As it will follow from the remainder estimates, the vorticity $\omega (t, \cdot)$ (solving \eqref{eq:2Dbouss-grad}) does not blow up as well (it may explode for later times).
\end{Rmk}
\begin{Rmk}[An explicit formula for $\csi$] \label{rmk:explicitcsi}
One can easily solve for $\csi$ in \eqref{eq:leading1}, obtaining the explicit formula 
\be\label{eq:csi-explicit}
\csi (t, \cdot) =  \left(1-\exp{\left(-\frac{1}{2\alpha} \int_0^t {\cL(g(\tau))}\, d\tau\right)}\right). 
\ee
Moreover, using \eqref{eq:g-upperandlower}, the above formula allows to get an upper bound in $L^\infty (\R^2)$, 
\be
\|\csi(t)\|_{L^\infty (\R^2)} = \| \de_y \rho_{\mathrm{app}}(t)\|_{L^\infty} \le 1+ \left(1+ \frac{c_2 t}{2\alpha} C_0 \right)^{-\frac{1}{c_2}}, 
\ee
with $c_2$ and $C_0$ as in Lemma \ref{prop:LOM} and Proposition \ref{prop:expl} respectively.
In the time interval $[0, T^*(\alpha)]$ as in Proposition \ref{prop:expl}, it holds
\be
\|\csi(t)\|_{L^\infty (\R^2)} = \| \de_y \rho_{\mathrm{app}}(t)\|_{L^\infty} \le 1+ \left(1+ {\log |\log \alpha|}  \right)^{-\frac{1}{c_2}}  < 3,
\ee
so that $\csi (t, \cdot)$ does not blow up in $L^\infty (\R^2)$ and $\de_y \rho (t, \cdot )$ (solving \eqref{eq:2Dbouss-grad}) does not blow up in the time interval $[0, T^*(\alpha)]$. 
\end{Rmk}

 The result below provides $\cH^k$ estimates of the Leading Order Model \eqref{eq:leading1}.
 \begin{Prop}[Estimates on the \eqref{eq:leading1}]\label{lem:35} 
For any $k \ge 3$, let $(\Ome^{\alpha, \delta}(t), \eta_\mathrm{app}^{\alpha, \delta}(t), \xi_\mathrm{app}^{\alpha, \delta}(t))$ be a solution to \eqref{eq:leading1} with initial data $(\Omega_{0, \mathrm{app}}^{\alpha, \delta}, \eta_{0, \mathrm{app}}^{\alpha, \delta}, \xi^{\alpha, \delta}_{0, \mathrm{app}}) \in \cH^k$, where 
\begin{align}
\Omega^{\alpha, \delta}_{0, \mathrm{app}}=\bar{g}_0^{\alpha, \delta} (R)\sin(2\beta),\quad 
\eta^{\alpha, \delta}_{0, \mathrm{app}}= \bar \eta_0^{\alpha, \delta}(R),\quad  \xi^{\alpha, \delta}_{0, \mathrm{app}}=0,
\end{align}
and where $(\bar{g}_0^{\alpha, \delta} (R), \bar{\eta}^{\alpha, \delta}_0(R)) \in C_c^\infty ([5/8, \infty))$
with $\bar g_0^{\alpha, \delta}(R)\ge 0$. Then, there exist constants $0<C_k=C_k(\|\bar g_0^{\alpha, \delta}\|_{\cH^k}, \|\bar \eta_0^{\alpha, \delta}\|_{\cH^k})$ as defined in \eqref{eq:size-initial} such that the following estimates hold 
\be\ba\label{eq:est-LOM}
\|\Ome(t)\|_{\cH^k}&\le  \left(\|\bar g_0^{\alpha, \delta}\|_{\cH^k}+\alpha\exp\left({\frac{C_k}{\alpha} t}\right) \right) \le C_{k} + \alpha\exp\left({\frac{C_k}{\alpha} t}\right),\\
\|\et(t)\|_{\cH^k} &\le  \|\bar \eta_0^{\alpha, \delta}\|_{\cH^k}\exp\left({\frac{C_k}{\alpha} t} \right)\le C_{k}\exp \left({\frac{C_k}{\alpha} t}\right),\\
\|\csi(t)\|_{\cH^k} &\le  (C_k+\alpha \|\bar \eta_0^{\alpha, \delta}\|_{\cH^{k+1}}  )\exp\left({\frac{C_k}{\alpha} t}\right) \le C_k \exp\left({\frac{C_k}{\alpha} t}\right).
\ea
\ee
Moreover, the function $\Ps(\Ome)(R)=\frac{1}{4 \alpha}\cL (\Ome)(R)\sin (2\beta)$, as in \eqref{eq:odd},
satisfies the estimates below
\be\label{eq:est-psi2}
\|\Ps (\Ome)\|_{\cW^{k+1,\infty}} \le   \frac{C_{k+1}}{\alpha}, \qquad \|\Ps (\Ome)\|_{\cH^{k+1}} \le   \frac{C_{k+1}}{\alpha}.
\ee
\end{Prop}
Before providing a proof of Proposition \ref{lem:35}, we state and prove a preliminary result.
\begin{Lem}\label{StimeLg}
For any $ k\ge 3$, if $g(t)=g(t)(t, R, \beta)$ fulfills formula \eqref{eq:g-formula} with $g(t)|_{t=0}=\bar g_0^{\alpha, \delta}(R) \in \cH^{k+1}_R$ such that $0<\bar g_0^{\alpha, \delta} (R) \in C_c^\infty ([1, \infty))$, then there exists $0<\alpha_0 \ll 1$ small enough such that the following estimates hold for all $\alpha \le \alpha_0 $
\begin{align}
\|\mathcal{L}(g(t))\|_{L^{\infty}_R}&\leq C_{1} ,\quad \|\mathcal{L}(g(t))\|_{L^{2}_R}\leq C_{0},\label{leggere10}\\ 
\|(1+R)\partial_R\mathcal{L}(g(t))\|_{L^{\infty}_R}&\leq C_{2},\quad \| (1+R)\partial_R\mathcal{L}(g(t))\|_{L^{2}_R}\leq C_{1} \label{pesate10},\end{align}
and, for $k \ge 2$, 
\begin{align}
\|\partial_R^{k+1}\mathcal{L}(g(t))\|_{L^{\infty}_R}&\leq C_{k+1},\quad \|\partial_R^{k+1}\mathcal{L}(g)\|_{L^{2}_R}\leq C_{k},\label{leggere1}\\ 
\|R^{k+1}\partial_R^{k+1}\mathcal{L}(g(t))\|_{L^{\infty}_R}&\leq C_{k+1},\quad \| R^{k+1}\partial_R^{k+1}\mathcal{L}(g(t))\|_{L^{2}_R}\leq C_{k}
\label{pesate1},\end{align}
and the function $g(t)$ satisfies
\be\label{eq:est-g}
\|g(t)\|_{\cH^k} \le C_k,
\ee
where the constants $C_k=C_k(\|\bar g_0^{\alpha, \delta}\|_{\cH^k_R})$ depend only on the initial data.
\end{Lem}
\begin{proof}
By \eqref{eq:Lg-bounds}, it follows that
\be
\cL(g(t)) (R) \le  \int_R^\infty \frac{\bar g_0^{\alpha, \delta}(s)}{s} \exp{\left(-\frac 1 \alpha \int_0^t \cL(g(\tau)) \, d\tau \right)} \, ds,
\ee
and recalling from \eqref{eq:g-upperandlower} that $\cL(g(t)) \ge 0$ (being $\bar g_0^{\alpha, \delta}(R) >0$), one has, using  the embedding $H^1_R \hookrightarrow L^\infty$, that
\be \|\cL(g(t))\|_{L^2_R} \le \left\|\int_R^\infty \frac{\bar g_0^{\alpha, \delta}(s)}{s} \, ds\right\|_{L^2_R} \le  \|\bar g_0^{\alpha, \delta}\|_{L^2_R} \le   C_0,
\ee
\be
\|\cL(g(t))\|_{L^\infty} \le  \left\|\int_R^\infty \frac{\bar g_0^{\alpha, \delta}(s)}{s} \, ds\right\|_{L^\infty} \le  \|\bar g_0^{\alpha, \delta}\|_{H^1_R} \le    C_1.
\ee
For the derivative $\de_R$, we plug the expression of $g(t)$ in \eqref{eq:g-formula} inside the definition \eqref{def:L}, yielding
\be\ba\label{eq:Lg-formula}
\cL(g(t))&= \frac 1 \pi \int_R^\infty \int_0^{2\pi} \frac{4 (\sin^2\beta)  \exp\left(-\frac 1 \alpha \int_0^t \cL (g(\tau)) \, d \tau \right)}{1+(\tan^2\beta) \exp\left(-\frac 2 \alpha \int_0^t \cL (g(\tau)) \, d \tau \right)} \, \frac{\bar g_0^{\alpha, \delta}(s)}{s} \, ds \, d\beta\notag\\
&= 4 \int_R^\infty  \frac{\exp\left(-\frac 1 \alpha \int_0^t \cL(g(\tau)) \, d\tau\right)}{\left(1+ \exp\left(-\frac 1 \alpha \int_0^t \cL(g(\tau)) \, d\tau\right)\right)^2}\frac{\bar g_0^{\alpha, \delta}(s)}{s} \, ds.
\ea\ee
Applying $\de_R$ yields 
\be\label{eq:der-Lg}
\de_R \cL(g(t))= - \frac{4 \bar g_0^{\alpha, \delta}(R)}{R} \frac{\exp\left(-\frac 1 \alpha \int_0^t \cL(g(\tau)) \, d\tau\right)}{\left(1+ \exp\left(-\frac 1 \alpha \int_0^t \cL(g(\tau)) \, d\tau\right)\right)^2},
\ee
from which one has, using again the embedding $H^1_R \hookrightarrow L^\infty$, 
\be
\|\de_R \cL(g(t))\|_{L^\infty} \le  4\|\bar g_0^{\alpha, \delta}\|_{H^1_R} \le C_1.
\ee
Next, let us consider
\be
\ba
\|\de_R^2 \cL(g(t))\|_{L^\infty} & \lesssim \left\| \frac{\exp\left(-\frac 1 \alpha \int_0^t \cL(g(\tau)) \, d\tau\right)}{\left(1+ \exp\left(-\frac 1 \alpha \int_0^t \cL(g(\tau)) \, d\tau\right)\right)^2}\right\|_{L^\infty} \\
& \qquad \times \left(\left\| \de_R\left(\frac{\bar g_0^{\alpha, \delta}(R)}{R} \right)\right\|_{L^\infty}+\frac 1 \alpha  \left\|\frac{\bar g_0^{\alpha, \delta}(R)}{R} \, \left(\int_0^t \de_R \cL(g(s)) \, ds \right)\right\|_{L^\infty} \right)\\
& \lesssim  \left\|\exp\left(-\frac 1 \alpha \int_0^t \cL (g(\tau)) \, d \tau \right)\right\|_{L^\infty}\left(\|\bar g_0^{\alpha, \delta}\|_{\cH^2_R} + \frac 1 \alpha  \|\bar g_0^{\alpha, \delta}\|_{\cH^1_R} \left\|\left(\int_0^t \de_R \cL(g(s)) \, ds \right)\right\|_{L^\infty}\right) \\
&\le  \left\|\exp\left(-\frac 1 \alpha \int_0^t \cL (g(\tau)) \, d \tau \right)\right\|_{L^\infty}\left( C_2 + \frac{C_1}{\alpha} \left\|\left(\int_0^t \de_R \cL(g(s)) \, ds \right)\right\|_{L^\infty}\right).
\ea\notag
\ee
Plugging \eqref{eq:der-Lg} into the above and using \eqref{eq:g-upperandlower}, for $0 < \alpha \ll \alpha_0 $ small enough one has 
\be
\ba
\|\de_R^2 \cL(g(t))\|_{L^\infty} & \le  \frac{ C_1}{\alpha} \left\| \int_0^t \exp\left(-\frac 1 \alpha \int_0^\tau \cL(g(s)) \, ds \right)\right\|_{L^\infty}\le  \frac{C_1}{\alpha},
\ea
\ee
thanks to the (nonnegative) sign of $\cL (g(s))$ due to the fact that at the initial time $\bar g_0^{\alpha, \delta}(R)\ge 0$, which is ensured by \eqref{eq:g-formula}-\eqref{eq:Lg-bounds}.

Furthermore, using Fa\`a di Bruno formula \cite{faa}, yields that
\be
\ba
\|\de_R^{k+1} \cL(g(t))\|_{L^\infty}= 4\left\|  \sum_{\ell=0}^{k} \de^{k-\ell}_R \left(\frac{\bar g_0^{\alpha, \delta}(R)}{R}\right) \de_R^\ell \left(\frac{\exp\left(-\frac 1 \alpha \int_0^t \cL(g(\tau)) \, d\tau\right)}{\left(1+ \exp\left(-\frac 1 \alpha \int_0^t \cL(g(\tau)) \, d\tau\right)\right)^2}\right) \right\|_{L^\infty}\\
\le  4 \sum_{\ell=0}^{k}\left\|\partial_R^{k-\ell}\left(\frac{\bar{g}_0(R)}{R}\right)\exp\left(-\frac{1}{2\alpha}\int_0^t\mathcal{L}(g(\tau))d\tau\right)\sum_{\nu=1}^\ell \frac{(2\alpha)^{-\nu}}{\nu!}\sum_{h_1+\cdots h_{\nu}=\ell}\prod_{j=1}^\nu\int_0^t \partial_R^{h_j}\mathcal{L}(g(\tau))d\tau\right\|_{L^\infty}
\le C_{k+1},
\ea
\ee
by a finite number of recursive substitutions and the positive sign of $\cL (g (\tau))$. The $L^2$-based estimates are obtained by the same procedure, replacing the constant $C_{k+1}$ by $C_{k}=C_{k}(\|\bar g_0^{\alpha, \delta}\|_{\cH^{k}_R})$ in \eqref{eq:size-initial}. In fact, notice that $L^\infty$-based estimates lose one additional derivative with respect to the $L^2$-based estimates because of the embedding $H^1_R \hookrightarrow L^\infty$.
Finally, the proof of estimate \eqref{eq:est-g} follows the same lines using the explicit formula \eqref{eq:g-formula}.
\end{proof}

\begin{proof}[Proof of Proposition \ref{lem:35}]
 First, recalling that $\bar g_0^{\alpha, \delta}(R)\ge 0,$ the proof of the $\cH^k$ estimate \eqref{eq:est-psi2} on $\Ps$ follows directly from the upper bound on $\cL(\cdot)$ in \eqref{eq:Lg-bounds} and Lemma \ref{StimeLg}.
Next, we start with the estimate for $\et$. Taking the scalar product of the equation for $\et$ against $\et$, one gets
\be\ba
 \frac{d}{dt} \|\et\|_{L^2}^2 \le 2 \|\Ps\|_{L^\infty} \|\et\|_{L^2}^2 \le \frac{C_k}{\alpha}  \|\et\|_{L^2}^2,
\ea\ee
where we used the estimate of $\Ps$ in \eqref{eq:est-psi2}. About the $\cH^k$ estimate, notice by Leibniz formula that the worst term is given by
\begin{align*}
    \langle \de_t \partial^k \et,  \partial^k \et\rangle_{L^2} + \langle \de^k \Ps \de_\beta \et, \de^k \et \rangle_{L^2}= \alpha^{-1} \langle \de^k \cL(\Ome) \et, \de^k \et\rangle_{L^2}.
\end{align*}
Again, by means of the $\cH^k$ estimate of $\Ps$ in \eqref{eq:est-psi2}, one has that
\be
\frac{d}{dt} \|\et\|_{\cH^k}^2 \le 2 \|\Ps\|_{\cH^k} \|\et\|_{\cH^k}^2 \le \frac{C_k}{\alpha}  \|\et\|_{\cH^k}^2,
\ee
yielding the desired estimate for $\et$. 
Next, we deal with the estimates for $\csi$. Similarly to what has been done before, one has
\begin{align*}
    \frac{d}{dt}\|\csi\|^2_{\cH^k} \le 2\|  \Ps\|_{\cH^k}\|\csi\|_{\cH^k} (1+\|\csi\|_{\cH^k})\le \frac{C_{k}}{\alpha} \|\csi\|_{\cH^k} (1+\|\csi\|_{\cH^k}),
\end{align*}
where we used the $\cH_k$ estimate \eqref{eq:est-psi2} on $\Ps$.
This yields (simplifying the squares by a classical approximation procedure and Gronwall lemma)
%
%
the desired estimate for $\csi$. 
Finally, we deal with $\Ome$.
Recall from \eqref{eq:def-g} that $$\Ome=g + \bar \eta_0^{\alpha, \delta}(R) \int_0^t \exp \left(\frac{1}{2\alpha} \int_0^\tau \cL(g) \, d\mu \right) \, d\tau=g+\int_0^{t}\et(\tau)\, d\tau,$$ where $g$ satisfies \eqref{eq:g}.
First, recall the estimate $\|g(t)\|_{\cH^k} \le C_k $ in \eqref{eq:est-g}. By using that $\|\et\|_{\cH^k} \le C_k\exp\left({\frac{C_k}{\alpha}t}\right)$, one has that
\be\ba
\|\Ome (t) \|_{\cH^k} & \le \|g (t)\|_{\cH^k} + \int_0^t \|\et (\tau) \|_{\cH^k} \, d\tau \le C_k + \int_0^t C_k\exp\left({\frac{C_k}{\alpha} \tau}\right) \, d \tau  \le C_k + \alpha\left({\frac{C_k}{\alpha} t}\right).
\ea\ee
\end{proof}

\section{Remainder estimates}\label{sec:rem}

In this section, we derive suitable estimates for the remainder terms 
\begin{equation}\label{def:omega-r}
    \Omega_r:=\Omega-\Ome, \quad \eta_r:=\eta-\et, \quad \xi_r:=\xi-\csi,
\end{equation}
where 
$$(\Omega, \eta, \xi)=(\Omega (t, R, \beta),\eta (t, R, \beta),\xi (t, R, \beta))=(\omega (t, x, y), \de_x \rho (t, x, y), \de_y \rho (t, x, y))
$$ 
denotes a solution to the full system \eqref{eq:2Dbouss-grad} with initial data $(\omega_0^{\alpha, \delta}, \de_x \rho_0^{\alpha, \delta}, \de_y\rho_0^{\alpha, \delta})$ as in Theorem \ref{thm:main}
and $$(\Ome, \et, \csi)$$ denotes the solution to \eqref{eq:leading1} with the initial data $(\Omega_{0, \mathrm{app}}^{\alpha, \delta},\eta_{0, \mathrm{app}}^{\alpha, \delta},\xi_{0, \mathrm{app}}^{\alpha, \delta})$ as in Proposition \ref{prop:expl}.
The main result of this section is concerned with smallness of the remainders in terms of the $L^{\infty}$ norm.

\begin{Prop}\label{prop:rem}
  Let $N\geq 3$ and introduce 
    \begin{equation}\label{def:functional}F(t):=\|\Omega_{r}(t)\|_{\cH^N}+\|\eta_{r}(t)\|_{\cH^N}+\|\xi_{r}(t)\|_{\cH^N}.
    \end{equation}
    Then 
    \begin{equation*}
        F(t)\lesssim \sqrt{\alpha}
    \end{equation*}
    for all $0\leq t\leq T$ with $T=\frac{\alpha}{4C_{N+1}}\log|\log\alpha|$ where $C_{N+1}$  depends on the $\cH^{N+1}$-norm of the initial data $(\omega_0^{\alpha, \delta}, \de_x \rho_0^{\alpha, \delta}, \de_y\rho_0^{\alpha, \delta})$ as in \eqref{eq:size-initial}. In particular,
    \begin{equation}\label{eq:est-rim-Linfty}
        \|\Omega_r(t)\|_{L^{\infty}}+ \|\eta_r(t)\|_{L^{\infty}}+ \|\xi_r(t)\|_{L^{\infty}}\lesssim \sqrt{\alpha}
    \end{equation}
    for all $0\leq t\leq T$ with $T=\frac{\alpha}{4C_{N+1}}\log|\log\alpha|$.
\end{Prop}
Note that the $\cH^N$-norm of the initial data of the remainders is given by \eqref{eq:data eta rem} and \eqref{eq:data xi rem} and it is of order $F(0)\lesssim \sqrt{\alpha}$. 
While the system of equations satisfied by $(\Omega_r,\eta_r,\xi_r)$ is derived in Section \ref{sec:system rem}, the proof of Proposition \ref{prop:rem} is achieved by means of energy estimates in the Sobolev spaces $\cH^N$ in Section \ref{sec:est rem}.

For the stream function $\Psi$ as in \eqref{eq:ellittica}, we set the notation
\begin{align}\label{def:psi-r}
\Psi_r(\Omega):&=\Psi(\Omega)-\Ps(\Ome) = \hE(\Omega)+\mathcal{R}^\alpha(\Omega) + \Ps(\Omega_r)
\end{align}
where $\Ps(\Omega), \mathcal{R}^\alpha(\Omega)$ are defined in \eqref{eq:psiapp}-\eqref{def:R} respectively, while 
$$\hE(\Omega)=\Psi(\Omega)-\Psa(\Omega)=\Psi(\Omega)-\Ps(\Omega) -\mathcal{R}^\alpha (\Omega)$$
is given in \eqref{eq:psi-main}. 
\begin{Rmk}
Note that due to the explicit formula of $\Ome$, see \eqref{eq:def-g}, the approximate stream function $\Ps(\Ome)$ in \eqref{eq:leading1} only contains the odd component as in \eqref{eq:odd}. In contrast, $\Ps (\Omega_r)$ does not have any symmetry and it is therefore given by the full formula in \eqref{eq:psiapp}.
\end{Rmk}
%


%
\subsection{Elliptic estimates}
\noindent This section provides estimates of $\Psi$ solving the elliptic equation \eqref{eq:ellittica},
in terms of its decomposition in $\Ps, \Psi_2, \Psi_\mathrm{err}$ as given in Theorem \ref{thm:tarek} due to Elgindi \cite{tarek1}, see also \cite{drivas}.

\begin{Prop}[Elliptic estimates for $\Ps(\Omega), \Psa(\Omega)$, $\hE(\Omega)$ and $\mathcal{R}(\Omega)$]\label{lam:Rest}
The following estimates hold for the regular part:
\begin{align}
\label{eq:elliptic1-rem}
\|\de_{\beta \beta} (\hE-\mathbb{P}_0(\Psi)-\mathbb{P}_1(\Psi)) \|_{H^k}&+\alpha \|R \de_{R\beta} (\hE-\mathbb{P}_0(\Psi)-\mathbb{P}_1(\Psi))\|_{H^k}\notag\\
& + \alpha^2 \|R^2 \de_{RR} (\hE-\mathbb{P}_0(\Psi)-\mathbb{P}_1(\Psi)) \|_{H^k} \lesssim  \|(\Omega - \Omega_2)\|_{{H}^{k}}.
\end{align}
The following estimates hold for the singular part: 
\begin{align}\label{eq:crucial}
{\alpha}\|\de_{\beta \beta} \Psa \|_{H^k}+\alpha \|R \de_{R\beta} \Psa\|_{H^k} + {\alpha^2} \|R^2 \de_{RR} \Psa\|_{H^k}  &\lesssim\|\mathbb{P}_2(\Omega)\|_{{H}^{k}},\\
\|\mathbb{P}_0(\Psi)\|_{H^k}+ \alpha \|R \de_R \mathbb{P}_0(\Psi)\|_{H^k}+\alpha^2 \|R^2 \de_{RR} \mathbb{P}_0(\Psi)\|_{H^k} & \lesssim \|\mathbb{P}_0(\Omega)\|_{H^k}, \label{eq:est-P0} \\
\|\mathbb{P}_1(\Psi)\|_{H^k}+ \alpha \|R \de_R \mathbb{P}_1(\Psi)\|_{H^k}+\alpha^2 \|R^2 \de_{RR} \mathbb{P}_1(\Psi)\|_{H^k} & \lesssim \|\mathbb{P}_1(\Omega)\|_{H^k}.\label{eq:est-P1}
\end{align}
The following estimates hold for the remainder of the singular part:
\begin{align}\label{eq:est-rem}
\|\de_{\beta \beta} \mathcal{R} \|_{H^k}+{\alpha} \|R \de_{R\beta} \mathcal{R}\|_{H^k} + {\alpha^2}\|R^2 \de_{RR} \mathcal{R} \|_{H^k}  \lesssim\|\mathbb{P}_2(\Omega)\|_{{H}^{k}}.
\end{align}
Finally, all the above estimates are still valid if $H^k$ is replaced by $L^2(R^k \de_R^k)$ and $\lesssim$ stands for $\le M$ where $M>0$ is independent of $\alpha$.
\end{Prop}
\begin{proof}
The estimates \eqref{eq:elliptic1-rem} are an adaptation of Lemma \ref{lem:psierr}.  
We now deal with the crucial estimates \eqref{eq:crucial}. From the explicit formula for $\Psa (\Omega)(R, \beta)$ in \eqref{eq:psi2}, exploiting a key cancellation, we obtain that
\begin{align*}
    \de_R \Psa (\Omega)(R, \beta) = -\frac{1}{\pi\alpha^2} \sin (2\beta) R^{-\frac 4 \alpha -1} \int_0^R \int_0^{2\pi} s^\frac{4}{\alpha} \frac{\Omega (s, \beta) \sin (2\beta)}{s} \, d\beta \, ds. 
\end{align*}
This way
\begin{align}
     R\de_{\beta R} \Psa (\Omega)(R, \beta) = -\frac{2}{\pi\alpha^2} \cos (2\beta) R^{-\frac 4 \alpha -1} \int_0^R \int_0^{2\pi} s^\frac{4}{\alpha} \frac{\Omega (s, \beta) \sin (2\beta)}{s} \, d\beta \, ds = - \frac{2}{\alpha} \cos(2\beta) \mathcal{R}^\alpha (\Omega). 
\end{align}
From \eqref{eq:hardy}, we obtain that $\alpha \|R\de_{\beta R} \Psa\|_{L^2} \lesssim  \|\Omega\|_{L^2}.$
Next
\begin{align*}
    \de_{RR} \Psa (\Omega)(R, \beta)&= \frac{1}{\pi \alpha^2} \left(\frac{4}{\alpha}+1\right) \sin (2\beta) R^{-\frac{4}{\alpha}-2} \int_0^R \int_0^{2\pi} s^\frac{4}{\alpha} \frac{\Omega (s, \beta) \sin (2\beta)}{s} \, d\beta \, ds \\
    &\quad - \frac{1}{\pi \alpha^2} \sin (2\beta) R^{-2} \int_0^{2\pi} \Omega (s, \beta) \sin (2\beta) \, d\beta,
\end{align*}
yielding 
\begin{align*}
    R^2\de_{RR} \Psa (\Omega)(R, \beta)&= \left(\frac{4}{\alpha}+1\right) \frac{1}{\alpha} \sin (2\beta) \mathcal{R}^\alpha (\Omega) - \frac{1}{\pi \alpha^2} \sin (2\beta) \int_0^{2\pi} \Omega (s, \beta) \sin (2\beta) \, d\beta,
\end{align*}
from which, using \eqref{eq:hardy} again,
\begin{align*}
    \alpha^2\|\de_{RR}\Psa \|_{L^2} \lesssim \|\Omega\|_{L^2}. 
\end{align*}
Altogether, we obtain \eqref{eq:crucial}. The proof of the estimates for $\mathcal{R}^\alpha$ in \eqref{eq:est-rem} follows the same lines and therefore we omit them. 

Finally, let us look at the equation satisfied by $\Psi_0=\mathbb{P}_0(\Psi),$
\begin{align*}
    4 \Psi_0 + 4\alpha R\de_R \Psi_0 + \alpha^2 R^2 \de_{RR}\Psi_0 = - \mathbb{P}_0(\Omega). 
\end{align*}
Taking the scalar product with $R \de_R \Psi_0$ and integrating by parts gives
\begin{align*}
    \alpha \|R \de_R \Psi_0\|_{L^2} \lesssim \|\mathbb{P}_0(\Omega)\|_{L^2}. 
\end{align*}
Next, multiplying the equation by $\Psi_0$ and $R^2 \de_{RR} \Psi_0$, we obtain the desired estimates. The same procedure applies to $\mathbb{P}_1(\Psi)$. 
The proof is concluded.
\end{proof}

We recall the following embedding property for $\cH^k$ proven in \cite[Lemma 5.1]{tarek2}.
\begin{lem}[{\cite[Lemma 5.1]{tarek2}}]\label{lem:embedding}
    Let $k \in \N$, then there exists $M_k>0$, independent of $\alpha$, such that
    \begin{equation}\label{est:embed1}
    \|R^m\de_R^m\de_\beta^nf\|_{L^{\infty}}\leq M_k \|f\|_{\cH^{m+n+2}}
    \end{equation}
    for all $m+n+2\leq k$ and $f\in \cH^k$.
\end{lem}
We derive elliptic estimates for the stream function $\Psi_r$ defined \eqref{def:psi-r} in the spirit of Proposition \ref{lam:Rest}.
\begin{lem}\label{lem:Psir}
Let $k,m, N\in \N$. With the definitions of $\Psi_r, \Ps(\Omega_r)$ in \eqref{def:psi-r}, it follows that
\begin{equation}\label{eq:est-Psir-sing-infty}
\left\|\de_R^k\de_\beta^{m}\Ps(\Omega_r)\right\|_{L^{\infty}}+\left\|R^k\de_R^k\de_\beta^{m}\Ps(\Omega_r)\right\|_{L^{\infty}}\leq \frac{M_k}{\alpha}\|\mathbb{P}_2(\Omega_r)\|_{\cH^{k+m+1}}
\end{equation}
for all $k,m\in \N$ with $k+m+1\leq N$, where $M_k>0$ is independent of $\alpha$.
Moreover, 
\begin{equation}\label{eq:est-Psir-beta}
\|\Psi_r\|_{\cH^{k}}+\|\de_{\beta\beta}\Psi_r\|_{\cH^k}+\|\de_\beta\Psi_r\|_{\cH^k}+\alpha\|R\de_R\Psi_r\|_{\cH^k}\lesssim \left(C_k+\alpha\exp\left({\frac{C_k}{\alpha}t}\right)+\|\Omega_r\|_{\cH^k}+\frac{\|\mathbb{P}_2(\Omega_r)\|_{\cH^k}}{\alpha}\right),
\end{equation}
where the symbol $\lesssim$ stands for $\le M$, where $M>0$ is independent of $\alpha$, while $C_k$ is defined in \eqref{eq:size-initial}.
\end{lem}
Note that the elliptic estimates of Lemma \ref{lam:Rest} do not allow for an estimate of the type $\|\Psi(\Omega)\|_{\cH^{k}}\leq M_{k}\|\Omega\|_{\cH^{k-1}}$ as the estimates involving $R\de_R$ are not uniform in $\alpha$, even for the remainder terms $\mathcal{R}^{\alpha}(\Omega), \hE(\Omega)$. 
The case $m=0$ is not included in the above $L^{\infty}$-estimate, such an estimate is not required for our analysis.  

\begin{proof}
The desired estimates follow by combining Proposition \ref{lam:Rest} with the embeddings of Lemma \ref{lem:embedding}. Recall from \eqref{def:psi-r} that $\Psi_r=(\hE+\mathcal{R}^\alpha)(\Omega)+\Ps (\Omega_r)$. 
The estimates for $\mathcal{R}(\Omega)$ and $\hE(\Omega)$ follow from \eqref{eq:elliptic1-rem} and \eqref{eq:est-P0}-\eqref{eq:est-rem}, \eqref{est:embed1} and the bound 
\eqref{eq:est-LOM}, namely that
\begin{equation*}
    \|\Ome\|_{\cH^k}\lesssim  \left(C_k+\alpha\exp\left({\frac{C_k}{\alpha}t}\right)\right),
\end{equation*}
where  $C_k=C_k(\|\overline{g}^{\alpha, \delta}_0\|_{\cH^k}, \|\overline{\eta}^{\alpha, \delta}_0\|_{\cH^k})>0$ depends on the initial data as in \eqref{eq:size-initial}.
Concerning \eqref{eq:est-Psir-sing-infty}, it follows from \eqref{eq:crucial},\eqref{eq:est-rem} and \eqref{est:embed1} that 
\begin{equation*}
\left\|\de_R^k\de_\beta^{m}\Ps(\Omega_r)\right\|_{L^{\infty}}+\left\|R^k\de_R^k\de_\beta^{m}\Ps(\Omega_r)\right\|_{L^{\infty}}\leq \|\Ps(\Omega_r)\|_{\cH^{k+m+2}}
\leq \frac{M_k}{\alpha}\|\mathbb{P}_2(\Omega_r)\|_{\cH^{k+m+1}} \leq \frac{M_k}{\alpha}\|\mathbb{P}_2(\Omega_r)\|_{\cH^{N}}
\end{equation*}
for all $k,m\in \N$ with $k+m+1\leq N$. 
\end{proof}


\subsection{System of remainders}\label{sec:system rem}

We derive the system satisfied by $(\Omega_r, \eta_r, \xi_r)$, in \eqref{def:omega-r}, from \eqref{eq:2Dbouss-grad} and 
\eqref{eq:leading1}, namely
\be
\ba
\de_t \Omega_r & + (-\alpha R \de_\beta (\Psi_r+\Psi_{\mathrm{app}})) \de_R (\Omega_r+\Omega_{\mathrm{app}}) + (2\Psi_{\mathrm{app}} \de_\beta \Omega_r + 2 \Psi_r \de_\beta \Omega_r + 2 \Psi_r \de_\beta \Omega_{\mathrm{app}}) \\
& + (\alpha R(\de_R \Psi_r + \de_R \Psi_{\mathrm{app}})) (\de_\beta \Omega_{\mathrm{app}} + \de_\beta \Omega_r)=\eta_r, \\\\
\de_t \eta_r  & + (-\alpha R \de_\beta (\Psi_r+\Psi_{\mathrm{app}})) \de_R (\eta_r+\eta_{\mathrm{app}}) + (2\Psi_{\mathrm{app}} \de_\beta \eta_r + 2 \Psi_r \de_\beta \eta_r + 2 \Psi_r \de_\beta \eta_{\mathrm{app}}) \\
& + (\alpha R(\de_R \Psi_r + \de_R \Psi_{\mathrm{app}})) (\de_\beta \eta_{\mathrm{app}} + \de_\beta \eta_r)\\
&=\de_x u_2 (1-\csi-\xi_r) - \de_x u_1(\et +\eta_r) - \frac{ \cL(\Omega_{\mathrm{app}}) }{2\alpha}\eta_{\mathrm{app}}=: \text{(RHS)}_\eta , \\\\
\de_t \xi_r  & + (-\alpha R \de_\beta (\Psi_r+\Psi_{\mathrm{app}})) \de_R (\csi+\xi_r) + (2\Psi_{\mathrm{app}} \de_\beta \xi_r + 2 \Psi_r \de_\beta \xi_r + 2 \Psi_r \de_\beta\csi \\
& + (\alpha R(\de_R \Psi_r + \de_R \Psi_{\mathrm{app}})) \de_\beta(\csi+ \xi_r)=\de_x u_1 (\csi-1+\xi_r) - (\de_y u_1)(\et+\eta_r)\\
&-\frac{ \cL(\Omega_{\mathrm{app}}) }{2\alpha}(1-\csi)=: \text{(RHS)}_\xi.
\ea
\ee
From now on, the shortened notation $\Ps$ stands for $\Ps (\Ome)$ as in \eqref{eq:odd}, unless differently specified. We write the right-hand sides (RHS)$_\eta$, (RHS)$_\xi$ of the last two equations more explicitly. We shall repeatedly use the trigonometric identities $\sin(2\beta)=2\sin(\beta)\cos(\beta)$ and $\cos(2\beta)=1-2\sin^2(\beta)$. 
To this end, notice from \eqref{eq:spatial-der} and \eqref{eq:u1}
that
\begin{equation*}
    \de_x u_1=-\cos(2\beta)\de_\beta\Psi+\frac12\sin(2\beta)\de_{\beta\beta}\Psi-\left(1+\frac{\alpha}{2}\right)\sin(2\beta)\alpha R\de_R\Psi-\frac12\sin(2\beta)(\alpha R)^2\de_{RR}\Psi-\cos(2\beta)\alpha R\de_{\beta R}\Psi
\end{equation*}
so that 
\be\label{def:dexu1}
\ba
-(\de_x u_1)(\eta_{\mathrm{app}}+\eta_r)& - \frac{ \cL(\Omega_{\mathrm{app}}) }{2\alpha}\eta_{\mathrm{app}} =\left(\cos(2\beta)\de_\beta\Ps-\frac12\sin(2\beta)\de_{\beta\beta}\Ps\right)\left(\et+\eta_r\right)- \frac{ \cL(\Omega_{\mathrm{app}}) }{2\alpha}\eta_{\mathrm{app}} \\
&-\Big[-\cos(2\beta)\de_\beta\Psi_r+\frac12\sin(2\beta)\de_{\beta\beta}\Psi_r-\left(1+\frac{\alpha}{2}\right)\sin(2\beta)\alpha R\de_R(\Ps+\Psi_r)\\
&-\frac12\sin(2\beta)(\alpha R)^2\de_{RR}(\Ps+\Psi_r)-\cos(2\beta)\alpha R\de_{\beta R}(\Ps+\Psi_r)\Big]\left(\et+\eta_r\right)
\ea
\ee
%
Substituting the expression of $\Ps$ from \eqref{eq:odd} into the terms involving $\Psi_{\mathrm{app}}$ in the first line on the RHS yields
\be
\ba
\left[\cos(2\beta) \de_\beta \Psi_{\mathrm{app}} - \frac 12 \sin(2\beta) \de_{\beta \beta} \Psi_{\mathrm{app}}\right](\eta_{\mathrm{app}}+\eta_r)
&= \left[\cos^2(2\beta)\frac{\cL(\Omega_{\mathrm{app}}) }{2\alpha}+ \sin^2(2\beta)\frac{\cL(\Omega_{\mathrm{app}}) }{2\alpha} \right](\eta_r+\eta_{\mathrm{app}})\\
&= \frac{\cL(\Omega_{\mathrm{app}}) }{2\alpha}(\eta_r+\eta_{\mathrm{app}}).
\ea\notag
\ee
This leads to 
\be
\ba
-(\de_x u_1)(\eta_{\mathrm{app}}+\eta_r)& - \frac{ \cL(\Omega_{\mathrm{app}}) }{2\alpha}\eta_{\mathrm{app}} =\frac{\cL(\Omega_{\mathrm{app}}) }{2\alpha}\eta_r
\\
&-\Big[-\cos(2\beta)\de_\beta\Psi_{r}+\frac12\sin(2\beta)\de_{\beta\beta}\Psi_{r}-\left(1+\frac{\alpha}{2}\right)\sin(2\beta)\alpha R\de_R(\Ps+\Psi_r)\\
&-\frac12\sin(2\beta)(\alpha R)^2\de_{RR}(\Ps+\Psi_r)-\cos(2\beta)\alpha R\de_{\beta R}(\Ps+\Psi_r)\Big]\left(\et+\eta_r\right).
\ea\notag
\ee
%
Similarly, one computes
\be\notag
\ba
\de_x u_2&={2 \Psi}-\sin(2\beta)\de_\beta\Psi+\sin^2(\beta)\de_{\beta\beta}\Psi+\left(2+\cos(2\beta)+\alpha\cos^2\beta\right)\alpha R\de_R\Psi+\cos^2\beta(\alpha R)^2\de_{RR}\Psi\\
&\quad -\sin(2\beta)\alpha R\de_{\beta R}\Psi.
\ea
\ee

Let us look at the leading order terms. For $\Ps$ given in \eqref{eq:odd} it holds
\begin{equation*}
    {2 \Ps}-\sin(2\beta)\de_\beta\Ps+\sin^2(\beta)\de_{\beta\beta}\Ps=0
\end{equation*}
Plugging it inside the leading order terms above yields 
\be\notag\ba
2 \Psi -\sin(2\beta) \de_\beta \Psi+ \sin^2\beta \de_{\beta \beta}\Psi&= 2 \Psi_r - \sin(2\beta) \de_\beta \Psi_r + (\sin^2\beta) \de_{\beta \beta} \Psi_r.
\ea\ee
Altogether, we have
\be
\ba
\text{(RHS)}_\eta&=\frac{\cL(\Omega_{\mathrm{app}}) }{\alpha} \eta_r + \left[\cos(2\beta) \de_\beta \Psi_r - \frac 1 2 \sin(2 \beta) \de_{\beta \beta} \Psi_r\right](\eta_r+\eta_{\mathrm{app}})\\
& \quad +  [(\alpha R \sin(2\beta)+ \frac{\alpha^2 R}{2} \sin(2\beta))(\de_R \Psi_{\mathrm{app}} + \de_R \Psi_r) \\
&\quad + \alpha R \cos (2\beta) (\de_{R\beta} \Psi_{\mathrm{app}} + \de_{R\beta} \Psi_r) + \frac{\alpha^2 R^2}{2} \sin(2\beta)(\de_{RR}\Psi_{\mathrm{app}}+{\de_{RR}\Psi_r})](\eta_r+\eta_{\mathrm{app}})\\
&\quad + \left[2 \Psi_r - \sin(2\beta) \de_\beta \Psi_r + (\sin^2\beta) \de_{\beta \beta} \Psi_r\right](1-\csi-\xi_r). 
\ea
\ee
It remains to rewrite $\text{(RHS)}_\xi$. We have that
\be
\ba
\de_x u_1 (\xi-1)&=  -\frac{\cL(\Omega_{\mathrm{app}}) }{2\alpha}(\xi-1)\\
&\quad  -  [(\alpha R \sin(2\beta)+ \frac{\alpha^2 R}{2} \sin(2\beta))(\de_R \Psi_{\mathrm{app}} + \de_R \Psi_r)  \\
&\quad + \alpha R \cos (2\beta) (\de_{R\beta} \Psi_{\mathrm{app}} + \de_{R\beta} \Psi_r) + \frac{\alpha^2 R^2}{2} \sin(2\beta)(\de_{RR}\Psi_{\mathrm{app}}+\de_{RR}\Psi_r)\\
&\quad + \cos(2\beta) \de_\beta \Psi_r -\frac{1}{2}\sin(2\beta) \de_{\beta \beta} \Psi_r ](\xi-1).
\ea 
\ee
Finally, we need the expression of $\de_y u_1$. 
 %
In polar coordinates $(R, \beta)$, from the elliptic equation for $\Psi$ in \eqref{eq:psi-main}, we have
\be
\ba
\de_y u_1 & = {-2 \Psi} + \alpha (2+\alpha) R (\cos^2\beta) \de_R \Psi - \alpha (\alpha + 3) R \de_R \Psi {- \sin(2\beta) \de_\beta \Psi} - \alpha R  \sin(2\beta)  \de_{R \beta} \Psi \\
&\quad - (\alpha R)^2 (\sin^2\beta) \de_{RR} \Psi {- (\cos^2\beta) \de_{\beta \beta} \Psi}.
\ea
\ee
Using the decomposition $\Psi=\Psi_2+\Psi_r$, we have
\be
\ba
\text{(RHS)}_\xi&=  -\frac{\cL(\Omega_{\mathrm{app}}) }{2\alpha}\xi_r - \left[(1-2\sin^2\beta) \de_\beta \Psi_r-\frac 12 \sin (2\beta) \de_{\beta \beta} \Psi_r\right] (\xi_r+\csi-1)\\
&\quad -  [(\alpha R \sin(2\beta)+ \frac{\alpha^2 R}{2} \sin(2\beta) )(\de_R \Psi_{\mathrm{app}} + \de_R \Psi_r)  \\
&\quad + \alpha R \cos (2\beta) (\de_{R\beta} \Psi_{\mathrm{app}} + \de_{R\beta} \Psi_r) + \frac{\alpha^2 R^2}{2} \sin(2\beta) (\de_{RR}\Psi_{\mathrm{app}}+\de_{RR}\Psi_r) ](\xi_r+\csi-1)\\
&\quad + [2 \Psi_r + \sin(2\beta) \de_\beta \Psi_r + (\cos^2\beta) \de_{\beta \beta} \Psi_r](\eta_{\mathrm{app}}+\eta_r).
\ea
\ee
We can finally write down the system for the remainder terms as follows:
\be
\ba
\de_t \Omega_r & + (-\alpha R \de_\beta (\Psi_r+\Psi_{\mathrm{app}})) \de_R (\Omega_r+\Omega_{\mathrm{app}}) + (2\Psi_{\mathrm{app}} \de_\beta \Omega_r + 2 \Psi_r \de_\beta \Omega_r + 2 \Psi_r \de_\beta \Omega_{\mathrm{app}}) \\
& + (\alpha R(\de_R \Psi_r + \de_R \Psi_{\mathrm{app}})) (\de_\beta \Omega_{\mathrm{app}} + \de_\beta \Omega_r)=\eta_r; \\\\
\de_t \eta_r  & + (-\alpha R \de_\beta (\Psi_r+\Psi_{\mathrm{app}})) \de_R (\eta_r+\eta_{\mathrm{app}}) + (2\Psi_{\mathrm{app}} \de_\beta \eta_r + 2 \Psi_r \de_\beta \eta_r + 2 \Psi_r \de_\beta \eta_{\mathrm{app}}) \\
& + (\alpha R(\de_R \Psi_r + \de_R \Psi_{\mathrm{app}})) (\de_\beta \eta_{\mathrm{app}} + \de_\beta \eta_r)\\&= \text{(RHS)}_\eta =\frac{\cL(\Omega_{\mathrm{app}}) }{2\alpha} \eta_r + \left[\cos(2\beta) \de_\beta \Psi_r - \frac 1 2 \sin(2 \beta) \de_{\beta \beta} \Psi_r\right](\eta_r+\eta_{\mathrm{app}})\\
& \quad +  [(\alpha R \sin(2\beta)+ \frac{\alpha^2 R}{2} \sin(2\beta) )(\de_R \Psi_{\mathrm{app}} + \de_R \Psi_r) \\
&\quad + \alpha R \cos (2\beta) (\de_{R\beta} \Psi_{\mathrm{app}} + \de_{R\beta} \Psi_r) + \frac{\alpha^2 R^2}{2} \sin(2\beta) (\de_{RR}\Psi_{\mathrm{app}}+\de_{RR}\Psi_r)](\eta_r+\eta_{\mathrm{app}})\\
&\quad + [2 \Psi_r - \sin(2\beta) \de_\beta \Psi_r + (\sin^2\beta) \de_{\beta \beta} \Psi_r](1-\xi_r-\csi); \\\\
\de_t \xi_r  & + (-\alpha R \de_\beta (\Psi_r+\Psi_{\mathrm{app}})) \de_R (\xi_r+\csi)  + (2\Psi_{\mathrm{app}} \de_\beta \xi_r + 2 \Psi_r \de_\beta \xi_r+2 \Psi_r \de_\beta \csi) \\
& + (\alpha R(\de_R \Psi_r + \de_R \Psi_{\mathrm{app}})) \de_\beta (\xi_r+\csi)\\&=\text{(RHS)}_\xi=-\frac{\cL(\Omega_{\mathrm{app}}) }{2\alpha}\xi_r- \left[\cos (2\beta) \de_\beta \Psi_r-\frac 12 \sin (2\beta) \de_{\beta \beta} \Psi_r \right](\xi_r+\csi-1)\\
&\quad -  [(\alpha R \sin(2\beta)+ \frac{\alpha^2 R}{2} \sin(2\beta) )(\de_R \Psi_{\mathrm{app}} + \de_R \Psi_r)  \\
&\quad + \alpha R \cos (2\beta) (\de_{R\beta} \Psi_{\mathrm{app}} + \de_{R\beta} \Psi_r) + \frac{\alpha^2 R^2}{2} \sin(2\beta) (\de_{RR}\Psi_{\mathrm{app}}+\de_{RR}\Psi_r) ](\xi_r+\csi-1)\\
&\quad + [2 \Psi_r + \sin(2\beta) \de_\beta \Psi_r + (\cos^2\beta) \de_{\beta \beta} \Psi_r](\eta_{\mathrm{app}}+\eta_r).
\ea
\ee
In order to lighten the notation, we introduce the transport operator $\cT$ acting on a function $f(R, \beta)=f=f_r+\fe$ as follows 
\be\label{eq:transport terms}
\ba
\cT f(R, \beta) := (-\alpha R \de_\beta (\Psi_r+\Psi_{\mathrm{app}})) \de_R (f_r+\fe) &+ (2\Psi_{\mathrm{app}} \de_\beta f_r + 2 \Psi_r \de_\beta f_r + 2 \Psi_r \de_\beta \fe) \\
& + (\alpha R(\de_R \Psi_r + \de_R \Psi_{\mathrm{app}})) (\de_\beta \fe + \de_\beta f_r).
\ea
\ee
 For further reference, we rely on the notation for the remainder estimates in \cite[Section 6]{tarek2}. The system for the remainder terms can then be written in the compact form below
\begin{align}
\de_t \Omega_r + \cT \Omega &=\eta_r;  \notag \\\notag\\
\de_t \eta_r   + \cT \eta&=\frac{\cL(\Omega_{\mathrm{app}}) }{2\alpha} \eta_r +\left[ \cos(2\beta) \de_\beta \Psi_r - \frac 1 2 \sin(2 \beta) \de_{\beta \beta} \Psi_r\right](\eta_r+\eta_{\mathrm{app}}) \; [=:\text{(RHS)}_\eta^1]\label{eq:rhseta1}\\
& \quad + [\mathrm{I}_4 + \mathrm{I}_6 + \mathrm{I}_7] (\eta_r+\eta_{\mathrm{app}}) \; [=:\text{(RHS)}_\eta^2]\nonumber \\
&\quad + [2 \Psi_r - \sin(2\beta) \de_\beta \Psi_r + (\sin^2\beta) \de_{\beta \beta} \Psi_r](1-\csi-\xi_r);  \; [=:\text{(RHS)}_\eta^3]  \nonumber\\\notag\\
\de_t \xi_r   + \cT \xi &=-\frac{\cL(\Omega_{\mathrm{app}}) }{2\alpha}\xi_r-  \left[\cos (2\beta) \de_\beta \Psi_r-\frac 12 \sin (2\beta) \de_{\beta \beta} \Psi_r \right] (\xi_r+\xi_\mathrm{app}-1)\; [=:\text{(RHS)}_\xi^1]\label{eq:rhsxi1}\\
&\quad -  [\mathrm{I}_4 + \mathrm{I}_6 + \mathrm{I}_7](\xi_r+\xi_\mathrm{app}-1)\; [=:\text{(RHS)}_\xi^2]\nonumber\\
&\quad + [2 \Psi_r + \sin(2\beta) \de_\beta \Psi_r + (\cos^2\beta) \de_{\beta \beta} \Psi_r](\eta_{\mathrm{app}}+\eta_r), \; [=:\text{(RHS)}_\xi^3]\label{eq:rhs-xir1}\nonumber,
\end{align}
where 
\be
\ba
\mathrm{I}_4&:= (\alpha R \sin(2\beta)+\frac{\alpha^2 R}{2} \sin(2\beta))(\de_R \Psi_{\mathrm{app}} + \de_R \Psi_r)=:\mathrm{I}_{4,1}+\mathrm{I}_{4,2}, \\
\mathrm{I}_6&:=\alpha R \cos (2\beta) (\de_{R\beta} \Psi_{\mathrm{app}} + \de_{R\beta} \Psi_r)=: \mathrm{I}_{6,1}+\mathrm{I}_{6,2}, \\
 \mathrm{I}_7&:=\frac{\alpha^2 R^2}{2}\sin(2\beta) (\de_{RR}\Psi_{\mathrm{app}}+\de_{RR}\Psi_r)=: \mathrm{I}_{7,1}+\mathrm{I}_{7,2}.
\ea
\ee

\subsection{Estimates for remainder terms}\label{sec:est rem}

In this section, it is shown that the remainders $(\Omega_r,\eta_r,\xi_r)$ are sufficiently small in $\cH^k$ on suitable time scales. To that end, we adapt and extend the strategy in to prove the remainder estimates in \cite[Section 6]{tarek2}, to the present setting. Note that here we are concerned with a system while a scalar equation is considered in \cite{tarek2}. Furthermore, additional estimates for $\Psi_r$ are required in view of the decomposition \eqref{def:psi-r} differing from the one in \cite{tarek2} and the structure of (RHS)$_\eta$, (RHS)$_\xi$.

The remainder estimates are proven in the Sobolev spaces $\cH^k$ and $\cW^{k,\infty}$ with weights tailored for that purpose as defined in \eqref{def:Hk} and \eqref{def:Wk}.

In addition to the elliptic estimates of Proposition \ref{lam:Rest} and Lemma \ref{lem:Psir}, we repeatedly rely on the bounds for solutions to the leading order model.
More precisely, \eqref{eq:est-LOM} and \eqref{eq:est-psi2} respectively yield for $f_{\mathrm{app}}=\et, \csi$ that 
\begin{equation}\label{eq:est-fapp}
     \|\Ome\|_{\cH^k}\leq C_k+ \alpha\exp\left({\frac{C_k}{\alpha}t}\right), \; \|\fe\|_{\cH^k}\leq C_k \exp\left({\frac{C_k}{\alpha}t}\right), \; \|\Ps (\Ome)\|_{\cW^{k,\infty}}+\|\Ps (\Ome)\|_{\cH^{k}} \le \frac{C_{k}}{\alpha},
\end{equation}
where $k\in \N$. We further observe that, in view of \eqref{eq:odd} and the elliptic estimates of Proposition \ref{lam:Rest},  one has 
\begin{equation}\label{est:Psiapp ell}
     \|\alpha R\de_R\de_\beta\Ps(\Ome)\|_{\cH^N}+\|\alpha R\de_R\Ps(\Ome)\|_{\cH^N}\le M \|\Ome\|_{\cH^N}\leq M\left(C_{N}+\alpha\exp\left({\frac{C_N}{\alpha}t}\right)\right).
\end{equation}

Next, we provide energy estimates for $(\Omega_r,\eta_r,\xi_r)$ as in \eqref{def:omega-r}.
To that end, we consider the various contributions separately and start by deriving the estimates for the transport terms $\mathcal{T}$ defined in \eqref{eq:transport terms}. 
\begin{lem}\label{lem:remT}
Let $N\geq3$. Then, for $f=\Omega,\eta,\xi$ with the usual decomposition $f=\fe+f_r$ it holds
\begin{align*}
     \left|\left\langle \mathcal{T}(f), f_r \right\rangle_{\cH^N}\right|
     &\leq C_{N+1}\exp\left({\frac{C_{N+1}}{\alpha}t}\right)\left(C_{N+1}+\alpha\exp\left({\frac{C_{N}}{\alpha}t}\right)+\frac{\|\mathbb{P}_2\Omega_r\|_{\cH^N}}{\alpha}\right)\|f_r\|_{\cH^N}\\
     &\quad +\left(\frac{C_{N+1}}{\alpha}+\alpha\exp\left({\frac{C_{N}}{\alpha}t}\right)+\frac{\|\mathbb{P}_2\Omega_r\|_{\cH^N}}{\alpha}\right)\|f_r\|_{\cH^N}^2.
\end{align*}
\end{lem}
\begin{Rmk}\label{rmk:largest}
We emphasize that the term $\frac{C_{N+1}}{\alpha}\exp\left({\frac{C_{N+1}}{\alpha}t}\right)\|\Omega_r\|_{\cH^N}\|f_r\|_{\cH^N}$ leads to the largest contribution in the final estimate of the remainder terms \eqref{eq:boot} and determines the choice of the $\log \log$ time scale in Proposition \ref{prop:rem}. This contribution stems from $2\Psi_r\de_\beta \fe$ in the transport terms \eqref{eq:transport terms}.
\end{Rmk}
\begin{proof}
Recalling \eqref{eq:transport terms}, we set 
\begin{equation*}
\ba
\cT f(R, \beta) &:= (-\alpha R \de_\beta (\Psi_r+\Ps)) \de_R (f_r+\fe)+ (2\Ps \de_\beta f_r + 2 \Psi_r \de_\beta f_r + 2 \Psi_r \de_\beta \fe) \\
&\quad + (\alpha R(\de_R \Psi_r + \de_R \Ps)) (\de_\beta \fe + \de_\beta f_r)\\&=:I_1+I_2+I_3.
\ea
\end{equation*}

In order to derive $\cH^N$-bounds, we first derive $H^N$-bounds and then consider the contribution of the weighted norms. 
We further split $I_1$ as 
\begin{equation*}
    I_1=-\alpha \de_\beta \Ps R\de_R \fe-\alpha \de_\beta \Ps R\de_R f_{r}-\alpha \de_\beta \Psi_r R\de_R \fe-\alpha \de_\beta \Psi_r R\de_R f_{r}=:I_{1,1}+I_{1,2}+I_{1,3}+I_{1,4}
\end{equation*}
and treat these terms separately. For the first contribution, one has 
\begin{equation*}
     \left|-\alpha \left\langle \de_\beta\Ps R\de_R \fe, f_r\right\rangle_{H^N}\right|
\leq C\alpha\|\de_\beta\Ps\|_{W^{N,\infty}}\|R\de_R\fe\|_{\cH^{N}}\|f_r\|_{\cH^{N}}\leq C_{N+1}^2\exp\left({\frac{C_{N+1}}{\alpha}t}\right)\|f_r\|_{\cH^N},
\end{equation*}
where we used \eqref{eq:est-fapp} to bound $\fe$ and $\Ps$. In the following, we only consider the contribution to the $H^N$ scalar-product with derivatives of order $N$ as the bounds for terms with lower order derivatives follow along the same lines. For $I_{1,2}$, we have    
\begin{equation*}
    \left\langle -\alpha\de^N\left( R\de_\beta \Ps\de_R f_{r}\right), \de^N f_r\right\rangle_{L^2}=\left\langle -\alpha \sum_{i=0}^N\de^i\de_\beta \Ps \de^{N-i}(R\de_R f_{r}), \de^N f_r\right\rangle_{L^2}.
\end{equation*}
If $i=0$, an integration by parts in $R$ yields
\begin{align*}
    \left|-\alpha\int\int\de_\beta\Ps \frac{R}{2}\de_R(\de^Nf_r)^2\right| & =\left|\alpha \int\int(\de_\beta\Ps+R\de_R\de_\beta \Ps)\frac{(\de^Nf_r)^2}{2}\right|
    \leq C_{2}\|f_r\|_{\cH^N}^2,
\end{align*}
upon using \eqref{eq:est-fapp}. For $1\leq i\leq N$, it follows
\begin{equation*}
   \left|\left\langle -\alpha \sum_{i=1}^N\de^i\de_\beta \Ps \de^{N-i}(R\de_R f_{r}), \de^N f_r\right\rangle_{L^2}\right|\lesssim \alpha \|\de_\beta\Ps\|_{\cW^{N,\infty}}\|f_r\|_{\cH^N}^2\leq C_{N+1}\|f_r\|_{\cH^N}^2.
\end{equation*}
Hence
\begin{equation*}
    \left|\left\langle I_{1,2}, f_r\right\rangle_{H^N}\right| \leq C_{N+1}\|f_r\|_{\cH^N}^2. 
\end{equation*}
For $I_{1,3}$, the contribution involving derivatives of order $N$ amounts to
\begin{equation*}
     \left|\left\langle\de^N\left(-\alpha \de_\beta \Psi_r R\de_R \fe\right),\de^N f_r\right\rangle_{L^2}\right|\leq \left\|\sum_{i=0}^N\de^i\left(\alpha\de_\beta\Psi_r\right)\de^{N-i}(R\de_R\fe)\right\|_{L^2}\|f_r\|_{H^N}.
\end{equation*}
If $0\leq i\leq \frac{N}{2}$, then \eqref{eq:est-Psir-sing-infty} yields
\begin{equation*}
    \left\|\de^i\left(\alpha\de_\beta\Psi_r\right)\de^{N-i}(R\de_R\fe)\right\|_{L^2}\leq \|\alpha\de^i\de_\beta\Psi_r\|_{L^{\infty}}\|\fe\|_{\cH^{N+1}}\leq C \left(\alpha\|\Ome\|_{\cH^N}+\|\Omega_{r}\|_{\cH^N}\right)\|\fe\|_{H^{N+1}}.
\end{equation*}
Similarly for $\frac{N}{2}< i \leq N$, \ref{lem:embedding} that
\begin{align*}
    \left\|\de^i\left(\alpha\de_\beta\Psi_r\right)\de^{N-i}(R\de_R\fe)\right\|_{L^2}
    &\leq \|\de^i\left(\alpha\de_\beta\Psi_r\right)\|_{L^2}\|\de^{N-i}(R\de_R\fe)\|_{L^{\infty}}\\&\leq C \left(\alpha\|\Ome\|_{\cH^N}+\|\Omega_{r}\|_{\cH^N}\right) \|\fe\|_{\cH^{N+1}}.
\end{align*}
Combining these estimates, we conclude that 
\begin{align*}
    \left|\left\langle I_{1,3},f_r\right\rangle_{H^N}\right|\leq C_{N+1}\exp\left({\frac{C_{N+1}}{\alpha}t}\right)\left(C_N\alpha+\alpha^2\exp\left({\frac{C_{N}}{\alpha}t}\right)+\|\Omega_r\|_{\cH^N}\right)\|f_r\|_{\cH^N}.
\end{align*}
The highest order term of $I_{1,4}$ is controlled in the same spirit of $I_{1,3}$. More precisely, we consider separately the cases $i=0$, $1\leq i\leq \frac{N}{2}$ and $\frac{N}{2}< i\leq N$, leading to
\begin{multline*}
\left\langle-\alpha \de^N\left(\de_\beta \Psi_r R\de_R f_{r}\right), \de^N f_r \right\rangle_{L^2}
=-\left\langle\alpha\left(\de_\beta \Psi_r \de^N(R\de_R f_{r})\right), \de^N f_r \right\rangle_{L^2}\\-\left\langle\sum_{i=1}^{\lfloor \frac{N}{2}\rfloor}\de^i\left(\alpha\de_\beta \Psi_r \right)\de^{N-i}(R\de_R f_{r}), \de^N f_r \right\rangle_{L^2}-\left\langle\sum_{i=\lceil\frac{N}{2}\rceil}^{N}\de^i\left(\alpha\de_\beta \Psi_r \right)\de^{N-i}(R\de_R f_{r}), \de^N f_r \right\rangle_{L^2}.
\end{multline*}
Arguing for each contribution similarly to the respective cases for $I_{1,3}$, namely integrating by parts in $R$ for $i=0$ and using the $L^{\infty}$-estimate for $\de_\beta\Psi_r$ for $i\leq \frac{N}{2}$ and respectively the $L^2$-estimate for $i>\frac{N}{2}$ provided by Lemma \ref{lem:Psir}, one has
\begin{equation*}
    \left|\left\langle-\alpha \de^N\left(\de_\beta \Psi_r R\de_R f_{r}\right), \de^N f_r \right\rangle_{L^2}\right|\leq \left(\alpha C_N+\alpha^2\exp\left({\frac{C_N}{\alpha}t}\right)+\|\Omega_r\|_{\cH^N}\right)\|f_r\|_{\cH^N}^2
\end{equation*}
and hence the bound for $I_{1,4}$ follows along the same lines. For the contributions involving weighted norms, the strategy to prove the remainder estimates in \cite[Section 6]{tarek2} can be adapted in the same spirit as for the estimates detailed above.
Finally,
\begin{align*}
    \left|\left\langle I_1, f_r\right\rangle_{\cH^N}\right|&\leq C_{N+1}\exp\left({\frac{C_{N+1}}{\alpha}t}\right)\left(C_{N+1}+\alpha C_N\exp\left({\frac{C_N}{\alpha}t}\right)+\|\Omega_r\|_{\cH^N}\right)\|f_r\|_{\cH^N}\\&\quad+\left(C_{N+1}+\alpha C_N\exp\left({\frac{C_N}{\alpha}t}\right)+\|\Omega_r\|_{\cH^N}\right)\|f_r\|_{\cH^N}^2.
\end{align*}
Next, we consider
\begin{equation*}
    I_2=2\Ps \de_\beta f_r + 2 \Psi_r \de_\beta f_r + 2 \Psi_r \de_\beta \fe.
\end{equation*}
The explicit estimates for $\fe$ and $\Ps$, see \eqref{eq:est-fapp} and \eqref{est:Psiapp ell} allow one to bound
\begin{equation*}
    \left|\left\langle 2\Ps \de_\beta f_r,f_r \right\rangle_{\cH^N
    }\right|\leq \frac{C_{N+1}}{\alpha}\|f_r\|_{\cH^N}^2
\end{equation*}
by arguing as in the proof of $I_{1,2}$, namely considering the cases $i=0$, $i\leq \frac{N}{2}$ and $\frac{N}{2}<i\leq N$ and using the respective estimates from Lemma \ref{lem:Psir}. In the same spirit, we obtain
\begin{equation*}
     \left|\left\langle2 \Psi_r \de_\beta f_r,f_r\right\rangle_{\cH^N
    }\right|\leq C \|\de_\beta\Psi_r\|_{\cH^{N}}\|f_r\|_{\cH^N}^2\leq \left(C_N+\alpha\exp\left({\frac{C_N}{\alpha}t}\right)+\frac{\|\Omega_r\|_{\cH^N}}{\alpha}\right)\|f_r\|_{\cH^N}^2.
\end{equation*}
The last term of $I_2$ is bounded by
\begin{align*}
    \left|\left\langle2 \Psi_r \de_\beta \fe,f_r\right\rangle_{\cH^N}\right|&\leq C_{N+1}\exp\left({\frac{C_{N+1}}{\alpha}t}\right)\|\Psi_r\|_{
    \cH^N}\|f_r\|_{\cH^N}\\
    &\leq C_{N+1}\exp\left(^{\frac{C_{N+1}}{\alpha}t}\right)\left(C_N+\alpha\exp\left({\frac{C_{N}}{\alpha}t}\right)+\frac{\|\Omega_r\|_{\cH^N}}{\alpha}\right)\|f_r\|_{\cH^N}.
\end{align*}
Therefore, we conclude
\begin{align*}
    \left|\left\langle I_2, f_r\right\rangle_{\cH^N}\right|&\leq C_{N+1}\exp\left({\frac{C_{N+1}}{\alpha}t}\right)\left(C_N+\alpha\exp\left({\frac{C_{N}}{\alpha}t}\right)+\frac{\|\Omega_r\|_{\cH^N}}{\alpha}\right)\|f_r\|_{\cH^N}\\&\quad + \left(\frac{C_{N+1}}{\alpha}+C_N+\alpha\exp\left({\frac{C_{N}}{\alpha}t}\right)+\frac{\|\Omega_r\|_{\cH^N}}{\alpha}\right)\|f_r\|_{\cH^N}^2.
\end{align*}
Finally, it remains to bound 
\begin{equation*}
    I_3=\alpha R(\de_R \Psi_r + \de_R \Ps) (\de_\beta \fe + \de_\beta f_r).
\end{equation*}
Following the same strategy as for the previous terms and using the estimates \eqref{eq:est-Psir-sing-infty}-\eqref{eq:est-Psir-beta}, we infer 
\begin{equation*}
    \left|\left\langle \de^N\left(\alpha R\de_R\Psi_r\de_\beta\fe\right), \de^N f_r\right\rangle_{L^2}\right|
    \leq C_{N+1}\exp\left({\frac{C_{N+1}}{\alpha}t}\right)\left(C_N+\alpha\exp\left({\frac{C_{N}}{\alpha}t}\right)+\|\Omega_r\|_{\cH^N}\right)\|f_r\|_{\cH^N}.
\end{equation*}
Moreover, proceeding as for $I_{1,4}$ by separating the cases $i=0$, $1\leq i\leq N/2$ and $N/2< i\leq N$ and from the estimates of Lemma \ref{lem:embedding} as well as Lemma \ref{lem:Psir} we have
\begin{equation*}
     \left|\left\langle \de^N\left(\alpha R\de_R\Psi_r\de_\beta f_r\right), \de^N f_r\right\rangle_{L^2}\right|\leq \alpha\|R\de_R\de_\beta\Psi_r\|_{\cH^N}\|f_r\|_{\cH^N}^2\leq C\left(C_N+\alpha\exp\left({\frac{C_N}{\alpha}t}\right)+\|\Omega_r\|_{\cH^N}\right) \|f_r\|_{\cH^N}^2
\end{equation*}
Finally, we prove along the same lines and upon using \eqref{est:Psiapp ell}  that
\begin{equation*}
     \left|\left\langle \de^N\left(\alpha R\de_R\Ps\de_\beta\fe\right), \de^N f_r\right\rangle_{L^2}\right|\leq C_{N+1}\exp\left({\frac{C_{N+1}}{\alpha}t}\right)\left(C_{N}+\alpha\exp\left({\frac{C_{N}}{\alpha}t}\right)\right)\|f_r\|_{\cH^N},
\end{equation*}
and 
\begin{equation*}
    \left|\left\langle \de^N\left(\alpha R\de_R\Ps\de_\beta f_r\right), \de^N f_r\right\rangle_{L^2}\right|\leq \left(C_{N}+\alpha\exp\left({\frac{C_N}{\alpha}t}\right)\right)\|f_r\|_{\cH^N}^2.
\end{equation*}
We conclude that
\begin{align*}
    \left|\left\langle I_3, f_r\right\rangle_{\cH^N}\right|&\leq C_{N+1}\exp\left({\frac{C_{N+1}}{\alpha}t}\right)\left(C_N+\alpha\exp\left({\frac{C_{N}}{\alpha}t}\right)+\|\Omega_r\|_{H^N}\right)\|f_r\|_{\cH^N}\\
    &\quad + \left(C_N+\alpha\exp\left({\frac{C_N}{\alpha}t}\right)+\|\Omega\|_{\cH^N}\right)\|f_r\|_{\cH^N}^2, 
\end{align*}
where we have omitted the estimates on the weighted norms which can be recovered by combining the method detailed above and in to prove the remainder estimates in \cite[Section 6]{tarek2}.
Combing the estimates for $I_i$ with $i=1,2,3$ and up to modifying $C_{N+1}$ we conclude that 
\begin{align*}
     \left|\left\langle \mathcal{T}(f), f_r \right\rangle_{\cH^N}\right|&
     \leq C_{N+1}\exp\left({\frac{C_{N+1}}{\alpha}t}\right)\left(C_{N+1}+C_N+\alpha\exp\left({\frac{C_{N}}{\alpha}t}\right)+\frac{\|\Omega_r\|_{\cH^N}}{\alpha}\right)\|f_r\|_{\cH^N}\\
     &\quad +\left(\frac{C_{N+1}}{\alpha}+C_N+\alpha\exp\left({\frac{C_{N}}{\alpha}t}\right)+\frac{\|\Omega_r\|_{\cH^N}}{\alpha}\right)\|f_r\|_{\cH^N}^2.
\end{align*}
Observing that $C_N\leq C_{N+1}\leq \alpha^{-1}C_{N+1}$ yields the desired estimate. This completes the proof.
\end{proof}

Next we estimate the contributions stemming from $\mathrm{(RHS)}_\eta^1$ and $\mathrm{(RHS)}_\xi^1$ respectively. 

\begin{lem}\label{lem:RHS1}
    The following hold true for $\mathrm{(RHS)}_\eta^1$ and $ \mathrm{(RHS)}_\xi^1$ as defined in \eqref{eq:rhseta1} and \eqref{eq:rhsxi1} respectively,  
    \begin{align}
        \left\langle \mathrm{(RHS)}_\eta^1, \eta_r\right\rangle_{\cH^N}&\lesssim  \left(C_N+\alpha\exp\left({\frac{C_N}{\alpha} t}\right)+\frac{\|\Omega_r\|_{\cH^N}}{\alpha}\right)C_N\exp\left({\frac{C_N}{\alpha} t}\right){{\|\eta_r\|_{\cH^N}}}\notag\\
        &\quad +\left(\frac{C_N}{\alpha}+\alpha \exp\left({\frac{C_N}{\alpha} t}\right)+\frac{\|\Omega_r\|_{\cH^N}}{\alpha}\right)\|\eta_r\|_{\cH^N}^2
        \label{rhs_eta1},\\
        \left\langle \mathrm{(RHS)}_\xi^1, \xi_r\right\rangle_{\cH^N}&\lesssim \left(C_N+\alpha\exp\left({\frac{C_N}{\alpha} t}\right)+\frac{\|\Omega_r\|_{\cH^N}}{\alpha}\right)\left(C_N\exp\left({\frac{C_N}{\alpha} t}\right)+1\right){{\|\xi_r\|_{\cH^N}}}\notag\\
        &\quad+\left(\frac{C_N}{\alpha}+\alpha \exp\left({\frac{C_N}{\alpha} t}\right)+\frac{\|\Omega_r\|_{\cH^N}}{\alpha}\right)\|\xi_r\|_{\cH^N}^2
    \label{rhs_xi1}.
    \end{align}
\end{lem}
\begin{proof}
    We prove \eqref{rhs_eta1}, the other one follows similarly. By using Cauchy-Schwartz inequality we have
    \begin{align*}
        \left\langle\frac{\cL(\Ome)}{2\alpha}\eta_r,\eta_r\right\rangle_{\cH^N}&\leq  \alpha^{-1}\|\cL(\Ome)\eta_r\|_{\cH^N}\|\eta_r\|_{\cH^N}.
    \end{align*}
    We want to estimate the first factor in the r.h.s. of the above inequality. We recall that $\cL(\cdot)$ is a radial function, so that expanding the norms we obtain
    \begin{equation}\label{zucca}
\|\cL(\Ome)\eta_r\|_{\cH^N}\lesssim\sum_{m=0}^N\sum_{i=0}^m\sum_{j=0}^i\|\partial_R^j\cL(\Ome)\partial_R^{i-j}\partial_{\beta}^{m-i}\eta_r\|_{L^2}+\|R^j\partial_R^j\cL(\Ome)R^{i-j}\partial_R^{i-j}\partial_{\beta}^{m-i}\eta_r\|_{L^2}.
    \end{equation}
    When $j<N$ we can put $\partial^j_R\cL(\Ome)$ and $R^j\partial^j_R\cL(\Ome)$ in $L^{\infty}$, in such a way we bound from above \eqref{zucca} by $\|\eta_r\|_{\cH^N}\|\mathcal{L}(\Ome)\|_{\cH^N}$, after using the one dimensional (since $\mathcal{L}$ is radial) Sobolev embedding. The case $j=N$ reduces to $\|\partial_R^N\cL(\Ome)\eta_r\|_{L^2}+\|R^N\partial_R^N\cL(\Ome)\eta_r\|_{L^2}$, in this case we may put $\eta_r$ in $L^{\infty}$ and conclude again by (two dimensional) Sobolev embedding. By means of \eqref{pesate1}, we may therefore conclude that
    \begin{equation*}
    \left\langle\frac{\cL(\Ome)}{2\alpha}\eta_r,\eta_r\right\rangle_{\cH^N}\lesssim\alpha^{-1}\|\cL(\Ome)\|_{\cH^N}\|\eta_r\|_{\cH^N}^2\lesssim \frac{C_N}{\alpha}\|\eta_r\|^2_{\cH^N},
    \end{equation*}
where we used again that $\cL(\Ome)=\cL(g)$ for which one has the bounds of Lemma \ref{StimeLg}. Next, we consider $\mathbb{P}_2\Psi_r=\mathcal{R}(\Omega)+\Ps(\Omega_r)$. Note that 
\begin{equation*}
     \cos(2\beta) \de_\beta\left(\mathcal{R}(\Omega)+\Ps(\Omega_r)\right) - \frac 1 2 \sin(2 \beta) \de_{\beta \beta}\left(\mathcal{R}(\Omega)+\Ps(\Omega_r)\right)=2\mathcal{R}_{\sin}(\Omega)+\frac{\cL(\Omega_r)}{2\alpha},
\end{equation*}
where $\mathcal{R}_{\sin}(\Omega)=\left\langle \mathcal{R}(\Omega),\sin(2\beta)\right\rangle_{L_\beta^2}$. 
By arguing as before by means of Sobolev-embeddings, Lemma \ref{lem:Psir} and \eqref{eq:est-fapp}, one infers
\begin{equation*}
    \left\|\left(2\mathcal{R}_{\sin}(\Omega)+\frac{\cL(\Omega_r)}{2\alpha}\right)\left(\eta_r+\et\right)\right\|_{\cH^N}\leq \left(C_N+\alpha\exp\left(\frac{C_N}{\alpha}t\right)+\frac{\|\Omega_r\|_{\cH^N}}{\alpha}\right)\left(\|\eta_r\|_{\cH^N}+C_N\exp\left(\frac{C_N}{\alpha}t\right)\right),
\end{equation*}
from which the desired estimate follows. We turn to the remaining terms, first
$$\langle\sin(2\beta)(\partial_{\beta\beta}\hE(\Omega)\eta_r,\eta_{r}\rangle_{\cH^N}.$$ It is enough to estimate
    $\|\sin(2\beta)(\partial_{\beta\beta}\hE(\Omega)\eta_r\|_{\cH^N}$. We focus on $\|\partial^i(\partial_{\beta\beta}\hE(\Omega))\partial^{N-i}\eta_r\|_{L^2}$. When $i\geq 2$, we can put $\partial^{N-i}\eta_r$ in $L^{\infty}$ and then use Sobolev embedding in two dimensions to bound this term by $\|\eta_r\|_{\cH^N}$. We obtain
    \begin{equation}\label{margherita}
        \sum_{i=2}^N\|\partial^i(\partial_{\beta\beta}\hE(\Omega))\partial^{N-i}\eta_r\|_{L^2}\leq \|\partial_{\beta\beta}\hE(\Omega)\|_{\cH^N}\|\eta_r\|_{\cH^N}\leq  \left(C_N+\alpha \exp\left({\frac{C_N}{\alpha} t}\right)+\|\Omega_r\|_{\cH^N}\right)\|\eta_r\|_{\cH^N},
    \end{equation}
    where we used Proposition \ref{lam:Rest} and \eqref{eq:est-fapp} in the last step. In the cases $i=0,1$ we can put $\partial^i\partial_{\beta\beta}\hE(\Omega)$ in $L^{\infty}$, use Sobolev embedding and bound it by the same quantity in the r.h.s. of \eqref{margherita}, again using Proposition \ref{lam:Rest} at the level $k=2,3$. Similarly,
    \begin{equation*}
\|\sin(2\beta)\partial_{\beta\beta}\hE(\Omega)\eta_{\rm app}\|_{\cH^N}\lesssim \|\partial_{\beta\beta}\hE(\Omega)\|_{\cH^N}\|\eta_{\rm app}\|_{\cH^N}\lesssim
\left(C_N+\alpha\exp\left({\frac{C_N}{\alpha} t}\right)+\|\Omega_r\|_{\cH^N}\right) C_Ne^{\frac{C_N}{\alpha} t},
    \end{equation*}
    where in the last step we used \eqref{eq:est-LOM}.
    In view of \eqref{eq:est-Psir-beta}, the estimate for the term $\langle\cos(2\beta)\partial_{\beta}\Psi_r(\eta_r+\eta_{\rm app}),\eta_{r}\rangle_{\cH^N}$, may be obtained following word by word what done before. The last term is treated in a similar way.
The inequality for $\langle\mathrm{(RHS)}_\xi^1, \xi_r \rangle $ can be obtained in the same manner and therefore we omit it.
%
\end{proof}

We proceed to provide bounds for the terms involving $\text{(RHS)}_\eta^2, \text{(RHS)}_\xi^2$ appearing in the (RHS) of the equations for $\eta_r$ and $\xi_r$ respectively. 

\begin{lem}\label{lem:RHS2}
The following hold true for $\mathrm{(RHS)}_\eta^2$ and $ \mathrm{(RHS)}_\xi^2$:
    \begin{align*}
        \left\langle \mathrm{(RHS)}_\eta^2, \eta_r\right\rangle_{\cH^N}& \le \left(C_N+\alpha\exp\left({\frac{C_N}{\alpha}t}\right)\right)\left(C_N\exp\left({\frac{C_N}{\alpha}t}\right)+\|\Omega_r\|_{\cH^N}\right)\|\eta_r\|_{\cH^N}\\
        &\quad +\left(C_N+\alpha\exp\left({\frac{C_N}{\alpha}t}\right)+\|\Omega_r\|_{\cH^N}\right)\|\eta_r\|_{\cH^N}^2,\\
        \left\langle \mathrm{(RHS)}_\xi^2, \xi_r\right\rangle_{\cH^N}&\le \left(C_N+\alpha\exp\left({\frac{C_N}{\alpha}t}\right)\right)\left(C_N\exp\left({\frac{C_N}{\alpha}t}\right)+\|\Omega_r\|_{\cH^N}\right)\|\xi_r\|_{\cH^N}\\
        &\quad +\left(C_N+\alpha\exp\left({\frac{C_N}{\alpha}t}\right)+\|\Omega_r\|_{\cH^N}\right)\|\xi_r\|_{\cH^N}^2.
    \end{align*}
\end{lem}
The proof follows the same strategy as in Lemma \ref{lem:remT} for the transport terms, more specifically the one for $I_1$ and $I_3$ defined in the respective proof.

The terms $\mathrm{I}_j, \, j \in \{4, 6, 7\}$ are exactly the same terms (with the same notation) as in \cite[Section 6]{tarek2}, we adapt the proof to our setting. 

\begin{proof}
 We sketch the idea of the proof. The main point is to notice that, for any $j \in \{4,6,7\}$, the term $\mathrm{I}_j= \mathrm{I}_{j,1}+\mathrm{I}_{j,2}$ can be decomposed in the sum of $\mathrm{I}_{j,1}=\mathcal{F}^2_j(\Psi_{\mathrm{app}})$, depending only on $\Psi_{\mathrm{app}}$ in \eqref{eq:psiapp},
and $\mathrm{I}_{j,2}=\mathcal{F}^r_j(\Psi_r)$, depending only on the difference $\Psi_r$. We emphasize that as the decomposition \eqref{def:psi-r} of $\Psi_r$ differs from the one in \cite{tarek2}, the estimates for $\mathrm{I}_{j,2}$ need to be adapted. Appealing to Proposition \ref{lem:35} to bound $\Ps$ and the estimate for $\eta_{\mathrm{app}}$ given by Proposition \ref{lem:35}, one has that
\be
\ba
\left|(\mathrm{I}_{j,1}(\eta_r+\eta_{\mathrm{app}}), \eta_r)_{\cH^N} \right|& \lesssim  \alpha\left(\| R\de_R\Psi_{\mathrm{app}}+ R\de_{\beta R}\Psi_{\mathrm{app}}+\alpha R^2\de_{RR}\Psi_{\mathrm{app}}\|_{\cH^{N}}\right) (\|\eta_r \|_{\cH^N} +\|\eta_{\mathrm{app}}\|_{\cH^N})\|\eta_r \|_{\cH^N}\\
& \lesssim  \|\Ome\|_{\cH^N}\left(\|\eta_r\|_{\cH^N} +C_N\exp\left({\frac{C_N}{\alpha} t}\right)\right)\|\eta_r\|_{\cH^N}, 
\ea
\ee
where we proceeded like for $I_{1,2}$ in the proof of Lemma \ref{lem:remT} and in particular exploited \eqref{est:Psiapp ell}. 
Similarly,
\be
\ba
\left|(\mathrm{I}_{j,1}(\csi+\xi_r-1), \xi_r)_{\cH^N}\right| &  \lesssim  \|\Ome\|_{\cH^N}\left(1+\|\xi_r\|_{\cH^N} +C_N \exp\left({\frac{C_N}{\alpha} t}\right)\right)\|\xi_r\|_{\cH^N}. 
\ea
\ee
To control the contributions coming from $I_{j,2}$, we rely on the elliptic estimates of Lemma \ref{lem:Psir}. For the sake of a short exposition, we only consider the contribution of the $\cH^N$ scalar-product involving derivatives of $\de^N$. The lower order derivatives contributions are dealt with similarly. 
\be
\ba
(\de^N (\mathrm{I}_{j,2}(\eta_r+\eta_{\mathrm{app}})), \de^N \eta_r)_{L^2} & = \sum_{i=0}^N (\de^{i} (\mathrm{I}_{j,2})\de^{N-i}(\eta_r+\eta_{\mathrm{app}})), \de^N \eta_r)_{L^2} \\
&= \sum_{i=0}^{\frac{N}{2}} (\de^{i} (\mathrm{I}_{j,2}) \de^{N-i} (\eta_r+\eta_{\mathrm{app}}), \de^N \eta_r)_{L^2} =:\text{(A)}\\
&\quad + \sum_{k=\frac{N}{2}}^N (\de^{N-i} (\mathrm{I}_{j,2}) \de^k (\eta_r+\eta_{\mathrm{app}}), \de^N \eta_r)_{L^2}=:\text{(B)},
\ea
\ee
where
\be\ba
\text{(A)} & \le \sum_{i=0}^{\lfloor \frac{N}{2}\rfloor}  \|\de^{i} (\mathrm{I}_{j,2})\|_{L^\infty} \|\de^{N-i} (\eta_r+\eta_{\mathrm{app}})\|_{L^2} \|\eta_r\|_{H^N}  \le C \|\mathrm{I}_{j,2}\|_{H^N} \|\eta_r+\eta_{\mathrm{app}}\|_{H^N} \|\eta_r\|_{H^N},
\ea\ee
where we used that $\lfloor\frac{N}{2}\rfloor+2\leq N$ for {$N \ge 3$} and the $2D$ Sobolev embedding $H^2 \hookrightarrow L^{\infty}$, see also Lemma \ref{lem:embedding}. Similarly
\be\ba
\text{(B)} & \le  \sum_{i=\lceil\frac{N}{2}\rceil}^N \|\de^{i} (\mathrm{I}_{j,2})\|_{L^2} \| \de^{N-i} (\eta_r+\eta_{\mathrm{app}})\|_{L^{\infty}}  \|\eta_r\|_{H^N} \le C  \|\mathrm{I}_{j,2}\|_{H^{N}} \| \eta_r+\eta_{\mathrm{app}}\|_{H^N}  \|\eta_r\|_{H^N}.
\ea\ee
The contributions including weights are proving along the same lines. Upon observing that Lemma \ref{lem:Psir} allows one to infer that 
\begin{equation*}
    \|\mathrm{I}_{j,2}\|_{\cH^{N}}\lesssim \|\Ome\|_{\cH^N}+\|\Omega_r\|_{\cH^N},
\end{equation*}
and applying Proposition \ref{lem:35} to bound $\et,\csi$ we conclude that
\be
\ba
\left|(\mathrm{I}_{j,2}(\eta_r+\eta_{\mathrm{app}}), \eta_r)_{\cH^N} \right|& \leq \left(\|\Ome\|_{\cH^N}+\|\Omega_r\|_{\cH^N}\right)\left(\|\eta_r\|_{\cH^N}+C_N\exp\left({\frac{C_N}{\alpha}t}\right)\right)\|\eta_r\|_{\cH^N}\\
\left|(\mathrm{I}_{j,2}(\xi_r+\csi-1), \xi_r)_{\cH^N}\right| &  \le \left(\|\Ome\|_{\cH^N}+\|\Omega_r\|_{\cH^N}\right)\left(1+\|\xi_r\|_{\cH^N}+C_N\exp\left({\frac{C_N}{\alpha}t}\right)\right)\|\xi_r\|_{\cH^N}.
\ea
\ee
The final estimates follow upon applying \eqref{eq:est-LOM} for $\|\Ome\|_{\cH^N}$.
\end{proof}
 
Next, we deal with the terms $\text{(RHS)}_\eta^3, \text{(RHS)}_\xi^3$ for which, we derive the following estimates by exploiting the elliptic regularity of $\hE(\Omega)$.
\begin{lem}\label{lem:RHS3}
    The following hold true for $\mathrm{(RHS)}_\eta^3$ and $ \mathrm{(RHS)}_\xi^3$:
    \begin{align*}
        \left\langle \mathrm{(RHS)}_\eta^3, \eta_r\right\rangle_{\cH^N}&\lesssim
        \left(C_N+C_N^2\exp\left(\frac{C_N}{\alpha}t\right)+\alpha C_N\exp\left(\frac{2C_N}{\alpha}t\right)\right)\|\eta_r\|_{\cH^N}\\
        &+\left(\left(\frac{1}{\alpha}+\frac{C_N}{\alpha}\exp\left(\frac{C_N}{\alpha}t\right)\right)\|\Omega_r\|_{\cH^N}+\left(C_N+\alpha\exp\left(\frac{C_N}{\alpha}t\right)\right)\|\xi_r\|_{\cH^N}\right)\|\eta_r\|_{\cH^N}\\
        &+\left(1+\frac{1}{\alpha}\right)\|\Omega_r\|_{\cH^N}\|\xi_r\|_{\cH^N}\|\eta_r\|_{\cH^N},
        \end{align*}
        \begin{align*}
        \left\langle \mathrm{(RHS)}_\xi^3, \xi_r\right\rangle_{\cH^N}&\lesssim
        \left(C_N^2\exp\left(\frac{C_N}{\alpha}t\right)+\alpha C_N\exp\left(\frac{2C_N}{\alpha}t\right)\right)\|\xi_r\|_{\cH^N}\\
        &+\left(\frac{C_N}{\alpha}\exp\left(\frac{C_N}{\alpha}t\right)\|\Omega_r\|_{\cH^N}+\left(C_N+\alpha \exp\left(\frac{C_N}{\alpha}t\right)\right)\|\eta_r\|_{\cH^N}\right)\|\xi_r\|_{\cH^N}\\
        &+\frac{1}{\alpha}\|\Omega_r\|_{\cH^N}\|\xi_r\|_{\cH^N}\|\eta_r\|_{\cH^N}.
        \end{align*}
\end{lem}

\begin{proof}
We observe that $\mathbb{P}_2(\Psi_r)=\mathcal{R}(\Omega)+\Ps(\Omega_r)$ and 
\begin{equation*}
    2 \left(\mathcal{R}(\Omega)+\Ps(\Omega_r)\right) + \sin(2\beta) \de_\beta\left(\mathcal{R}(\Omega)+\Ps(\Omega_r)\right) + (\cos^2\beta) \de_{\beta \beta} \left(\mathcal{R}(\Omega)+\Ps(\Omega_r)\right)=\frac{\cL^c(\Omega_r)}{2\alpha}+\mathcal{R}_{\cos}(\Omega),
\end{equation*}
where $\mathcal{R}_{\cos}(\Omega)=\left\langle \mathcal{R}(\Omega),\cos(2\beta)\right\rangle_{L_\beta^2}$. 
Arguing as in the proof of Lemma \ref{lem:RHS1}, we infer that
\begin{align*}
    &\left|\left\langle\left(\frac{\cL^c(\Omega_r)}{2\alpha}+\mathcal{R}_{\cos}(\Omega)\right)(1-\xi_{\mathrm{app}}-\xi_r),\eta_r\right\rangle_{\cH^N}\right|\\
    &\lesssim \left(\frac{\|\mathbb{P}_2(\Omega_r)\|_{\cH^N}}{\alpha}+\|\Omega_r\|_{\cH^N}+\|\Ome\|_{\cH^N}\right)(1+\|\csi\|_{\cH^N}+\|\xi_{r}\|_{\cH^N})\|\eta_r\|_{\cH^N}\\
    &\lesssim \left(\frac{\|\Omega_r\|_{\cH^N}}{\alpha}+C_N+\alpha\exp\left(\frac{C_N}{\alpha}t\right)\right)\left(1+C_N\exp\left(\frac{C_N}{\alpha}t\right)+\|\xi_r\|_{\cH^N}\right)\|\eta_r\|_{\cH^N},
\end{align*}
where we used \eqref{eq:est-fapp} in the last step. To bound the terms involving $\hE(\Omega)$ one proceeds analogously to the proof of Lemma \ref{lem:RHS1}, where we oberse that $\mathbb{P}_2(\hE(\Omega))=0$ and hence its derivative in the angular direction admit  uniform (in $\alpha$) $\cH^N$-bounds. It follows from Proposition \ref{lam:Rest} that
\begin{equation*}
    \left\|2 \hE(\Omega) - \sin(2\beta) \de_\beta \hE(\Omega) + (\sin^2\beta) \de_{\beta \beta} \hE(\Omega)\right\|_{\cH^N}\lesssim \|\Omega\|_{\cH^N}\lesssim\left(C_N\exp\left({\frac{C_N}{\alpha}t}\right)+\|\Omega_r\|_{\cH^N}\right).
\end{equation*}
It then follows arguing as in the proof of Lemma \ref{lem:remT} that
\begin{align*}
        \left|\left\langle \text{(RHS)}_\eta^3, \eta_r\right\rangle_{\cH^N}\right|&= \left|\left\langle [2 \hE(\Omega) - \sin(2\beta) \de_\beta \hE(\Omega) + (\sin^2\beta) \de_{\beta \beta} \hE(\Omega)](1-\csi-\xi_r), \eta_r\right\rangle_{\cH^N}\right|
        \\
        &\le C \|\Omega\|_{\cH^N}\left(1+\|\csi\|_{\cH^N}+\|\xi_r\|_{\cH^N}\right)\|\eta_r\|_{\cH^N}\\&
        \le C\left(C_N+\alpha\exp\left({\frac{C_N}{\alpha}t}\right)+\|\Omega_{r}\|_{\cH^N}\right)\left({(1+C_{N}\exp\left({\frac{C_N}{\alpha}t}\right)} + \|\xi_r\|_{\cH^N}\right)\|\eta_r\|_{\cH^N}\\
        &\le C_N{(1+\alpha C_{N+1})\exp\left({\frac{2C_N}{\alpha}t}\right)} \|\eta_r\|_{\cH^N}\\
        &\quad + \left(C_N\exp\left({\frac{C_N}{\alpha}t}\right)(\|\xi_r\|_{\cH^N}+\|\Omega_r\|_{\cH^N})+\|\Omega_r\|_{\cH^N}\|\xi_r\|_{\cH^N}\right)\|\eta_r\|_{\cH^N}.
    \end{align*}
The estimate for the term involving $\xi$ follows along the same lines. 
\end{proof}

We are now ready to proof the desired estimates on the remainder terms $(\Omega_r,\eta_r,\xi)$.

\begin{proof}[Proof of Proposition \ref{prop:rem}]
    Combining the statements from Lemma \ref{lem:remT}-\ref{lem:RHS3}, we conclude that
\begin{align*}
     \frac{\mathrm{d}}{\mathrm{d}t}F(t)^2&\leq C_{N+1}\exp\left({\frac{C_{N+1}}{\alpha}t}\right)\left(C_{N+1}+\alpha\exp\left({\frac{C_N}{\alpha}t}\right)\right)
     F(t)\\
     &\quad +\left(\frac{C_{N+1}}{\alpha}+C_N\exp\left({\frac{C_N}{\alpha}t}\right)+\frac{C_{N+1}}{\alpha}\exp\left({\frac{C_{N+1}}{\alpha}t}\right)\right)F(t)^2 +\frac{C}{\alpha} F(t)^3\\
     &\leq 
     C_{N+1}\exp\left({\frac{C_{N+1}}{\alpha}t}\right)\left(C_{N+1}+\alpha\exp\left({\frac{C_N}{\alpha}t}\right)\right)
     F(t)\\
     &\quad 
     +\left(\frac{C_{N+1}}{\alpha}+\frac{C_{N+1}}{\alpha}\exp\left({\frac{C_{N+1}}{\alpha}t}\right)\right)F(t)^2+\frac{C}{\alpha} F(t)^3.
\end{align*}
By means of a classical approximation argument one obtains
\begin{equation}\label{eq:boot}
\ba
F(t)&\leq F(0)+C_{N+1}\int_0^tC_{N+1}\exp\left({\frac{C_{N+1}}{\alpha}\tau}\right)+\alpha \exp\left({\frac{2C_{N+1}}{\alpha}t}\right) \, d\tau \\ &\quad +\int_0^{t}\left(\frac{C_{N+1}}{\alpha}+\frac{C_{N+1}}{\alpha}\exp\left({\frac{C_{N+1}}{\alpha}\tau}\right)\right)F(\tau)d\tau+C
\int_0^{t}\frac{1}{\alpha} F(\tau)^2 \, d\tau.
\ea
\end{equation}
We recall that the initial condition satisfies $F(0)\lesssim C_{N+1}\alpha \leq \alpha^{1-\eps}$ for any $\eps>0$ for $\alpha$ sufficiently small. Owing to the bootstrap argument provided by the following Lemma, we conclude the proof.
\end{proof}
\begin{Lem}\label{eq:bootstrap}
    There exists $0<\alpha_0 < \frac 13$ such that for any $0<\alpha<\alpha_0$ if 
   \begin{equation}\label{hyp:boot}
        F(0)\leq \alpha^{1-\eps}, \quad \sup_{t\in[0,T]}F(t)\leq 4 \sqrt{\alpha}, \quad T\leq \frac{\alpha\log |\log(\alpha)|}{4C_{N+1}},
    \end{equation}
    then $\displaystyle\sup_{t\in[0,T]}F(t)\leq 3 \sqrt{\alpha}$.
    \end{Lem}
\begin{proof}
  Because  of inequality \eqref{eq:boot} and assumptions \eqref{hyp:boot}, we have
  \begin{align*}
      F(t)&\leq \alpha^{1-\eps}+C_{N+1}\int_0^tC_{N+1}\exp\left({\frac{C_{N+1}}{\alpha}\tau}\right)+\alpha \exp\left({\frac{2C_{N+1}}{\alpha}t}\right) \, d\tau\\
      &\quad
      +\frac{C_{N+1}}{\alpha}\int_0^{t}\left(1+\exp\left({\frac{C_{N+1}}{\alpha}\tau}\right)\right)F(\tau)d\tau+
      16\alpha\int_0^{t}\frac{1}{\alpha}d\tau\\
      &\leq 
      \alpha^{1-\eps}+C_{N+1}\alpha |\log\alpha|^{\frac14}+\alpha^2|\log\alpha|^{\frac12}
      +\frac{C_{N+1}}{\alpha}\int_0^{t}\left(1+\exp\left({\frac{C_{N+1}}{\alpha}\tau}\right)\right)F(\tau)d\tau+
      \frac{\alpha\log|\log\alpha|}{4C_{N+1}}\\
      &\leq 3\alpha^{1-\varepsilon}+\frac{C_{N+1}}{\alpha}\int_0^{t}\left(1+\exp\left({\frac{C_{N+1}}{\alpha}\tau}\right)\right)F(\tau)\, d\tau,
  \end{align*}
  where in the last inequality we have chosen $\alpha$ small enough. By means of the integral form of Gronwall lemma we obtain
  \begin{equation*}
F(t)\leq 3\alpha^{1-\varepsilon}\exp{\left(\frac{C_{N+1}}{\alpha}\int_0^{t}\left(1+\exp\left({\frac{C_{N+1}}{\alpha}\tau}\right)\right)d\tau\right)}\leq 3\alpha^{1-\varepsilon}\exp{ \left(\frac{\log|\log\alpha|}{4}+|\log\alpha|^{1/4}\right)}\leq 3 \sqrt{\alpha},
  \end{equation*}
  where, again, we have chosen $\alpha$ small enough.
\end{proof}

\section{Proof of Theorem \ref{thm:main} and Theorem \ref{thm:main2} }\label{sec:last}
We finally have all the ingredients to prove our result. 
We first initialize the approximating Leading Order Model \eqref{eq:leading1} with the following data 
\be
\Omega_{0, \mathrm{app}}^{\alpha, \delta}(R, \beta)= \Omega_0^{\alpha, \delta}(R, \beta), \quad \eta_{0, \mathrm{app}}^{\alpha, \delta}(R, \beta)= \bar \eta_0^{\alpha, \delta}(R), \quad \xi_{0, \mathrm{app}}^{\alpha, \delta}(R, \beta)= \de_y \rho_0^{\alpha, \delta}(x, y). 
\ee
We can then use Proposition \ref{prop:expl}, yielding that the approximated (horizontal) density gradient blows up as 
\be\ba
\sup\limits_{t \in [0, T^*(\alpha)]} \|\de_x \rho_{\mathrm{app}}(t) \|_{L^\infty(\R^2)}  &\ge  \|\bar \eta_{0}^{\alpha, \delta}\|_{L^\infty (\R^2)} \left(1+{\log |\log \alpha|}\right)^{\frac{1}{c_2}},
\ea\ee
where $T^*(\alpha) =  C{\alpha}\log |\log (\alpha)| \, \to \, 0$ and $C>0$ independent of $\alpha$.
Now, for $\de_x \rho (t, x, y)$ solving \eqref{eq:2Dbouss-grad} with initial datum $\de_x \rho_0^{\alpha, \delta}$, using Proposition \ref{prop:rem},
\be
\ba
\|\de_x \rho (t) \|_{L^\infty(\R^2)}& \ge  \| \de_x \rho_{\mathrm{app}} (t)\|_{L^\infty} -  \| \de_x \rho_{\mathrm{app}} (t) - \de_x \rho (t)\|_{L^\infty} \gtrsim \| \de_x \rho_{\mathrm{app}} (t)\|_{L^\infty} -  \| \de_x \rho_{\mathrm{app}} (t) - \de_x \rho (t)\|_{\cH^k}\\
& \ge \| \de_x \rho_{\mathrm{app}} (t)\|_{L^\infty} -  \| \eta_r (t)\|_{\cH^k}
 \ge \| \de_x \rho_{\mathrm{app}} (t)\|_{L^\infty} -  F(t) \ge \| \de_x \rho_{\mathrm{app}} (t)\|_{L^\infty} - \sqrt \alpha, 
\ea
\ee
for $t \in [0, T^*(\alpha)]$ and with $F$ as in \eqref{def:functional}.
Taking $0 < \alpha \le \alpha_0 \ll 1$ small enough, this implies that
\be
\|\de_x \rho (t) \|_{L^\infty(\R^2)} \ge \frac 12 \| \de_x \rho_{\mathrm{app}} (t)\|_{L^\infty}, \quad t \in [0, T^*(\alpha)],
\ee
so providing the result and concluding the proof of Theorem \ref{thm:main}.
\subsection{On the proof of Theorem \ref{thm:main2}}
The strategy of the proof and the related computations are exactly analogous: one first rewrites the original Boussinesq equations \eqref{eq:2Dbouss} in the respective version of \eqref{eq:2Dbouss-grad} and then derives the associated leading order model. It is immediate to realize that the equations for $\Ome$ and $\et$ in the Leading Order Model for the original 2D Boussinesq system \eqref{eq:2Dbouss} are exactly the ones of \eqref{eq:leading1}. The only difference would be the equation for $\csi$ in \eqref{eq:leading1}, where the term $\frac{\cL(\Ome)}{2\alpha}$ in the right-hand side should be discarded. Hence, the same growth rates as in Theorem \ref{thm:main} are derived verbatim.

      \section{Strong ill-posedness of 3D incompressible Euler equations with swirl}\label{sec:euler}
Consider the 3D incompressible Euler equations posed on $\R^3$ 
\be\label{eq:3DEuler}
\ba
\de_t \bu + (\bu \cdot \nabla) \bu + \nabla P &=0, \\
\nabla \cdot \bu&=0, 
\ea
\ee
where $\bu=\bu(t, x, y, z)=(u_1(t, x, y, z), u_2(t, x, y, z), u_3(t, x, y, z))^T: [0, \infty) \times \R^3 \to \R^3$ is the velocity field and $P=P(t, x, y, z) : [0, \infty) \times \R^3 \to \R$ is the incompressible pressure. Our approach relies on the analogy of \eqref{eq:3DEuler} and \eqref{eq:2Dbouss} which is well-established provided that the solutions under consideration are supported away from the $z$-axis, see \cite{drivas, jeong2}. To that end, the compact support of the initial data is properly chosen below.
A bootstrap argument will ensure that this property holds on a sufficiently long time scale. The vorticity field 
\be
\omega := \nabla \times \bu,
\ee
satisfies the following equation
\be\label{eq:3Dvorticity}
\de_t \omega + (\bu \cdot \nabla) \omega = (\omega \cdot \nabla) \bu,
\ee
where the right-hand side is referred to as \emph{vortex stretching}.
Passing to cylindrical coordinates 
\be
r:= \sqrt{x^2+y^2}, \qquad \theta=\arctan\left(\frac{y}{x}\right), \qquad z,
\ee
 we consider the special case of the \emph{axisymmetric} 3D Euler equations, which are symmetric with respect to the vertical $z$-axis and therefore the velocity $\bu $ (resp. the vorticity $\omega$) and the pressure $P$ are independent of the angle $\theta$, namely
\be\ba
(u_r^\aa (t, r, z), u_\theta^\aa (t, r, z), u_z^\aa (t, r, z))&=: \bu (t, x, y, z)^T, \\
(\omega_r^\aa (t, r, z), \omega_\theta^\aa (t, r, z), \omega_z^\aa (t, r, z))&=: \omega  (t, x, y, z)^T, \\
P^\aa (t, r, z)&=: P(t, x, y, z).
\ea\ee
Note that the 3D incompressibility condition $\nabla \cdot \bu=0$ in cylindrical coordinates reduces to 
\begin{equation*}
   \de_r u_r^\aa+\de_z u_z^\aa+\frac{1}{r}u_r^\aa=0.
\end{equation*}
Provided that $r>0$, this yields the weighted incompressibility condition $\Div_{r,z}(r(u_r^\aa, u_z^\aa))=0$ and thus
\begin{equation*}
    (u_r^\aa,u_z^\aa)^{T}=-\frac{1}{r}\nabla_{(r,z)}^{\perp}\wt \psi^\aa,
\end{equation*}
for some stream-function $\wt \psi^\aa$ where we recall that $\nabla_{(r,z)}^\perp = (-\de_z, \de_r)^T$. The identity  $\omega_\theta^\aa=\de_zu_r^\aa-\de_ru_z^\aa$ motivates the introduction of 
the \emph{potential vorticity} 
\begin{equation}\label{eq:potential-vorticity}
    \frac{\omega_\theta^\aa}{r}
\end{equation}
such that
\begin{equation}\label{eq:elliptic3D1}
    \frac{\de_{zz} \wt \psi^\aa}{r^2} + \frac 1 r \de_r \left(\frac{\de_r \wt \psi^\aa}{r} \right)=- \left(\frac{\omega_\theta^\aa}{r}\right).
\end{equation}
The potential vorticity then satisfies the transport equation with source term below
\begin{equation*}
  \de_t \left(\frac{\omega_\theta^\aa}{r}\right) + u_r^\aa \de_r \left(\frac{\omega_\theta^\aa}{r}\right)  + u_z^\aa \de_z \left(\frac{\omega_\theta^\aa}{r}\right)  = - \frac{1}{r^4} \de_z (r u_\theta^\aa)^2,  
\end{equation*}
where the right-hand side accounts for the \emph{swirl}. The velocity field $(u_r^\aa,u_z^\aa)$ is uniquely determined by the potential vorticity $r^{-1}\omega_{\theta}^\aa$ through the elliptic problem \eqref{eq:elliptic3D1}, enforcing that $\wt \psi^\aa$ vanishes quadratically on $r=0$ (see \cite[Section 2]{tarek1} for details). The swirl component $u_\theta^\aa$ of the velocity field satisfies
\begin{equation}\label{eq: trans swirl}
    \de_t (r u_\theta^\aa) +  u_r^\aa \de_r (r u_\theta^\aa) + u_z^\aa \de_z (r u_\theta^\aa)=0,
\end{equation}
and furthermore exploiting the regularity properties of the solutions \cite{kato72, constantin, bertozzi} under consideration gives 
\begin{align}\label{eq:utheta2}
        \de_t (r u_\theta^\aa)^2 +  u_r^\aa \de_r (r u_\theta^\aa)^2 + u_z^\aa \de_z (r u_\theta^\aa)^2=0.
\end{align}

The analogy between the 3D axisymmetric Euler equations with \emph{swirl} and the 2D Boussinesq equations becomes evident at this stage: the couple potential vorticity $r^{-1}\omega_\theta^\aa$ and $(ru_\theta^\aa)^2$ plays the role of vorticity $\omega$ and buoyancy $\rho$ in the Boussinesq equations \eqref{eq:2Dbouss} respectively. Following this analogy, we derive the respective version of \eqref{eq:2Dbouss-grad} for the 3D axisymmetric Euler equations with \emph{swirl} by applying the gradient $\nabla_{(r,z)}$ to \eqref{eq:utheta2}. We are led to 
\be\label{eq:3Daxisymm}\ba
\de_t \left(\frac{\omega_\theta^\aa}{r}\right) + u_r^\aa \de_r \left(\frac{\omega_\theta^\aa}{r}\right)  + u_z^\aa \de_z \left(\frac{\omega_\theta^\aa}{r}\right)  &= - \frac{1}{r^4} \de_z (r u_\theta^\aa)^2,\\
\de_t \de_r (r u_\theta^\aa)^2 +  u_r^\aa \de_{rr} (r u_\theta^\aa)^2 + u_z^\aa \de_{zr} (r u_\theta^\aa)^2&=- (\de_r u_r^\aa )\de_r  (r u_\theta^\aa)^2 - (\de_r u_z^\aa) \de_{z} (r u_\theta^\aa)^2 ,\\
 \de_t\de_z(r u_\theta^\aa)^2 +  u_r^\aa \de_{rz} (r u_\theta^\aa)^2 + u_z^\aa \de_{zz} (r u_\theta^\aa)^2&=-(\de_zu_r^\aa)\de_r(ru_\theta^\aa)^2-(\de_zu_z^\aa) \de_z(ru_\theta^\aa)^2,
\ea\ee
together with the elliptic equation 
\be
-{\de_{zz}  \psi^\aa}-{\de_{rr}  \psi^\aa}  - \frac 1 r {\de_r  \psi^\aa} + \frac{\psi^\aa}{r^2}={\omega_\theta^\aa},  \label{eq:elliptic3D}
\ee
where we set $\psi^\aa=r^{-1}\wt\psi^\aa$ and where we impose  an odd symmetry in $z$ with conditions
\be\label{eq:3Daxiboundarycond}
\psi^\aa (r, 0)= \psi^\aa (0, z)=0.
\ee
Now, since the support of our initial data will be far from the vertical axis $r=0$, we set 
\begin{equation}\label{eq:changeofvariablesalpha}
    r= 1+\zeta \qquad \Rightarrow \, \de_r \, \to \, \de_\zeta.
\end{equation}
By the change of coordinates in the $(r,z)$- plane to polar coordinates $(\rho, \beta)$ centered in $(r_0,z_0)=(1,0)$ with 
\begin{equation*}
    \rho=\sqrt{(r-1)^2+z^2}=\sqrt{\zeta^2+z^2}, \quad \beta=\arctan\left(\frac{z}{\zeta}\right),
\end{equation*}
the radial scaling $R=\rho^\alpha$ for $0 < \alpha \ll 1$ in \eqref{eq:scaling}, together with the new variables defined as in \eqref{eq:new-var} and \eqref{eq:psi-polar} by
\be\label{eq:psibig}
\Omega (\cdot, R, \beta):= {\omega^\aa_\theta (\cdot, r , z)}, \qquad \Psi (\cdot, R, \beta):= \rho^{-2} \psi^\aa (\cdot, r, z), 
\ee
one finally has that
\be\label{eq:der-3D}
\de_r \; \Rightarrow \; R^{-\frac{1}{\alpha}}\left(\alpha\cos \beta R \de_R - \sin \beta \de_\beta \right), \qquad \de_z  \; \Rightarrow \;   R^{-\frac{1}{\alpha}}\left(\alpha\sin \beta R \de_R +\cos \beta \de_\beta\right)
\ee
which we note to be similar to \eqref{eq:spatial-der}.
This change of coordinates yields
\be\label{eq:uru3}\ba
u_r^\aa &=\de_z\psi^\aa= \rho (2 (\sin \beta) \Psi + (\cos \beta) \de_\beta \Psi + \alpha (\sin \beta) R \de_R \Psi), \\
u_z^\aa&=-\de_r\psi^\aa-\frac{1}{1+\zeta}\psi^\aa= \rho \left(- \frac{1}{\cos \beta} \Psi - 2 (\cos \beta) \Psi + (\sin \beta) \de_\beta \Psi - \alpha (\cos \beta) R \de_R \Psi\right),
\ea\ee
which is analogous to \eqref{eq:u1}-\eqref{eq:u2} (up to a switching of signs).
The elliptic equation \eqref{eq:elliptic3D} reads
\begin{align}
-\alpha^2 R^2 \de_{RR} \Psi - \alpha (\alpha + 5) R \de_R \Psi - \de_{\beta \beta} \Psi + \de_\beta ((\tan \beta) \Psi)-6 \Psi&=\Omega, \label{eq:elliptic3Dnew}
\end{align}
with 
\be
\Psi (R, 0)= \Psi \left(R, \frac \pi 2\right)=0.
\ee
We refer for instance to \cite[Section 2]{tarek1} and \cite[Section 1.2]{jeong} and references therein for further details on the derivation and properties of \eqref{eq:elliptic3Dnew}.
As in \cite{chen1}, we will rely on elliptic estimates in the weighted space
\begin{align}\label{def:weights3D}
    \|f\|_{\cH^k(\rho_i)}:= \sum_{i=0}^k \|R^i\de_R^i f \rho_1^{1/2}\|_{L^2}+ \|R^i\de_R^i \de_\beta^{k-i} f \rho_2^{1/2}\|_{L^2}, \; \rho_i= \frac{(1+R)^4}{R^4} \sin (2\beta)^{-\sigma_i}, \; \sigma_1=\frac{99}{100}, \; \sigma_2=1+\frac{\alpha}{10}.
\end{align}
Throughout this section, we use the notation $\cH^k:=\cH^k(\rho_i)$ and shall prove the following. 

\begin{Thm}[Strong ill-posedness of the 3D axisymmetric Euler equations with swirl]\label{thm:main3sec3}
There exists $0<\alpha_0 \ll 1$ such that for any $0<\alpha \le \alpha_0 $ and any $\delta>0$, there exist initial data 
\be\label{eq:initialdata3D}
\omega^\aa_{\theta,0}(r, z)=-\Omega_0^{\alpha, \delta}(R, \beta), \qquad (r u_{\theta,0}^\aa)^2(r,z)=U_0^{\alpha, \delta}(R, \beta)
\ee
with 
\be\ba
\left\|({\omega^\aa_{\theta,0}}, \nabla_{r, z}(r u_{\theta,0}^\aa)^2)\right\|_{L^{\infty}(\R^2)} &=  \delta, 
\ea \ee
so that one of the following holds:
\begin{enumerate}[(i)]
\item if the initial datum takes the form 
\begin{align}\label{eq:3Deuler-initialU}
    \Omega_0^{\alpha, \delta}(R, \beta)=  \bar g_0^{\alpha, \delta} (R) \sin (2\beta)\cos \beta, \qquad U_0^{\alpha, \delta}(R, \beta)= R^{1/\alpha} \bar \eta_0^{\alpha, \delta} (R) \cos \beta,
\end{align}
where $\bar g_0^{\alpha, \delta} (R), \bar \eta_0^{\alpha, \delta} (R) \in C_c^\infty([1/10, \infty))$, with $\bar g_0^{\alpha, \delta} (R)\ge 0$, then the  solution $\omega_\theta^\aa (t,r,z), u_\theta^\aa (t, r, z)$ to the Cauchy problem associated with the 3D axi\-symmetric Euler equations with swirl \eqref{eq:3Daxisymm}-\eqref{eq:elliptic3D} satisfies
\be\label{eq:crucial-inequality}\ba
\sup\limits_{t \in [0, T^*(\alpha)]} \left\|\left(\frac{(u_\theta^\aa)^2}{r}+\frac12\de_r (u_\theta^\aa)^2\right) (t) \right\|_{L^\infty}  &\ge \frac 12  \|\de_r(u_{\theta,0}^\aa)^2\|_{L^\infty} \left(1+{\log |\log \alpha|}\right)^{\frac{1}{c_2}};
\ea\ee
\item if the initial datum takes the form 
\begin{align}
    \Omega_0^{\alpha, \delta}(R, \beta)=  \bar g_0^{\alpha, \delta} (R) \sin (2\beta)\cos \beta, \qquad U_0^{\alpha, \delta}(R, \beta)=R^{1/\alpha} \bar \eta_0^{\alpha, \delta} (R) \sin \beta,
\end{align}
where $\bar g_0^{\alpha, \delta} (R), \bar \eta_0^{\alpha, \delta} (R) \in C_c^\infty([1/10, \infty))$, with $\bar g_0^{\alpha, \delta} (R)\le 0$,
then 
\be\label{eq:omegablow-upthm}
\ba
\sup\limits_{t \in [0, T^*(\alpha)]} \|\omega_\theta (t) \|_{L^\infty } &\ge \frac 12 \|\omega_{\theta, 0}^\aa\|_{L^\infty} \left(1+{\log |\log \alpha|}\right)^{\frac{1}{c_2}},\\
\sup\limits_{t \in [0, T^*(\alpha)]} \|\de_z(u_\theta^\aa)^2 (t) \|_{L^\infty } &\ge \frac 12  \|\de_z (u_{\theta, 0}^\aa)^2 \|_{L^\infty} \left(1+{\log |\log \alpha|}\right)^{\frac{1}{c_2}},
\ea\ee
\end{enumerate}
where 
\be\label{def:talpha-star}
T^*(\alpha)=C {\alpha}\log |\log (\alpha)|,\ee
and $C, c_2>0$ are independent of $\alpha$.
\end{Thm}
\begin{Rmk}[Comparison with the results of Elgindi \& Masmoudi \cite{tarek3} and  Bourgain \& Li \cite{bourgain2015}]\label{rmk:comparison}
The first results of ill-posedness of the 3D incompressible Euler equations in $L^\infty$ based spaces are due to Elgindi \& Masmoudi \cite{tarek3} and  Bourgain \& Li \cite{bourgain2015}. The mechanism of ill-posedness of Theorem \ref{thm:main3sec3} is completely different from \cite{bourgain2015}, and in particular from \cite[Theorem 1.6]{bourgain2015}, which is about the 3D axisymmetric Euler equations without swirl (see also \cite{BL21} for the strong ill-posedness of the 3D axisymmetric Euler equations without swirl in various critical spaces ). Notice indeed that in the first point of our Theorem \ref{thm:main3sec3}, the $L^{ \infty}$-norm inflation of the full vorticity field when $\bar g_0^{\alpha, \delta}\ge 0$ is a consequence of the norm inflation of the gradient of the swirl $u_\theta^\aa$, which is exactly zero in \cite{bourgain2015}. Hence, it is more interesting to compare Theorem \ref{thm:main3sec3} with \cite[Proposition 10.1]{tarek3}, where instead of considering the 3D axisymmetric Euler equations with velocity field depending only on $(r=\sqrt{x^2+y^2}, z)$ (and symmetric with respect to the vertical axis), the authors consider the situation where the velocity field depends only on the plane $(x, y)$. Though the framework is different, they provide an example of a datum for which the third component of the velocity field exhibits a norm inflation, while the horizontal vorticity (which may be compared to our potential vorticity) remains bounded.
\end{Rmk}
The proof of Theorem \ref{thm:main3sec3} follows. While the method of the proof is very similar to the one of Theorem \ref{thm:main}, the main difference consists in an argument ensuring the validity of the analogy of the Boussinesq \eqref{eq:2Dbouss-grad} and axisymmetric Euler \eqref{eq:3Daxisymm}-\eqref{eq:elliptic3D} for solutions supported away from the symmetry axis. We outline the proof of Theorem \ref{thm:main3sec3},  highlighting the important differences only.
\begin{enumerate}
    \item Derivation of the Leading Order Model \eqref{eq:LOM3D1} and growth rates for a suitable class of initial data supported away from $z=0$ on $[0,T^*(\alpha)]$.
    \item Localized (in space) elliptic estimates on $[0,T']$, where $T'$ is the supremum over all times for which the support of the solution remains bounded away from $r=0$, see \eqref{eq:T'}.
    \item Remainder estimates in $[0,T']$, especially for the contribution of the swirl in the first equation of \eqref{eq:3Daxisymm}.
    \item A bootstrap argument yielding that $T'\geq T^*(\alpha)$.
    \item Conclusion of the proof of Theorem \ref{thm:main3sec3}.
\end{enumerate}
Each step is presented in a dedicated sub-section. 
\subsection{Derivation of the Leading Order model (LOM)}
As in Section \ref{sec:derivation} for the Bousinesq equations, the first step towards the LOM consists in a suitable expansion of the stream-function $\Psi$. We recall the 3D version of Elgindi's decomposition \cite{tarek1} of the Biot-Savart law.
\begin{Thm}[{\cite[Proposition 7.1]{tarek1}}, \cite{drivas}]\label{thm:tarek3D}
Given $\Omega=\Omega (R, \beta) \in H^k$ with $\Omega (R, 0)=\Omega (R, \frac \pi 2)=0$,
there is a unique solution to
\begin{equation}\label{eq:ellittica3D}
-\alpha^2 R^2 \de_{RR} \Psi - \alpha (\alpha + 5) R \de_R \Psi - \de_{\beta \beta} \Psi + \de_\beta ((\tan \beta) \Psi)-6 \Psi= \Omega,
\end{equation}
with boundary conditions $\Psi(R, 0)=\Psi(R, \frac \pi 2)=0$. It is given by
\be\label{eq:psi-main3D}
\Psi=\Psi(\Omega)(R, \beta) =\Psa + \hE, 
\ee 
where
\be \label{eq:psi23D}
\Psa=\Psa( \Omega)(R, \beta):=  \Ps( \Omega) + \mathcal{R}^\alpha (\Omega)
\ee
and 
\be \label{eq:psiapp3D}
\Ps=\Ps( \Omega)(R, \beta):=  \frac{\cL_{12} ( \Omega)(R)}{4 \alpha}\sin (2\beta),\quad \cL_{12} ( \Omega)(R):= \frac{3}{8 \pi} \int_R^\infty \int_0^{2\pi} \frac{ \Omega(s, \beta) \sin (2\beta)\cos \beta}{s}\, d\beta \, d s,
\ee
satisfies the equation
\be\label{eq:L0}
6\Ps-\de_{\beta \beta }\Ps+ \de_\beta ((\tan \beta) \Ps) =0
\ee
and the error term 
$\mathcal{R}^\alpha( \Omega)$ satisfies the inequality
\begin{align}\label{eq:hardy3D}
\|\mathcal{R}^\alpha( \Omega)\|_{H^k} \le M_k \| \Omega\|_{H^k}.
\end{align}
Moreover, $\Ps (\Omega), \Psi_2(\Omega), \hE(\Omega)$ satisfy the following elliptic estimates
\be\ba
\alpha \left\| \de_\beta \left(\frac{\Ps}{\cos \beta}\right) \right\|_{L^2}+ \alpha \|\de_{\beta \beta} \Ps\|_{L^2}+\alpha^2 \|R^2 \de_{RR}\Ps\|_{L^2} & \le M_k  \|  \Omega\|_{L^2},\\
\alpha \left\| \de_\beta \left(\frac{\Psi_2}{\cos \beta}\right) \right\|_{L^2}+ \alpha \|\de_{\beta \beta} \Psi_2\|_{L^2}+\alpha^2 \|R^2 \de_{RR}\Psi_2\|_{L^2} & \le M_k  \| \Omega\|_{L^2},\\
\left\| \de_\beta \left(\frac{\hE}{\cos \beta}\right) \right\|_{L^2}+ \|\de_{\beta \beta} \hE\|_{L^2}+\alpha^2 \|R^2 \de_{RR}\hE\|_{L^2} & \le M_k \|  \Omega\|_{L^2}.
\ea\ee
The constant $M_k>0$ is independent of $\alpha$.
\end{Thm}
\begin{Rmk}
Note, as remarked in \cite[Remark 7.2]{tarek1}, that the mixed derivative $\alpha \|R\de_{R\beta} \Psi_2\| \le M_k \|\Omega\|_{L^2}$ is obtained by interpolation (the same applies to $\Ps$ and $\hE$), upon changing the constant $M_k$.
\end{Rmk}
Let us derive our Leading Order Model arguing as in Section \ref{sec:derivation}. From \eqref{eq:der-3D}-\eqref{eq:uru3}, as previously done to derive \eqref{eq:approx-der-vel}, we obtain that 
\be\label{eq:3Dexpansions}\ba
\de_r (r u_\theta^\aa)^2 & = - \frac{\sin \beta}{R^\frac 1 \alpha} \de_\beta  (r u^\aa_\theta){^2} + \text{l.o.t}, \qquad \de_z (r u_\theta^\aa)^2  =  \frac{\cos \beta}{R^\frac 1 \alpha} \de_\beta (r u^\aa_\theta){^2} + \text{l.o.t}, \\
u_r^\aa&= {R^\frac 1 \alpha} \frac{\cL_{12}( \Omega)}{2\alpha} \cos \beta + \text{l.o.t},\qquad 
u_z^\aa= - {R^\frac 1 \alpha} \frac{\cL_{12}( \Omega)}{\alpha} \sin \beta + \text{l.o.t}, \\
\de_z u_r^\aa&=\text{l.o.t},\qquad 
\de_z u_z^\aa= - \frac{\cL_{12}( \Omega)}{\alpha}  + \text{l.o.t}, \\
\de_r u_r^\aa&=  \frac{\cL_{12}( \Omega)}{2\alpha}+ \text{l.o.t},\qquad 
\de_r u_z^\aa=  \text{l.o.t},
\ea\ee
and the transport term
\be\label{eq:transport3D}
u_r^\aa \de_r + u_z^\aa \de_z = - \frac{3 \cL_{12}( \Omega)}{2\alpha} \sin (2\beta) \de_\beta + \text{l.o.t.}= - 6 \Ps \de_\beta + \text{l.o.t.}
\ee
Plugging now all the expansions \eqref{eq:3Dexpansions} into \eqref{eq:3Daxisymm}, keeping only the main order terms and introducing the notation
\be
\eta (\cdot, R, \beta) = \de_r (r u_\theta^\aa)^2, \qquad \xi(\cdot, R, \beta) = \de_z (r u_\theta^\aa)^2,
\ee
we derive the following Leading Order Model
\be\label{eq:LOM3D2}\ba
\de_t \wt \Omega_{\mathrm{app}} - 6 \Ps ( \wt \Omega_{\mathrm{app}}) \de_\beta { \wt \Omega_{\mathrm{app}}}&= \frac{\cL_{12}( \wt \Omega_{\mathrm{app}})}{2\alpha}  \wt \Omega_{\mathrm{app}} - {\frac{\csi }{ (1+\zeta)^3}}, \\
\de_t \et - 6 \Ps ( \wt \Omega_{\mathrm{app}}) \de_\beta \et &= - \frac{\cL_{12}( \wt \Omega_{\mathrm{app}})}{2\alpha} \et, \\
\de_t \csi - 6 \Ps ( \wt \Omega_{\mathrm{app}}) \de_\beta \csi & = \frac{\cL_{12}( \wt \Omega_{\mathrm{app}})}{\alpha} \csi.
\ea\ee
We will prove in the following that the second addend in the right-hand side of the equation for $\wt \Omega_\text{app}$ in \eqref{eq:LOM3D2} is negligible at main order, so that, changing $ \wt \Omega_{\mathrm{app}} \, \to \, -\Ome$, \eqref{eq:LOM3D2} rewrites as 
\be\label{eq:LOM3D1}\tag{LOM3D}\ba
\de_t {\Ome} + 6 \Ps (\Ome) \de_\beta {\Ome}&= - \frac{\cL_{12}(\Ome)}{2\alpha} \Ome, \\
\de_t \et + 6 \Ps (\Ome) \de_\beta \et &=  \frac{\cL_{12}(\Ome)}{2\alpha} \et, \\
\de_t \csi + 6 \Ps (\Ome) \de_\beta \csi & = - \frac{\cL_{12}(\Ome)}{\alpha} \csi.
\ea\ee
%
%
Let us solve \eqref{eq:LOM3D1} for a class of initial data.
Similarly to the 2D Boussinesq equations, we choose 
\be\label{eq:ID-OA-3D}
{\Ome (0, \cdot)} = \bar g_0^{\alpha, \delta} (R) \sin (2\beta) \cos \beta,
\ee
where 
\be
\bar g_0^{\alpha, \delta}(R)\ge 0; \qquad  
\et (0, \cdot)=\bar \eta_0^{\alpha, \delta}(R)
\ee
are radial functions with compact support. As in Section \ref{sec:initialdata}, the initial datum for the full system \eqref{eq:3Daxisymm} is the following :
\be\label{eq:ID-u-3D}
(r u_\theta^\aa)^2 (0, \cdot)= R^\frac 1 \alpha \bar \eta_0^{\alpha, \delta}(R)\cos \beta,
\ee
where for example $\bar \eta_0^{\alpha, \delta}(R)$ is similar to the one given in Section \ref{sec:initialdata}.
We introduce, following \cite{chen1}, the \emph{support size} 
of
$\supp\{(\omega_\theta^\aa,u_\theta^\aa)(t, \cdot)\}$ 
for $t \in [0, T^*(\alpha)]$ denoted by 
\begin{align}\label{eq:supp-eta}
    \cS(t):= \text{ess inf} \{\tilde{\rho} \, : \, \omega_\theta^\aa(t, r, z)=(ru_\theta^\aa(t, r, z))^2=0 \quad  \text{for} \quad (r-1)^2+z^2\ge \tilde{\rho}^2\}, 
    \end{align}
    where we recall that $\rho=0$ for $(r,z)=(1,0)$.
Note that as $R=\rho^\alpha$, taking $\bar\eta_0^{\alpha,\delta}=\eta_0^{\alpha, \delta}(R-1/8)$ where $\eta_0^{\alpha, \delta}(\cdot) \in C_c^\infty $ is supported in $(1/30, 1/60)$, and similarly for $\bar g_0^{\alpha, \delta}(R)$, then for $\eps_0=1/60$, 
\be\label{eq:S0}
   (1/8+\eps_0)^\frac 1\alpha  \le  \cS(0) \le (1/8+2\eps_0)^\frac 1 \alpha. 
\ee
A central step of the proof of Theorem \ref{thm:main3sec3} will be to control the support size $\cS(t)$ up to the time for which $L^{\infty}$-norm inflation of  $\Omega (t, \cdot), \eta (t, \cdot), \xi (t, \cdot)$ occurs, see Section \ref{sec:supp-size}. For our choice of the initial data, see \eqref{eq:3Deuler-initialU} (Case (i) of Theorem \ref{thm:main3sec3}) where $(r u^\aa_{\theta, 0})^2=R^{1/\alpha} \bar \eta_0^{\alpha, \delta} (R) \cos \beta$, taking $\et (0, \cdot)=\bar \eta_0^{\alpha, \delta}$,
\be
\| \eta_r (0, \cdot)\|_{\cH^N}= \|\de_r (r u_\theta^\aa)^2 (0, \cdot)- \et (0, \cdot)\|_{\cH^N} = \alpha C_{N+1},
\ee
where $C_N=C_N(\|\bar g_0^{\alpha, \delta}\|_{\cH^N}, \|\bar \eta_0^{\alpha, \delta}\|_{\cH^N})$.
About $\csi(0, \cdot)$, we can choose $\csi(0, \cdot)=0$, yielding
\begin{align}
    \| \xi_r (0, \cdot)\|_{\cH^N}= \|\de_z (r u_\theta^\aa)^2 (0, \cdot)- \csi (0, \cdot)\|_{\cH^N}=\alpha C_{N+1}.
\end{align}
With the notation $g(t):=\Ome$, one has 
from \eqref{eq:LOM3D1} that
\be
\de_t g + \frac{3}{\alpha } \gamma \cL_{12}(g) \de_\gamma g=-\frac{\cL_{12}(g)}{2\alpha }   g,\qquad \gamma:=\tan \beta.
\ee
Then, setting 
\begin{equation}
    g(t)=: \gamma^{-\frac{1}{6}} \wt g (t),
\end{equation}
one has that $\wt g$ satisfies
\be\label{eq:omega-formula3D}
\de_t \wt g + \frac{3}{\alpha} \gamma \cL_{12}(\gamma^{-\frac 16} \wt g) \de_\gamma \wt g=0.
\ee
Now, the flow map is 
\begin{align}\label{eq:tildeg}
\phi_{\gamma} (t)= \gamma\exp\left({\frac{3}{\alpha} \int_0^t \cL_{12} (\wt g(\tau) \gamma^{-1/6}) \, d\tau}\right) \, \Rightarrow \; (\phi_{\gamma} (t))^{-1}= (\tan \beta)\exp\left({-\frac{3}{\alpha} \int_0^t \cL_{12} (\wt g(\tau) \gamma^{-1/6}) \, d\tau}\right), 
\end{align}
so that the solution to \eqref{eq:omega-formula3D} with initial datum  $\wt g(0)=\bar g_0^{\alpha, \delta}(R) \sin (2\beta)\cos \beta (\tan \beta)^{\frac 16}$ such that $$\Omega_{0, \mathrm{app}}=g(0)=\bar g_0^{\alpha, \delta}(R) \sin (2\beta) \cos \beta,$$
reads
\begin{align}
    \wt g(t)=\bar g_0^{\alpha, \delta}(R) \frac{2\gamma^\frac 76}{(1+\gamma^2)^\frac 32}= \bar g_0^{\alpha, \delta}(R)\frac{2 (\tan \beta)^\frac 76 \left(\exp\left({-\frac{3}{\alpha} \int_0^t \cL_{12} (\gamma^{- 1/6} \wt g(\tau)) \, d\tau}\right)\right)^\frac 76}{\left(1+(\tan \beta)^2\exp\left({-\frac{6}{\alpha} \int_0^t \cL_{12} (\gamma^{- 1/6} \wt g(\tau)) \, d\tau}\right)\right)^\frac 32}.
\end{align}
We then deduce that
\begin{align}\label{eq: g explicit}
    g(t) & = (\tan \beta)^{-1/6}  \wt g(t) =\bar g_0^{\alpha, \delta}(R)\frac{2 (\tan \beta) \left(\exp\left({-\frac{7}{2\alpha} \int_0^t \cL_{12} (g(\tau)) \, d\tau}\right)\right)}{\left(1+(\tan \beta)^2\exp\left({-\frac{6}{\alpha} \int_0^t \cL_{12} ( g(\tau)) \, d\tau}\right)\right)^\frac 32}.
\end{align}
This expression of $g(t)$ is completely analogous to \eqref{eq:g-formula}, and the respective upper and lower bounds can be obtained as the ones of Lemma \ref{prop:LOM} . Once again, we observe that the sign of $\cL_{12}(g(t))$ is determined by the initial datum $\bar g_0^{\alpha, \delta}(R)$. In particular $\|g(t)\|_{L^\infty} \lesssim \delta$ if $\bar g_0^{\alpha, \delta}(R)\ge 0$ (Case (i) of Theorem \ref{thm:main3sec3}), while it can display strong norm inflation if $\bar g_0^{\alpha, \delta}(R)\le 0$ (Case (ii) of Theorem \ref{thm:main3sec3}).
As done in Section \ref{sec:derivation}, we can now solve \eqref{eq:LOM3D1} for $\et$, yielding
\be\label{eq:eta-sol3D}\et(t, R, \beta)= \bar \eta_0^{\alpha, \delta} (R) \exp{\left(  \frac{3}{\alpha} \int_0^t {\cL_{12}\left({\Ome (\tau)}\right)}\, d\tau\right)}=\bar \eta_0^{\alpha, \delta} (R) \exp{\left(  \frac{3}{\alpha} \int_0^t {\cL_{12}\left({ g (\tau)}\right)}\, d\tau\right)}.\ee
We adapt Lemma \ref{prop:LOM} and apply it to see that, choosing $\bar g_0^{\alpha, \delta}(R)\ge 0$, the inequalities \eqref{eq:Lg-bounds}-\eqref{eq:g-upperandlower} hold, so that
\begin{equation}\label{eq:bound L12}
\frac{1}{\alpha} \int_0^t \cL_{12}(g (\tau)) \, d\tau \ge \frac{1}{ c_2} \log \left( 1+ \frac{c_2  t}{2\alpha} \int_R^\infty \frac{ \bar g_0^{\alpha, \delta}(s)}{s} \, ds  \right),
\end{equation}
for some constant $c_2>0$, independent of $\alpha$.
We thus obtain the analogous of Proposition \ref{prop:expl} for \eqref{eq:LOM3D1}.
\begin{Prop}\label{prop:expl3D}
For any $\delta>0$, there exists $0<\alpha_0 \ll 1$ such that, for any $0<\alpha \le \alpha_0 $, there exist initial data $(\Omega_{0, \mathrm{app}}^{\alpha, \delta}(R, \beta), \xi_{0, \mathrm{app}}^{\alpha, \delta}(R, \beta), \eta_{0, \mathrm{app}}^{\alpha, \delta}(R, \beta))$ with
$$\|(\Omega_{0, \mathrm{app}}^{\alpha, \delta}, \eta_{0, \mathrm{app}}^{\alpha, \delta}, \xi_{0, \mathrm{app}}^{\alpha, \delta})\|_{L^\infty(\R^2)} =  \delta,$$ 
such that either (i) or (ii) of the following are satisfied
\begin{enumerate}
\item[(i)] if the initial datum takes the form  \begin{align}
 \Omega_{0,\mathrm{app}}^{\alpha, \delta} (R, \beta) = \bar g_0^{\alpha, \delta} (R) \sin (2\beta)\cos \beta, \quad \eta_{0, \mathrm{app}}^{\alpha, \delta} (R, \beta)= \bar \eta_0^{\alpha, \delta}(R), \quad  \xi^{\alpha, \delta}_{0, \mathrm{app}}(R, \beta)=0, 
\end{align}
where $\bar g_{0}^{\alpha, \delta} (R), \bar \eta_{0}^{\alpha, \delta} (R)  \in C_c^\infty([1/10, \infty))$
and $\bar g_0^{\alpha, \delta} (R)\ge 0$, then the solution $(\Ome^{\alpha, \delta}, \eta_\mathrm{app}^{\alpha, \delta}, \xi_\mathrm{app}^{\alpha, \delta})$ to the Cauchy problem associated with \eqref{eq:LOM3D1} satisfies
\be\ba
\|\et^{\alpha, \delta}(t) \|_{L^\infty(\R^2)}  &\ge \|\eta_{0, \mathrm{app}}^{\alpha, \delta}\|_{L^\infty (\R^2)} \left(1+\frac{c_2 t}{2\alpha }C_0\right)^{\frac{1}{c_2}},
\ea\ee
where $C_0=\sup\limits_{R \in \mathrm{supp}(\bar g_0^{\alpha, \delta}(R))} \int_R^\infty \frac{\bar g_0^{\alpha, \delta}(s)}{s} \, ds$ and $c_2>0$ is independent of $\alpha$. In particular, this yields that 
\be\label{eq:etablow-up}\ba
\sup\limits_{t \in [0, T^*(\alpha)]} \|\et^{\alpha, \delta}(t) \|_{L^\infty}  &\ge  \|\eta_{0, \mathrm{app}}^{\alpha, \delta}\|_{L^\infty} \left(1+{\log |\log \alpha|}\right)^{\frac{1}{c_2}};
\ea\ee
\item[(ii)] if the initial datum has the following form \begin{align}
 \Omega_{0,\mathrm{app}}^{\alpha, \delta} (R, \beta) = \bar g_0^{\alpha, \delta} (R) \sin (2\beta)\sin \beta, \quad \eta_{0, \mathrm{app}}^{\alpha, \delta} (R, \beta)= 0, \quad  \xi^{\alpha, \delta}_{0, \mathrm{app}}(R, \beta)=\bar \eta_0^{\alpha, \delta}(R),
\end{align}
where $\bar g_{0}^{\alpha, \delta} (R), \bar \eta_{0}^{\alpha, \delta} (R) \in C_c^\infty([1/10, \infty))$
and 
$\bar g_0^{\alpha, \delta} (R)\le 0$, then 
\be\label{eq:omegablow-up}\ba
\sup\limits_{t \in [0, T^*(\alpha)]} \|\Ome^{\alpha, \delta}(t) \|_{L^\infty}  &\ge  \|\Omega_{0, \mathrm{app}}^{\alpha, \delta}\|_{L^\infty} \left(1+{\log |\log \alpha|}\right)^{\frac{1}{c_2}},
\ea\ee
\end{enumerate}
where $T^*(\alpha)= C{\alpha}\log |\log (\alpha)|$ and $C>0$ is independent of $\alpha$.
\end{Prop}
\subsubsection{Further estimates on the solutions to the \eqref{eq:LOM3D1}}
Applying Lemma \ref{prop:LOM} to equation \eqref{eq: g explicit} yields
\begin{itemize}
    \item Case (i): \eqref{eq:etablow-up} holds, while ${\Ome}$ does not blow up
\be
\left\|{\Ome}\right\|_{L^\infty} \lesssim \|\bar g_0^{\alpha, \delta}\|_{L^\infty} \lesssim \delta, \qquad  \left
\|{\Ome} \right\|_{\cH^N} \le C_N;
\ee
\item Case (ii): \eqref{eq:omegablow-up} holds and
\be\label{eq:xiblow-up}\ba
\sup\limits_{t \in [0, T^*(\alpha)]} \|\csi^{\alpha, \delta}(t) \|_{L^\infty}  &\ge  \|\xi_{0, \mathrm{app}}^{\alpha, \delta}\|_{L^\infty} \left(1+{\log |\log \alpha|}\right)^{\frac{1}{c_2}},
\ea\ee
while $\et$ does not blow up
\be
\left\|{\et}\right\|_{L^\infty} \lesssim \|\bar \eta_0^{\alpha, \delta}\|_{L^\infty} \lesssim \delta.
\ee
\end{itemize}
Notice that in Case (i) and Case (ii) of Proposition \ref{prop:expl3D}, the roles of $\et^{\alpha, \delta}, \csi^{\alpha, \delta}$ are simply switched, therefore most of the proofs will only treat Case (i).  
In Case (i), exactly as in Proposition \ref{lem:35}, one gets that
\begin{align}\label{eq:est-app}
   \|\Ome (t) \|_{\cH^k} \lesssim C_k \exp\left({\frac{C_k}{\alpha} t}\right), \quad \|\et (t) \|_{\cH^k} \lesssim C_k \exp\left({\frac{C_k}{\alpha} t}\right), \quad  \|\csi (t) \|_{\cH^k} \lesssim  C_{k+1}\alpha\exp\left({\frac{C_k}{\alpha} t}\right).
\end{align}


\begin{Rmk}[Comparison with the fundamental model in \cite{tarek1}]
    The fundamental model for the 3D axisymmetric Euler equations without swirl in \cite[page 668]{tarek1} reads
    \be
    \de_t \Ome = \frac{\cL_{12}(\Ome)}{2\alpha}\Ome,
    \ee
    which is solved by
\be\Ome=\frac{\Omega_{0, \mathrm{app}}^{\alpha, \delta}}{\left(1- \frac{t}{2\alpha}\cL_{12}(\Omega_0^{\alpha, \delta})\right)^2}.
    \ee
    Passing from \eqref{eq:LOM3D2} to \eqref{eq:LOM3D1}, we performed the change of variable $\Ome \, \to \, -\Ome$, so that, applying it again, the above equation is
     \be
    \de_t \Ome = -\frac{\cL_{12}(\Ome)}{2\alpha}\Ome,
    \ee
    whose solution reads 
    \be
    \Ome=\frac{\Omega_{0, \mathrm{app}}^{\alpha, \delta}}{\left(1+ \frac{t}{2\alpha}\cL_{12}(\Omega_{0, \mathrm{app}}^{\alpha, \delta})\right)^2},
    \ee
    with $\Omega_{0, \mathrm{app}}^{\alpha, \delta}= R^{-\frac 1 \alpha} \bar g_0^{\alpha, \delta}(R) \sin (2\beta) \cos \beta$, such that
    \begin{itemize}
        \item if $\bar g_0^{\alpha, \delta}(R)\ge 0$, then $\cL_{12}(\Omega_{0, \mathrm{app}}^{\alpha, \delta})=\frac{3}{16} \int_R^\infty \bar g_0^{\alpha, \delta}(s)(s) \, ds >0$ and $\Ome$ does not blow up;
        \item if $\bar g_0^{\alpha, \delta}(R)\le 0$, then $\cL_{12}(\Omega_{0, \mathrm{app}}^{\alpha, \delta})<0$, then $\Ome (t)$ blows up at algebraic rate $x^{-2}$ for $x \sim 0$ for times $t \sim \alpha$.
    \end{itemize}
    From inequality \eqref{eq:omegablow-upthm}, we see that the inflation rate of the true solution $\omega_\theta^\aa (t) $ is actually $\log |\log x|$ as $x \sim 0$ for a time-scale $t \sim \alpha \log |\log \alpha|$.
\end{Rmk}
\subsection{Localized elliptic estimates}\label{sec:supp-size}
We aim to develop elliptic estimates for the stream function $\Psi$ as in \eqref{eq:psibig} which are localized in space to avoid the singularity occurring in $r=0$.
For that purpose, we introduce
\begin{equation}\label{eq:T'}
    T'=\sup\left\{ T\in [0,T^*(\alpha)] \,\, : (1/10)^\frac{1}{\alpha}\leq \cS(t)\leq (1/6)^\frac{1}{\alpha} \, \text{for all} \, t\in [0,T]\right\}, \quad \text{with} \quad T^*(\alpha) \; \text{in} \, \eqref{def:talpha-star},
\end{equation}
and prove the required estimates for all $t\in [0,T']$. Note that in view of the local well-posedness of \eqref{eq:3Daxisymm}-\eqref{eq:elliptic3D} one has $T'>0$ by continuity.
First, we provide an estimate for $\Psi$ away from the support of $\omega_\theta^\aa$ the proof of which follows the same lines as \cite[Lemma 9.4]{chen1}.
\begin{Lem}\label{lem:outside}
    Let $t\in [0,T']$, $\psi^\aa$ be the solution to \eqref{eq:elliptic3D}-\eqref{eq:3Daxiboundarycond} and $\cS(t)$ be defined in \eqref{eq:supp-eta}. For any $(2\cS(t))^\alpha<\lambda<\frac{1}{2}$    
, recalling from \eqref{eq:psibig} that $\Psi(\cdot, R, \beta)=\rho^{-2}\psi^\aa (\cdot, r, z)$, it holds that 
    \begin{align}
        \|\Psi \mathbf{1}_{\lambda \le R \le 2\lambda} \|_{L^2} \lesssim (1+|\log (\lambda^\frac 1 \alpha)|) \cS(t) \lambda^{-\frac 1\alpha}\|\Omega\|_{L^2}.
    \end{align}
\end{Lem}
The $L^2$-estimate is proven along the same lines as the respective version in \cite[Lemma 9.4]{chen1}. However, the proof given here is simplified as, in contrast to \cite{chen1}, no dynamic scaling is used and the elliptic problem is posed on the whole space for which an explicit Biot-Savart formula is known, e.g. \cite{marchioro}.

\begin{proof}
    Exploiting the explicit representation formula for the Biot-Savart law of the axi-symmetric Euler eq., namely for \eqref{eq:elliptic3D}, see \cite[Section 2]{marchioro}, one has that 
    \begin{equation*}
        \left|\psi^\aa(r,z)\right|\lesssim  \left|\int_{-\pi}^\pi \int_0^{\infty} \int_{-\infty}^{\infty}\frac{\cos\gamma}{4\pi\sqrt{\frac{(z-z')^2+(r-r')^2}{rr'}+2(1-\cos\gamma)}}\omega_\theta^\aa (r',z')\sqrt{\frac{r'}{r}}d\gamma dr'dz' \right|.
    \end{equation*}
    Taking into account that $(r,z)$ is chosen away from $\mathrm{supp}(\omega_\theta^\aa)$, integrating in $\gamma$ and using the estimate \cite[(2.27)-(2.28)]{marchioro}, one concludes that
    \begin{equation*}
        \left|\psi^\aa(r,z)\right|\lesssim  \int_0^{\infty}\int_{-\infty}^{\infty} |\omega_\theta^{\aa}(r',z')|\left(1+\left|\log\left(\frac{(r-r')^2+(z-z')^2}{rr'}\right)\right|\right)\sqrt{\frac{r'}{r}}dr'dz'.
    \end{equation*}
Note that in this setting, there exist two positive $r_0>0, r_0'>0$ such that $r>r_0, r'> r_0'$. We conclude that
\begin{equation*}
    \left|\Psi(R,\beta)\right|\lesssim R^{-\frac{2}{\alpha}}\left(1+\left|\log(R^{\frac{2}{\alpha}})\right|\right)\int |\omega_\theta^\aa(r',z')|\sqrt{\frac{r'}{r}}dr' dz'\lesssim R^{-\frac{2}{\alpha}}\left(1+\left|\log(R^{\frac{2}{\alpha}})\right|\right)\frac{\cS(t)^{2-\frac{2}{\alpha}}}{\sqrt{\alpha}}\|\Omega\|_{L^2},
\end{equation*}
provided  $\lambda\leq R \leq 2\lambda$. The estimate then follows by computing the localized $L^2$-norm.
\end{proof}

\subsubsection{Localized elliptic estimates}
Here and below we use the notation $D_R:=R\de_R$. Given $\chi(R) \in C_c^\infty([0,\infty)) $ such that 
\begin{equation*}
    \mathbf{1}_{[0,1]}(R)\leq \chi(R)\leq \mathbf{1}_{[0,2)}(R), \qquad D_R\chi(R)\lesssim \chi(R)
\end{equation*}
for all $R\geq 0$, let us introduce the cut-off function $\chi_\lambda (R):=\chi(R/\lambda)$, where $\lambda>0$ is a parameter to be chosen. 
We first develop some localized elliptic estimates that will be used to estimate the support size of our unknowns.

\begin{Prop}\label{lem:localized-ell}
Suppose that $t\in [0,T']$ with $T'$ in \eqref{eq:T'}, then there exists $(2\cS(t))^{\alpha}<\lambda<1/2$ such that the analogous decomposition of Theorem \ref{thm:tarek3D} holds, i.e.
\begin{align}
    \Psi_\lambda= \Psa^\lambda+\hE^\lambda, \quad \Psa^\lambda=\Psa(\Omega \chi_\lambda+Z_2)=\Ps(\Omega)+\widetilde{\mathcal{R}}^\alpha,
\end{align}
where $\Psa(\Omega), \Ps(\Omega)$ are given in \eqref{eq:psi23D} and \begin{align}
Z_2:=2\alpha^2 R^2 \frac{\chi'_\lambda}{\lambda}\de_R \Psi+
 \alpha^2R^2\frac{\chi_\lambda''}{\lambda^2}+\alpha(\alpha+5)R \frac{\chi_\lambda'}{\lambda}\Psi.
\end{align}
For $k \ge 3$, the following estimates hold:
\begin{align}\label{eq:first-loc}
    \alpha^2\|R \de_R \Psi_\lambda\|_{L^2}^2+\alpha\|\Psi_\lambda\|_{L^2}^2 + \alpha\|\de_\beta \Psi_\lambda\|_{L^2}^2 &\lesssim \alpha^{-1}\|\Omega\|_{L^2}^2, \\
    \alpha^2\|R \de_{RR} \Psi_\lambda\|_{\cH^k(\rho_i)} + \alpha \|R \de_{R\beta}\Psi_\lambda\|_{\cH^k(\rho_i)} + \|\de_{\beta \beta}(\Psi_\lambda-\Ps (\Omega))\|_{\cH^k(\rho_i)} & \lesssim \|\Omega\|_{\cH^k(\rho_i)}.
\end{align}
\end{Prop}
\begin{proof}
Note that with the bound of $\cS(t)$ provided by \eqref{eq:T'} and for $(2\cS(t))^{\alpha}<\lambda<1/2$ one has that 
$\lambda^{-\frac{1}{\alpha}}\cS(t)<  \frac{\cS(t)}{2\cS(t)}=\frac{1}{2}$. However, choosing for instance $\lambda=\frac{4}{9}$ yields 
$\lambda^{-\frac 1 \alpha} \cS(t) \lesssim \alpha.$
The localized stream function
$\Psi_\lambda$ satisfies
\begin{align*}
\alpha^2 R^2 \de_{RR} \Psi_\lambda + \alpha (\alpha + 5) R \de_R \Psi_\lambda + \de_{\beta \beta} \Psi_\lambda - \de_\beta ((\tan \beta) \Psi_\lambda)+6 \Psi_\lambda&=-\Omega \chi_\lambda+Z_2.
\end{align*}
Estimating the left-hand side follows the same lines of the proof of \cite[Proposition 7.8]{tarek1}, multiplying by $\Psi_\lambda$ yields
\begin{align*}
    \alpha^2 \|R \de_R \Psi_\lambda\|_{L^2}^2  + \frac{5\alpha -\alpha^2}{2}\|\Psi_\lambda\|_{L^2}^2+\|\de_\beta \Psi_\lambda\|_{L^2}^2 + \frac{1}{2} \|(\sec \beta )\Psi_\lambda\|_{L^2}^2 - 6 \|\Psi_\lambda\|_{L^2}^2
     \lesssim | \langle Z_2, \Psi_\lambda \rangle|+ \|\Omega \chi_\lambda\|_{L^2} \|\Psi_\lambda\|_{L^2}.
\end{align*}
Following the proof \cite[Lemma 9.6]{chen1}, we integrate by parts the first term of the right-hand side and by the property of the cut-off $(R \de_R \chi_\lambda)^2\lesssim \chi_\lambda, \, |D_R^k \chi_\lambda| \lesssim \mathbf{1}_{\lambda \le R \le 2\lambda}$, noticing that $|\log (\lambda^\frac 1 \alpha)| \lesssim \alpha^{-1}$ gives
\begin{align*}
    \langle Z_2, \Psi_\lambda \rangle \lesssim \alpha  \|\Psi \mathbf{1}_{\lambda \le R \le 2 \lambda}\|_{L^2}^2 \lesssim \alpha^{-1}  \|\Omega \|_{L^2}^2,
\end{align*}
where we used Lemma \ref{lem:outside}.
Using the Fourier Sine expansion $\{\sin(2n \beta)\}_{n \ge 1}$, note that for $n=1$, namely $\Psi_{2, \lambda}=\hat{\Psi}_2 \sin (2\beta),$ the term $\de_{\beta \beta}\Psi_{2, \lambda}-\de_\beta ((\tan \beta) \Psi_{2, \lambda}) + 6 \Psi_{\lambda, 2}=0$. In fact, the desired estimate for $\Psi_{2, \lambda}$ is obtained using its explicit formula.  
For the modes $n \ge 2$ it holds that
\begin{align*}
    \|\de_\beta \Psi_\lambda\|_{L^2}^2 \ge 8 \|\Psi_\lambda\|_{L^2}^2. 
\end{align*}
One can then rearrange some of the terms in the right-hand side of the above inequality, yielding
\begin{align*}
    -(6-\alpha) \|\Psi_\lambda\|_{L^2}^2 + \|\de_\beta \Psi_\lambda\|_{L^2}^2 \ge  \left(-\frac{(6-\alpha)}{8}+1\right) \|\de_\beta \Psi_\lambda\|_{L^2}^2\ge \frac 1 4 \|\de_\beta \Psi_\lambda\|_{L^2}^2.
\end{align*}
Finally, using that $\|\de_\beta \Psi_\lambda\|_{L^2} \ge 2 \|\Psi_\lambda\|_{L^2}$ and the Young  inequality yields 
\begin{align*}
\|\Omega\chi_\lambda\|_{L^2}\|\Psi_\lambda\|_{L^2} \le \frac 12 \|\Omega\|_{L^2} \|\de_\beta \Psi_\lambda\|_{L^2} \le \frac{1}{4\alpha}\|\Omega\|^2_{L^2}+\frac{\alpha}{4}\|\de_\beta \Psi_\lambda\|_{L^2}^2,
\end{align*}
so that the latter is absorbed by the left-hand side.
This concludes the proof of the first estimate. \\
Let us deal with the second estimate, whose proof, which is an adaptation of \cite[Proposition 9.9]{chen1}, follows by a weighted $L^2(\rho_i)$ estimate where the weights $\rho_i$ are defined in \eqref{def:weights3D}. More precisely, the estimate
\begin{align*}
   \alpha^2\|R^2 \de_{RR}\Psi_\lambda\|_{L^2(\rho_i)} + \alpha \|R \de_{R\beta}\Psi_\lambda\|_{L^2(\rho_i)} + \|\de_{\beta \beta}(\Psi_\lambda-\Ps (\Omega+Z_2))\|_{L^2(\rho_i)} & \lesssim \|\Omega \chi_\lambda+Z_2\|_{L^2(\rho_i)}
\end{align*}
is given by the proof of \cite[Proposition 9.9]{chen1}. We only have to show that
\begin{align*}
    \|\Ps(Z_2)\|_{L^2(\rho_i)} \lesssim \|\Omega\|_{L^2(\rho_i)}, \quad \|Z_2\|_{L^2(\rho_i)} \lesssim \|\Omega\|_{L^2(\rho_i)}.
\end{align*}
Consider the last addend of $Z_2$, for which one has
\begin{align*}
    \frac \alpha \lambda (\alpha +5)\| R \chi_\lambda' \Psi \rho_1^{1/2}\|_{L^2} &\lesssim\frac \alpha \lambda (\alpha +5)\|R \chi_\lambda' (\sin 2\beta)^{-\frac{99}{100}}\Psi\|_{L^2} \lesssim \frac \alpha \lambda (\alpha +5)\|R \chi_\lambda' \de_\beta \Psi\|_{L^2}\\& = \alpha (\alpha+5) \|D_R \chi_\lambda \de_\beta \Psi\|_{L^2} \lesssim  \alpha (\alpha+5)\alpha^{-1} \|\Omega\|_{L^2} \lesssim \|\Omega\|_{L^2(\rho_1)},
\end{align*}
where we used that $|D_R \chi_\lambda| \lesssim \mathbf{1}_{\lambda/2 \le R \le \lambda}$ and the estimate of Lemma \ref{lem:outside}. Estimating the other addends of $Z_2$ is analogous.  \\
Now, by Lemma \ref{lem:outside} and owing to  $\lambda^{-\frac 1 \alpha} \cS(t) \lesssim \alpha$ (which is stronger than $\lambda> (2\cS(t))^\alpha$ in Lemma \ref{lem:outside}), we have 
\begin{align}\label{eq:est-fuori}
    \|\Psi\mathbf{1}_{\lambda \le R \le 2\lambda}\|_{L^2} \lesssim \|\Omega\|_{L^2}.
\end{align}
From that and from the presence of the factor $\alpha$ in front of all the terms of $Z_2$, the estimate $\|\Ps(Z_2)\|_{L^2(\rho_i)} \lesssim \|\Omega\|_{L^2(\rho_i)}$ follows as well. To obtain weighted and localized estimates for the derivatives, the rest of the proof of \cite[Proposition 9.9]{chen2022} is based on a finite induction where at the step $j$, one considers $j$-th derivatives and $\lambda=\lambda_j$ decreasing in $j \in \{0, \cdots, J\}$ with $J \in \N$. Here we choose $\lambda_j$ so that $\lambda_0=\frac{4}{9}+2^{-5}$ and $\lambda_J=\frac{4}{9}+2^{-5-J}$.
The rest is identical to \cite[Proposition 9.9]{chen2022} and therefore it is omitted.
\end{proof}

\subsection{Remainder estimates 
 }\label{sec:remainder3D}
The only substantial difference between \eqref{eq:leading1} for the 2D Boussinesq equations and \eqref{eq:LOM3D1} for the 3D axisymmetric Euler equations is the equation for $\Ome$. More precisely, for the remainders due to the transport term, namely $\mathcal{T}f(R, \beta)$ in \eqref{eq:transport terms}, one can prove the analogous of Lemma \ref{lem:remT} with $\cH^N$ in \eqref{def:weights3D} and taking into account the localization of Proposition \ref{lem:localized-ell}. In particular, for $\Psi_\lambda = \Psi \chi_\lambda$ as in Proposition \ref{lem:localized-ell}, one notices that
\begin{align*}
    \Psi=\Psi_\lambda+\Psi \mathbf{1}_{R \ge 2\lambda}. 
\end{align*}
The first addend is estimated by Proposition \ref{lem:localized-ell}, while the latter is a lower-order term, with similar estimates to (actually better than) \eqref{eq:est-fuori}.
Therefore, to prove the analogous of Proposition \ref{prop:rem}, we only need to control the error due to the second addend in the right-hand side of the equation for $\Omega_\text{app}$ in \eqref{eq:LOM3D1}, namely
\be\ba\label{eq:resto3D}
\frac{\xi}{r^3}=\frac{\xi \chi_\lambda + \xi  (1-\chi_\lambda)}{(1+R^\frac 1\alpha \cos \beta)^3 }.
\ea\ee
\begin{Lem}\label{lem:1rem3D} There exists $\alpha_0 \ll 1$ small enough such that, if $\alpha \le \alpha_0$, for all $t\in [0,T']$ with $T'$ defined in \eqref{eq:T'}, taking $N \ge 4$ it holds that
\be
\left\langle \frac{\xi}{r^3}, \Omega_r \right \rangle_{\cH^N}
\le \left(C_N\exp\left({\frac{C_N}{\alpha}t}\right)+\|\xi_r\|_{\cH^N}\right) \|\Omega_r\|_{\cH^N}.
\ee
More precisely, there exists $\lambda>0$ such that 
\be
\left\langle \frac{\xi \chi_\lambda}{r^3}, \Omega_r \right \rangle_{\cH^N}
\le \left(C_N\exp\left({\frac{C_N}{\alpha}t}\right)+\|\xi_r \chi_\lambda\|_{\cH^N}\right) \|\Omega_r\|_{\cH^N}.
\ee
\end{Lem}
\begin{proof}
The considered remainder term is divided into two pieces as in \eqref{eq:resto3D}. We note that for $t\in [0,T']$ it holds that $ \supp(\xi)\subset \{(R,\beta) : \frac{1}{10}\leq R \leq \frac16\}$ by definition. Hence, for $\lambda=\frac49$ yields that 
\begin{equation*}
    (\xi(1-\chi_\lambda))(R)=0
\end{equation*}
for all $R\geq 0$. On the other hand, there exists $\alpha_0>0$ small enough such that, for all $0<\alpha\le \alpha_0$,
    \begin{align*}
        \frac{\chi_\lambda(R)}{(1+R^\frac{1}{\alpha}\cos\beta)^3} \le 2.
    \end{align*}
The desired $L^2(\rho_i)$ estimate is then a consequence of Cauchy-Schwarz inequality and the bound of $\|\csi\|_{\cH^k}$ in \eqref{eq:est-app}. To bound the $\cH^N$-norm, it is sufficient to observe that the most singular term is the one involving derivatives of $(1+R^\frac{1}{\alpha}\cos\beta)^{-3}$ which can be bounded by exploiting
\begin{equation*}
    \alpha^{-N} R^{\frac{1}{\alpha}}\lesssim 1,
\end{equation*}
for $\frac{1}{10}<R<\frac{1}{6}$ and $\alpha$ small enough.

\end{proof}
Note that the choice of the regularity index $N\ge 4$ is due to the fact that in $\cH^N, \, N\ge 4$, the remainder estimates below are simpler as the embedding theorems are better, as observed in \cite[Section 8]{tarek1}.
Arguing as for the proof of Proposition \ref{prop:rem} and relying on the Sobolev-embeddings for $\cH^k$ proven in \cite[Corollary 8.3]{tarek1}, we conclude that the remainders are small in $\cH^N$-norm up to time $T'$. 
\begin{Cor}\label{cor:rem3D}
 Let $N\geq 4$, then 
    \begin{equation*}
        F(t)=\|\Omega_{r}(t)\|_{\cH^N}+\|\eta_{r}(t)\|_{\cH^N}+\|\xi_{r}(t)\|_{\cH^N}.
    \end{equation*}
    Then 
    \begin{equation*}
        F(t)\le C \sqrt{\alpha}
    \end{equation*}
    for all $0\leq t\leq T'$ with $T'$ defined in \eqref{eq:T'} and where $C_{N+1}$ only depends on the $\cH^{N+1}$-norm of the initial data $(\Omega_0,\eta_0,\xi_0)$. In particular,
    \begin{equation}
        \|\Omega_r(t)\|_{L^{\infty}}+ \|\eta_r(t)\|_{L^{\infty}}+ \|\xi_r(t)\|_{L^{\infty}}\le C \sqrt{\alpha}
    \end{equation}
    for all $0\leq t\leq T'$, where $C>0$ is independent of $\alpha$.
\end{Cor}

\subsection{Support size}\label{sec:support}
The objective of this section is to show that the support of $\omega_\theta^\aa, (ru_\theta^\aa)^2$ remains sufficiently localized and especially away from the symmetric axis up to the time for which norm-inflation of \eqref{eq:LOM3D1} occurs. More precisely, we show in particular that $T'\geq T^*(\alpha)$ with $T'$ and $T^*(\alpha)$ defined in \eqref{def:talpha-star} and \eqref{eq:T'} respectively.
We exploit that the transport type equation with anelastic constraint
\begin{equation*}
    \de_t f+ u\cdot\nabla_{r,z}f=0, \qquad \Div_{r,z}(r u)=0, 
\end{equation*}
satisfied by $(ru_\theta^\aa)^2$, see \eqref{eq:utheta2},
can equivalently be described through the unique associate flow
\begin{equation}\label{eq:flow}
    \dot{X}(t,x_0)=u(t,X(t,x_0)), \qquad X(0)=x_0.
\end{equation}
This is well-known in the cases of (nearly) incompressible vector fields $u$, see e.g. \cite{bertozzi}. Note that $\Div(u)=-r^{-1}\nabla_{r,z}r\cdot u$ is in general not bounded. Existence and uniqueness of the associated flow are proven in \cite[Proposition 2.2]{hientzsch} for the lake equations which represents a generalization of the axi-symmetric Euler eq. (without swirl). 

\begin{Lem}\label{lem.local-supp}
    Suppose, as in \eqref{eq:S0}, that $$(1/8+\eps)^\frac 1 \alpha \le \cS(0) \le (1/7+\eps )^\frac 1 \alpha.$$ Then, for all $t \in [0, T^*(\alpha)]$ it holds that 
    $$ (1/8)^\frac 1 \alpha \le \cS(t) \le (1/7+2\eps )^\frac 1 \alpha.$$
    In particular, it holds that $T'=T^*(\alpha)$ for $T'$ defined in \eqref{eq:T'} and there exists $\lambda>0$ such that
    \begin{align}
        \lambda^{-\frac 1 \alpha} \cS(t) \lesssim \alpha \qquad \text{for all} \,\, t \in [0, T^*(\alpha)],
    \end{align}
    and for $\alpha \le \alpha_0 \ll 1 $ small enough.
\end{Lem}
As an immediate consequence, we note that the assumptions of Proposition \ref{lem:localized-ell} are satisfied until time $T^*(\alpha)$, so that the remainder terms can be controlled within that time interval.

\begin{proof}
We represent the flow defined in \eqref{eq:flow} in terms of the polar coordinates $X(t)=(R(t),\beta(t))^t$. Consider the equation for $(r u_\theta^\aa)^2$ in \eqref{eq:utheta2}. From \eqref{eq:der-3D}-\eqref{eq:uru3}, writing more carefully the transport term \eqref{eq:transport3D}, one obtains 
\begin{align}\label{eq:transport-polar}
    u_r \de_r + u_z \de_z&= (-\alpha R \de_\beta \Psi)\de_R +(2\Psi+\alpha R \de_R \Psi)\de_\beta.
\end{align}
We want to rely on Lemma \ref{lem:localized-ell} by means of a bootstrap argument. To this end, we work in the time interval $[0,T']$, where $T'$ is defined in \eqref{eq:T'}. Within such times, 
one has in particular that $\lambda^{-\frac 1 \alpha} \cS(t)\lesssim \alpha $ for some $\lambda > 0$ and $\alpha \le \alpha_0 \ll 1 $ small enough. From  \eqref{eq:transport-polar} (one can compare with \cite[(9.66) p. 64]{chen2022}), it follows that
\begin{align*}
    \dot R(t) &= -\alpha R \de_\beta \Psi (\Omega)(R(t), \beta(t)). 
\end{align*}
Then, we can apply Theorem \ref{thm:tarek3D}, yielding
\begin{align}
    \dot R(t) &= -\alpha R \de_\beta \Psi =-R\left(\frac{\cL_{12}(\Omega)}{2}+2\alpha \mathcal{R}^\alpha (\Omega)\right) \cos (2\beta)-\alpha R\de_\beta \hE(\Omega)\notag\\
    &=-R\left(\frac{\cL_{12}(\Ome)}{2}+2\alpha \mathcal{R}^\alpha (\Ome)\right) \cos (2\beta) -R\left(\frac{\cL_{12}(\Omega_r)}{2}+2\alpha \mathcal{R}^\alpha (\Omega_r)\right) \cos (2\beta)\notag\\
    &\quad - \alpha R  \de_\beta \hE(\Ome)- \alpha R  \de_\beta \hE(\Omega_r)\notag.
\end{align}
This gives 
\begin{align}
    \log \left(\frac{R(t)}{R(0)}\right) &=- \int_0^t  \frac{\cL_{12}(\Ome(\tau))}{2} \cos(2\beta) \, d\tau \notag\\&\quad - \int_0^t  \frac{\cL_{12}(\Omega_r(\tau))}{2} \cos(2\beta) \, d\tau - \alpha \int_0^t (2\mathcal{R}(\Ome+\Omega_r)\cos (2\beta) +\de_\beta \hE(\Ome+\Omega_r))(\tau) \, d\tau.\notag
\end{align}
Appealing now to the explicit formula of $\Ome$ in \eqref{eq: g explicit} and estimate \eqref{eq:g-upperandlower} with $g(t)=\Ome(t)$, for $ t \in [0, T^*(\alpha)]$, yields
\begin{align}
     \log \left(\frac{R(t)}{R(0)}\right)
    &\lesssim \alpha\log \left(1+\frac{c}{2\alpha}t\right) + \sqrt \alpha + \left|\int_0^t (\cL_{12}+\alpha \mathcal{R}^\alpha+\de_\beta \hE)(\Omega_r(\tau))\, d\tau\right|, \label{eq:est-R}
\end{align}
where we used Theorem \ref{thm:tarek3D}, Proposition \ref{lem:localized-ell} and \eqref{eq:est-app}, from which, in particular, it follows that
\begin{align*}
    \alpha \left| \int_0^t (\mathcal{R}^\alpha +\de_\beta \hE)(\Ome(\tau)) \, d\tau\right|\le 2 \alpha t \|\Ome\|_{\cH^N}  \le  \sqrt \alpha, \quad t \in [0, T^*(\alpha)].
\end{align*}
To handle the last term of the right-hand side of \eqref{eq:est-R}, we use Corollary \ref{cor:rem3D}, in particular
\begin{align}
    \|\Omega_r(t)\|_{L^\infty} \le 4\sqrt \alpha.
\end{align}
Applying the same reasoning as before this yields, for $\alpha $ small enough, that 
\begin{align*}
    R(0) + \sqrt \alpha \le \sup_{t \in [0, T^*(\alpha)]} R(t)\le R(0)+2\sqrt \alpha \quad \Rightarrow \quad \left(\frac{1}{8}\right)^\frac 1 \alpha  \le  \cS(t)  \le \left(\frac 17 + 2\eps\right)^\frac 1 \alpha,
\end{align*}
since $\eps \sim |\log|\log|\alpha||^{-\frac 14} \gg \sqrt \alpha$.
Now we turn to $\left(\frac{\omega_\theta^\aa}{r}\right)$ in \eqref{eq:3Daxisymm}, satisfying the equation below:
\begin{align*}
    (\de_t+u_r^\aa\de_r + u_z^\aa \de_z)\left(\frac{\omega_\theta^\aa}{r}\right)=-\frac{1}{r^4}\de_z(ru_\theta^\aa)^2. 
\end{align*}
The transport part is the same as before, the difference is the presence of the right-hand side.
We can rewrite the equation for $\omega_\theta^\aa$ as follows:
\begin{align}
    \frac{d}{dt}\omega_\theta^\aa(t, X(t, x_0)) = - \frac{1}{r^3} \de_z (r u_\theta^\aa (t, X(t, x_0)) )^2.
\end{align}
We already know from the above discussion that  $\left(\frac 18\right)^\frac 1 \alpha \lesssim \supp_{\rho} (ru_\theta^\aa)^2 \lesssim  \left(\frac 17\right)^\frac 1 \alpha$. Reasoning as before and solving the Liouville equation ensures that $\left(\frac 18\right)^\frac 1 \alpha \lesssim \cS(t) \lesssim  \left(\frac 17\right)^\frac 1 \alpha$.
\end{proof}

\subsection{Proof of Theorem \ref{thm:main3sec3}}
With all these ingredients, the proof of Theorem \ref{thm:main3sec3} is easily deduced. By definition of $\eta(\cdot, R, \beta)=\de_r (ru_\theta^\aa)^2$, it follows that
\begin{align*}
    \frac{\eta}{2r^2}=\frac{\de_r(r u_\theta^\aa)^2}{r^2}=\frac{ (u_\theta^\aa)^2}{r}+\frac{1}{2}\de_r (u_\theta^\aa)^2.
\end{align*}
Consider now the approximate solution $(\Ome (t), \et (t), \csi (t))$. Using the remainder estimates in Section \ref{sec:remainder3D} and the localization of the support in Section \ref{sec:support}, in Case (i) we have that 
\begin{align*}
    \sup_{t \in [0, T^*(\alpha)]} \left\| \frac{(u_\theta^\aa)^2}{r}+\frac12\de_r (u_\theta^\aa)^2\right\|_{L^\infty}& \gtrsim \left\| \et \right\|_{L^\infty}-\sqrt \alpha \gtrsim \|\eta_{0, \mathrm{app}}^{\alpha, \delta}\|_{L^\infty} \left(1+\frac{\log |\log \alpha|}{C_{N+1}}\right)^{\frac{1}{c_2}},
\end{align*}
from which \eqref{eq:crucial-inequality} follows. Similarly, in Case (ii) one obtains 
\begin{align*}
    \sup_{t \in [0, T^*(\alpha)]} \left\| \de_z (u_\theta^\aa)^2\right\|_{L^\infty}&  \gtrsim \|\xi_{0, \mathrm{app}}^{\alpha, \delta}\|_{L^\infty} \left(1+\frac{\log |\log \alpha|}{C_{N+1}}\right)^{\frac{1}{c_2}}.
\end{align*}
The proof is concluded.
\section*{Acknowledgment}
RB thanks Tarek M. Elgindi for several helpful discussions and suggestions, especially regarding the application to the 3D axisymmetric Euler equations. 
The authors thank the anonymous referees for their helpful comments and suggestions.
RB is partially supported by the GNAMPA group of INdAM, the Italian Ministry of University and Research, PRIN 2020 entitled ``PDEs, fluid dynamics and transport equation'' and PRIN 2022HSSYPN TESEO, NextGenEU. \\
LEH is funded by the Deutsche Forschungsgemeinschaft (DFG, German Research Foundation) – Project-ID 317210226 – SFB 1283.\\
FI is partially supported by  PRIN project 2017JPCAPN entitled ``Qualitative and quantitative aspects of nonlinear PDEs."

\bibliographystyle{siam}
\bibliography{biblio}

\end{document}